\documentclass[11pt,a4paper]{scrartcl} %
\usepackage[utf8]{inputenc}%
\usepackage[english]{babel}%
\usepackage[T1]{fontenc}%
\usepackage{lmodern} %
\usepackage{graphicx}
\usepackage{tabularx}
\usepackage{caption}
\usepackage{amsthm,amssymb,amsmath}
\usepackage{enumerate}
\usepackage{tikz,tikz-cd}
\usetikzlibrary{tikzmark,quotes, arrows, arrows.meta}
\usepackage{abstract}
\usepackage{hyperref}
\hypersetup{
    colorlinks=true,
    linktocpage,
    linkcolor=blue,
    citecolor=blue,
    urlcolor=blue,
    pdftitle={A Unified View on the Functorial Nerve Theorem and its Variations},
}
\usepackage[capitalise]{cleveref}
\crefname{thm}{Theorem}{Theorems}

\usepackage{multicol}
\usepackage[backend=biber,style=alphabetic,sorting=nty, maxbibnames=99]{biblatex}
\usepackage[mathscr]{eucal}

\usepackage{multirow}
\usepackage{hhline}

\setkomafont{pagehead}{\rmfamily}
\setkomafont{sectioning}{\rmfamily\bfseries} %

\numberwithin{equation}{section}

\theoremstyle{plain}

\newtheorem{thm}{Theorem}[section]
\newtheorem{thm-no-section}{Theorem}
\newtheorem*{thm*}{Theorem}
\newtheorem{cor}[thm]{Corollary}
\newtheorem{lemma}[thm]{Lemma}
\newtheorem{prop}[thm]{Proposition}
\newtheorem{thmx}{Theorem}

\theoremstyle{definition}
\newtheorem{defi}[thm]{Definition}

\theoremstyle{remark}
\newtheorem{ex}[thm]{Example}
\newtheorem{rmk}[thm]{Remark}

\newcommand{\op}{\operatorname}
\newcommand{\ignore}[1]{}
\newcommand{\Spc}{\operatorname{Spc}}
\newcommand{\Nrv}{\operatorname{Nrv}}
\newcommand{\Pos}{\operatorname{Pos}}
\newcommand{\Flag}{\operatorname{Flag}}
\newcommand{\opp}{\operatorname{op}}
\newcommand{\Set}{\mathsf{Set}}
\newcommand{\Top}{\mathsf{Top}}

\newcommand{\ClConv}{\mathsf{ClConv}}
\newcommand{\CGSpc}{\mathsf{CGSpc}}
\newcommand{\sSet}{\mathsf{sSet}}
\newcommand{\Fun}{\mathsf{Fun}}
\newcommand{\Hom}{\mathrm{hom}}
\newcommand{\Cov}{\mathsf{Cov}}

\newcommand{\sRMod}{\mathsf{s}(R\text{-}\mathsf{Mod})}
\newcommand{\sC}{\mathsf{s}(\mathcal{C})}
\newcommand{\RMod}{R\text{-}\mathsf{Mod}}
\newcommand{\freeR}{\mathsf{R}}
\newcommand{\iso}{\cong}
\newcommand{\SimpComplex}{\mathsf{Simp}}
\newcommand{\Po}{\mathsf{Po}}
\newcommand{\bbR}{\mathbb{R}}
\newcommand{\Sing}{\operatorname{Sing}}
\newcommand{\Sd}{\operatorname{Sd}}
\newcommand{\np}[1]{P_{#1}}
\newcommand{\BarCon}{\operatorname{Bar}}
\newcommand{\SmallBarCon}{\operatorname{Blowup}}
\newcommand{\Tri}{\operatorname{T}}
\newcommand{\PoBar}{\operatorname{PoBar}}
\newcommand{\AbBarCon}{\operatorname{Bar}}
\newcommand{\sk}[2]{\operatorname{sk}_{#1}#2}
\newcommand{\vertx}[1]{\operatorname{Vert}#1}
\newcommand{\CB}[2]{D_{#1}(#2)}
\newcommand{\B}[2]{B_{#1}(#2)}

\makeatletter
\let\@fnsymbol\@alph
\makeatother

\title{A Unified View on the Functorial Nerve Theorem and its Variations}
\author{%
Ulrich Bauer\thanks{ Technical University of Munich (TUM), Boltzmannstraße 3, Garching bei München, Germany}%
\and
Michael Kerber\thanks{ Graz University of Technology, Kopernikusgasse 24, Graz, Austria}%
\and
Fabian Roll\footnotemark[1]\ \thanks{Corresponding author. \textit{Email address:} \texttt{fabian.roll@tum.de} (Fabian Roll)}%
\and
Alexander Rolle\footnotemark[1]%
}
\date{}

\begin{document}
\maketitle
\begin{abstract}
The nerve theorem is a basic result of algebraic topology
that plays a central role in computational and applied aspects of the subject.
In topological data analysis, one often needs a nerve theorem that is functorial
in an appropriate sense,
and furthermore one often needs a nerve theorem for closed covers
as well as for open covers.
While the techniques for proving such functorial nerve theorems
have long been available,
there is unfortunately no general-purpose, explicit treatment of this topic in the literature.
We address this by proving a variety of functorial nerve theorems.
First, we show how one can use elementary techniques to prove nerve theorems
for covers by closed convex sets in Euclidean space,
and for covers of a simplicial complex by subcomplexes.
Then, we establish a more general, ``unified'' nerve theorem that subsumes many of the variants,
using standard techniques from abstract homotopy theory.
\end{abstract}


\newpage

\tableofcontents

\newpage

\section{Introduction}
\paragraph{Background}

If $\mathscr{U} = ( U_i )_{i \in I}$ is a cover of a topological space $X$,
then the \emph{nerve} of $\mathscr{U}$, which dates back to Alexandroff \cite{MR1512423}, is the simplicial complex $\Nrv (\mathscr{U})$
whose simplices are the finite subsets $J \subseteq I$ such that
the intersection $\cap_{i \in J} U_i$ is non-empty.
The nerve of a cover played an important role in the development of homology and cohomology theory.
In particular, \v{C}ech (co)homology is given by
the (co)limit of the (co)homology groups of the nerves of a directed system of open covers ordered by refinement.
A historical exposition
can be found in \cite[Chapter 2]{cech-history}.

The \emph{nerve theorem}, whose early versions are due to Leray \cite{MR15786}, Borsuk \cite{Borsuk1948}, and Weil \cite{weil},
is a basic result in algebraic and combinatorial topology.
Roughly speaking, it says that
if every non-empty finite intersection of cover elements is contractible,
then, subject to some further tameness conditions on $X$ and $\mathscr{U}$,
the space~$X$ is homotopy equivalent to the nerve of $\mathscr{U}$.

The literature on the nerve theorem is extensive but unfortunately hard to navigate.
In part, this is because there are many different variants of the nerve theorem.
There are many choices one can make for the ``further tameness conditions''
on the space and cover that yield a nerve theorem,
and there are further choices one can make for the kind of equivalence one works with.
For example, one could ask that the non-empty finite intersections are weakly homotopy equivalent to the one-point space, or to be acyclic.
Many of the original results use concepts that are now obscure,
and the many possible choices for hypotheses and proof techniques
make it difficult to compare all the available nerve theorems.

Nowadays, the nerve theorem and the aspect of functoriality play a crucial role in topological data analysis.
Nerves are the main way to replace a topological space, determined by the data points using geometric constructions, with a combinatorial model
that is suitable for computations.
Two prominent examples are the \emph{\v{C}ech complex} and the \emph{Delaunay complex},
which arise as nerves of a collection of closed balls and closed Voronoi balls, respectively.
Another important example is the \emph{Vietoris--Rips complex},
which is not usually defined as the nerve of a cover,
though it is isomorphic to a nerve \cite{ghrist,DBLP:journals/cgf/ChazalCGMO09}.
Note that, while one can choose whether to use open or closed sets
when defining the \v{C}ech and Vietoris--Rips complexes,
the only standard way to define the Delaunay complex uses closed sets.

\begin{center}
 \begin{minipage}[c]{0.45\textwidth}
\centering
 \includegraphics[width=.9\textwidth]{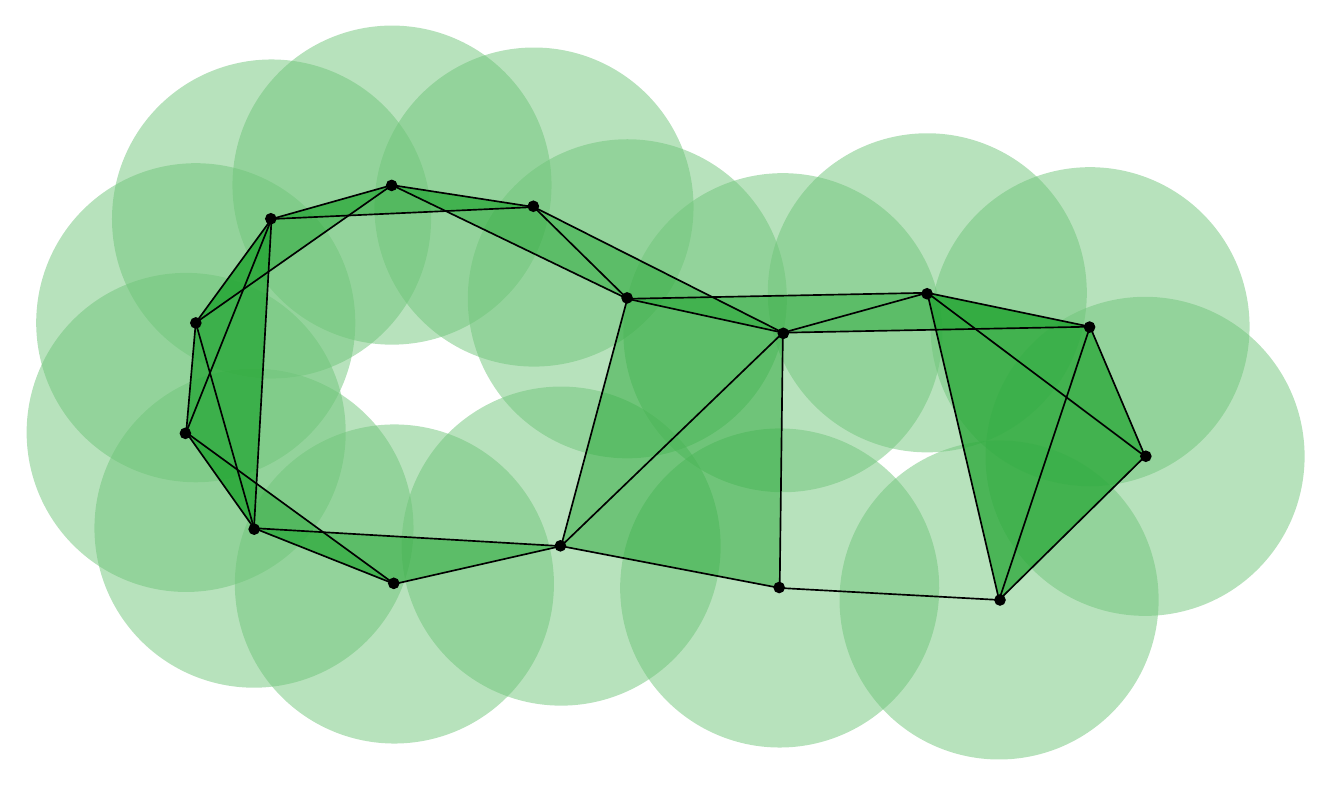}
\end{minipage}
\begin{minipage}[c]{0.45\textwidth}
\centering
  \includegraphics[width=.9\textwidth]{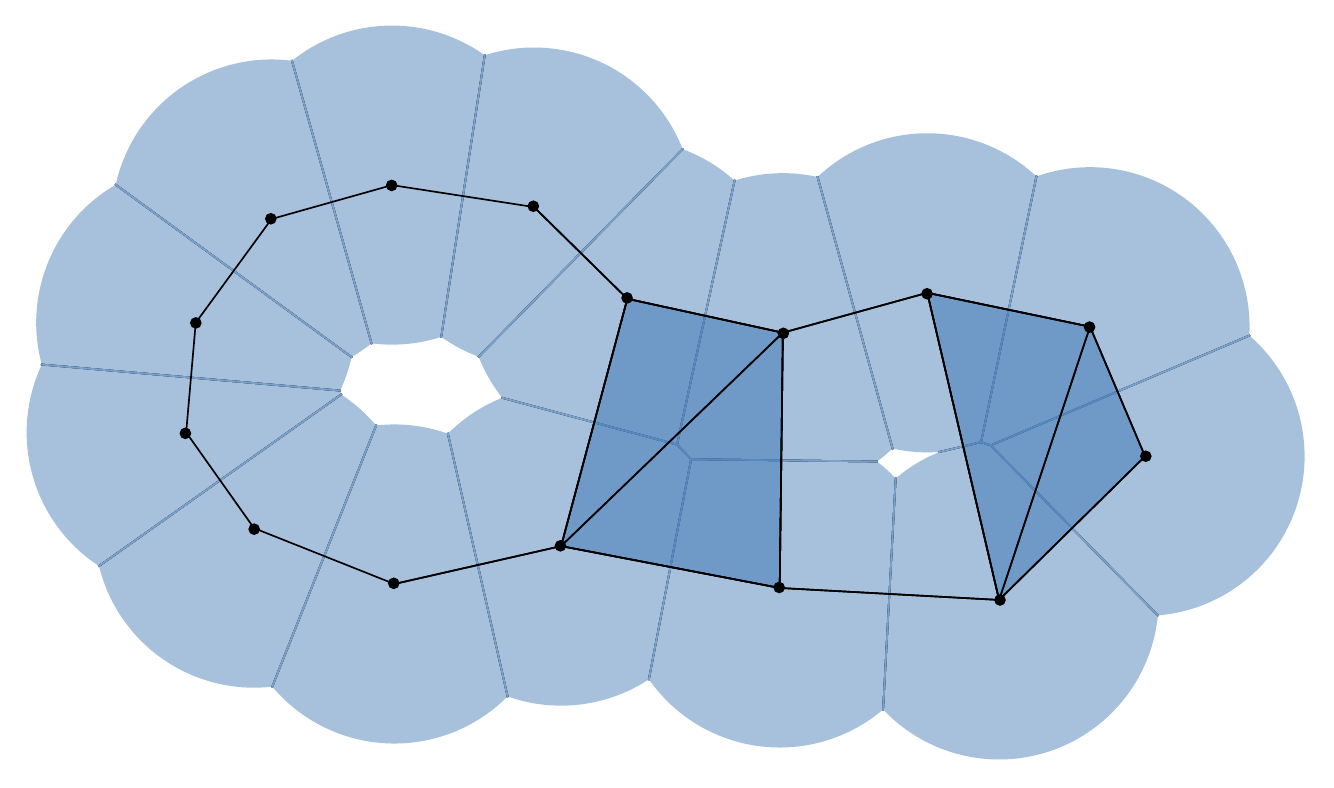}
\end{minipage}
\begingroup
\captionof{figure}{A cover by closed balls (left) and closed Voronoi balls (right) together with the corresponding \v{C}ech and Delaunay complex.}
\endgroup
\end{center}

These examples are typical, in that the topological spaces determined by data points
usually depend on one or more parameters, leading to filtrations of topological spaces and covers. 
Now functoriality ensures that the corresponding nerves form a filtration as well.
For example, if $X \subset \bbR^d$ is a finite set of points,
the {offset filtration} $O$ is the filtration of $\bbR^d$ with
$O_r = \cup_{x \in X} \CB{r}{x}$ for $r > 0$,
where $\CB{r}{x}$ is the closed ball about~$x$ of radius $r$.
The nerve of the cover $\mathscr{U}_r = ( \CB{r}{x} )_{x \in X}$ 
is the {\v{C}ech complex}, and as $r$ varies, 
these complexes form a filtration as well. 
In this case, the nerve theorem says that $O_r$
is homotopy equivalent to the nerve of $\mathscr{U}_r$. 

Going further, one wants a nerve theorem to provide homotopy equivalences that are somehow compatible with the inclusion maps in these two filtrations. 
This is necessary in particular if one is interested in 
persistent homology~\cite{eh-computational}, 
which is an algebraic invariant of filtrations that encodes the homology of each step of the filtration, as well as the maps in homology induced by each inclusion.
There are several ways in which the homotopy equivalences provided by a nerve theorem might be compatible with the inclusions in these two filtrations, 
as we will now explain. 

In order to prove that the persistent homology of the offset filtration is isomorphic to the persistent homology of the associated \v{C}ech complex filtration, 
it suffices to have isomorphisms
$H_n (O_r) \iso H_n (\Nrv(\mathscr{U}_r))$ such that all the squares of the following form commute:
\begin{equation}
\begin{tikzcd}
	H_n(O_r) \arrow[d] \arrow[r, "\iso"]
	& H_n( \Nrv (\mathscr{U}_r) ) \arrow[d] \\
	H_n(O_{r'}) \arrow[r, "\iso"']
	& H_n( \Nrv (\mathscr{U}_{r'}) )
\end{tikzcd}
\end{equation}
By \cref{intro-thm-one-map} below, such isomorphisms can be constructed from the induced maps of homotopy equivalences $\vert \Nrv(\mathscr{U}_r) \vert \to O_r$ between the nerves and the offsets such that all squares of the following form commute:
\begin{equation}
\label{offest-filtration-onearrow-functorial-nerve-intro}
\begin{tikzcd}
	O_r \arrow[d, hookrightarrow]
	& \vert \Nrv (\mathscr{U}_r) \vert \arrow[d, hookrightarrow] \arrow[l, "\simeq"'] \\
	O_{r'}
	& \vert \Nrv (\mathscr{U}_{r'}) \vert \arrow[l, "\simeq"]
\end{tikzcd}
\end{equation}
The construction of these compatible homotopy equivalences
relies on the fact that the cover elements of the offset filtration are convex and that the inclusions $O_r \hookrightarrow O_{r'}$ are affine linear.

For a more general filtration $(X_r,\mathscr{A}_r)$, with $X_r=\bigcup\mathscr{A}_r$, a similar strategy does not necessarily produce commuting diagrams as in \ref{offest-filtration-onearrow-functorial-nerve-intro}. However, if one is only interested in the filtration after applying homology or some other homotopy-invariant functor,
then it suffices to have homotopy equivalences $X_r \to \vert \Nrv(\mathscr{A}_r) \vert$
such that all squares of the following form commute \emph{up to homotopy:}
\begin{equation} \label{functorial-up-to-homotopy}
\begin{tikzcd}
X_{r} \arrow[r, "\simeq"] \arrow[d, hookrightarrow]
& \vert \Nrv (\mathscr{A}_r ) \vert \arrow[d, hookrightarrow] \\
X_{r'} \arrow[r, "\simeq"'] \arrow[ur, Rightarrow, "H"]
& \vert \Nrv (\mathscr{A}_{r'} ) \vert
\end{tikzcd}
\end{equation}
In the diagram, $H$ is a homotopy from the bottom route around the square to the top route.
Nerve theorems with this structure are often used in the study of persistent homology
(for references, see the end of this introduction).

However, in some homotopy-theoretic approaches to topological data analysis,
we need a nerve theorem that is compatible with the inclusions
$X_r \hookrightarrow X_{r'}$ on the nose, and not just up to homotopy.
In this paper, we will prove nerve theorems that provide strictly commuting diagrams,
at the cost of introducing an intermediary between the covered space and the nerve:
we obtain a filtration $Z_r$ and homotopy equivalences $Z_r \to X_r$ and $Z_r \to \vert\Nrv(\mathscr{A}_r)\vert$ such that
all the diagrams of the following form commute:
\begin{equation} \label{zig-zag}
\begin{tikzcd}
	X_r \arrow[d, hookrightarrow] & Z_r
	\arrow[r, "\simeq"] \arrow[d, hookrightarrow] \arrow[l, "\simeq"']
	& \vert \Nrv (\mathscr{A}_r) \vert \arrow[d, hookrightarrow] \\
	X_{r'} & Z_{r'}
	\arrow[r, "\simeq"'] \arrow[l, "\simeq"]
	& \vert \Nrv (\mathscr{A}_{r'}) \vert
\end{tikzcd}
\end{equation}

While one can avoid introducing intermediate objects in the special case of the offset filtration,
this is not possible in general, as we explain below.
Diagrams of the form \ref{zig-zag} appear classically in the study of homotopy categories \cite{MR0210125}. More recently, they appear in
Blumberg and Lesnick's work on the \emph{homotopy interleaving distance}
\cite{blumberg-lesnick},
a distance on diagrams of spaces that is universal among stable and homotopy-invariant distances.
The idea is to define an equivalence relation on filtered spaces such that
$F$ and $F'$ are related if they can be connected via an intermediate filtration, as above, with the horizontal arrows weak homotopy equivalences.
Then, filtered spaces $F_1$ and $F_2$ are homotopy interleaved if $F_1$ is related to some $F_1'$,
$F_2$ is related to some $F_2'$, and $F_1'$ and $F_2'$ are interleaved.
An important motivation for nerve theorems that provide diagrams of the form \ref{zig-zag}
is that they can be used in frameworks like the one of Blumberg--Lesnick.

In this paper, we prove a variety of nerve theorems 
that provide strictly commuting diagrams relating spaces and nerves. 
These nerve theorems are summarized in \cref{nerve-theorems-table}.
Before we introduce the contents of the paper in detail, 
we highlight some novelties in our treatment of the material. 
The blowup complex is often used as an intermediate object 
for proving nerve theorems (serving as the space $Z_r$ in Diagram \ref{zig-zag}). 
This space is closely related to the bar construction from abstract homotopy theory. 
We discuss this issue in detail, 
explaining why it is advantageous to state nerve theorems using the blowup complex 
rather than the bar construction, 
but that one can still use the bar construction for the proofs. 
We prove a nerve theorem for subsets of Euclidean space 
covered by closed, convex subsets by constructing a pair of maps between the nerve of the cover and the space that constitute a homotopy equivalence.
We extend this to a functorial nerve theorem using the blowup complex, and we also introduce a notion of pointed covers,
which allows us to prove a functorial nerve theorem 
for closed, convex covers that does not require an intermediate object at all.
We consider simplicial complexes covered by subcomplexes, 
and explain how one can use a bar construction in the category of posets 
to prove a functorial nerve theorem in this setting. 
Finally, after using standard model category arguments to prove 
more general nerve theorems, 
we give a series of examples that demonstrate 
that most of the assumptions in these theorems are necessary.

\begin{table*}[p]
\makebox[\textwidth][c]{
\begin{tabular}{|c|c|>{\centering}m{0.12\textwidth}|>{\centering}m{0.12\textwidth}|>{\centering}m{0.14\textwidth}|c|c|}
\hline
Reference & O/C & $X$ & $|\mathscr{A}|$ 
& $\mathscr{A}$ & Equivalence & Intermediate \\
\hhline{|=|=|=|=|=|=|=|}
Thm. \ref{thm:comp_conv_zigzag} & closed & $X \subset \mathbb{R}^d$ & finite 
& convex & hom. eq. & $\SmallBarCon(\mathscr{A})$ \\
\hline
Thm. \ref{one-arrow-functorial-nerve-thm} & closed & $X \subset \mathbb{R}^d$ & finite 
& pointed, convex & hom. eq. & none \\
\hhline{|=|=|=|=|=|=|=|}
Thm. \ref{functorial-simplicial-nerve} & closed & simplicial complex & none 
& good cover by subcomplexes & hom. eq. & $\SmallBarCon(\mathscr{A})$ \\
\hline
Thm. \ref{compact-semialgebraic} & closed & compact, semi-algebraic & finite 
& good, semi-algebraic & hom. eq. & $\SmallBarCon(\mathscr{A})$ \\
\hline
Thm. \ref{functorial-bjorner-nconnectivity} & closed & simplicial complex & loc. finite 
& $(k-t+1)$-good cover by subcomplexes & $k$-connected & $\SmallBarCon(\mathscr{A})$ \\
\hhline{|=|=|=|=|=|=|=|}
\multirow{6}{*}{Thm. \ref{abstract-nerve-thm}} 
& open & none & none & good, numerable & hom. eq. & $\SmallBarCon(\mathscr{A})$ \\\cline{2-7}
& open & none & none & weakly good & weak hom. eq. & $\SmallBarCon(\mathscr{A})$ \\\cline{2-7}
& open & none & none & CG, hom'gy good & hom'gy iso. & $\SmallBarCon(\mathscr{A})$ \\\cline{2-7}
& closed & CG & loc. finite, loc. finite-dim. & L-condition, good & hom. eq. & $\SmallBarCon(\mathscr{A})$ \\\cline{2-7}
& closed & CG & loc. finite, loc. finite-dim. & L-condition, weakly good & weak hom. eq. & $\SmallBarCon(\mathscr{A})$ \\\cline{2-7}
& closed & CG & loc. finite, loc. finite-dim. & L-condition, hom'gy good & hom'gy iso. & $\SmallBarCon(\mathscr{A})$ \\
\hline
\end{tabular}
}
\caption{A summary of the functorial nerve theorems in this paper,
for a cover $\mathscr{A}$ of a space $X$. Columns 2--5 summarize the assumptions: 
whether the cover elements are open or closed, assumptions on $X$, 
assumptions on the cardinality of $\mathscr{A}$, and additional assumptions on $\mathscr{A}$. 
Columns 6--7 summarize the conclusions: the type of equivalence established, 
and the intermediate object, if any. 
We use the abbreviations: 
homotopy equivalence (hom. eq.), 
homology isomorphism (hom'gy iso.), 
compactly generated (CG), 
homologically good (hom'gy good), 
Latching space condition (L-condition), 
locally finite (loc. finite), 
locally finite-dimensional (loc. finite-dim.).}
\label{nerve-theorems-table}
\end{table*}

\paragraph{Functorial Nerve Theorems}
In order to say precisely what we mean by a functorial nerve theorem,
we need to explain how the nerve can be viewed as a functor.
To this end, we will define the category of covered spaces.

To motivate our definition, we first briefly discuss a variant that is also common in the literature (see, e.g., \cite{barr}).
The objects in this category are of the form $(X,\mathscr{U})$, where~$X$ is a topological space and~$\mathscr{U}$ is an unindexed cover of $X$. A map between covered spaces $f\colon (X,\mathscr{U})\to (Y,\mathscr{V})$ is then given by a continuous map $f\colon X\to Y$ such that for any cover element $U\in \mathscr{U}$ there exists~$V\in\mathscr{V}$ with $f(U)\subseteq V$. Choosing such a cover element $V\in\mathscr{V}$ for every element $U\in \mathscr{U}$ determines a simplicial map $\Nrv(\mathscr{U})\to \Nrv(\mathscr{V})$ between the nerves. In general, different choices give different simplicial maps, but it will always be unique up to contiguity (see \cite[p. 67]{munkres} for a definition). In particular, it follows that any two choices determine, up to homotopy, the same map on the geometric~realization.

To avoid having to make choices, we work with indexed covers and record the choice of cover elements as above in a map between the indexing sets. This way, we circumvent the ambiguity of the induced map between the nerves up to homotopy.

\begin{defi}
Let $X$ and $Y$ be topological spaces,
$( U_i )_{i\in I}$ a cover of $X$, and $( V_{\ell} )_{\ell \in L}$ a cover of $Y$. A \emph{map of indexed covers} $\varphi\colon ( U_i )_{i\in I}\to ( V_{\ell} )_{\ell \in L}$ is specified formally by a map $\varphi \colon I \to L$ between the indexing sets, which we denote with the same symbol. We say that a continuous map $f \colon X \to Y$ is \emph{carried by} $\varphi$ if for all $i \in I$ we have $f(U_i) \subseteq V_{\varphi(i)}$.
\end{defi}

If $f$ is carried by $\varphi$ and $g$ is carried by $\psi$, then $g\circ f$ is carried by $\psi\circ \varphi$ if the compositions are defined. Hence, we get the following category.

\begin{defi} \label{category-of-covered-spaces}
The objects of the \emph{category of covered spaces} $\Cov$
are pairs of the form $(X, ( U_i )_{i \in I})$,
where $X$ is a topological space and $( U_i )_{i \in I}$ is a cover of $X$.
A \emph{morphism of covered spaces} $(f, \varphi) \colon (X, ( U_i )_{i \in I}) \to (Y, ( V_{\ell} )_{\ell \in L})$
consists of a continuous map $f \colon X \to Y$ and 
a map $\varphi\colon I \to L$ 
such that $f$ is carried by the corresponding map of indexed covers $\varphi\colon ( U_i )_{i\in I}\to ( V_{\ell} )_{\ell \in L}$.
\end{defi}

With this category in hand, we can define a functor $\Spc \colon \Cov \to \Top$ by forgetting the cover:
$\Spc$ takes a pair $(X, ( U_i )_{i \in I})$ to $X$.
By taking the geometric realization of the nerve of a cover, we obtain another such functor. We denote by $\SimpComplex$ the category of simplicial complexes.

\begin{defi}
Let $X$ be a topological space,
and let $\mathscr{U} = (U_i)_{i \in I}$ be a cover of $X$.
For any $J \subseteq I$, we write $U_J = \bigcap_{i \in J} \, U_i$.
The \emph{nerve} of $\mathscr{U}$ is the simplicial complex $\Nrv(\mathscr{U})$
with simplices
\[
	\{ J \subseteq I \mid |J| < \infty \text{ and } U_J \neq \emptyset \} \; .
\]
A morphism of covered spaces $(f, \varphi) \colon (X, \mathscr{U}) \to (Y, \mathscr{V})$ induces a simplicial map between the nerves of the covers $\varphi_*\colon \Nrv(\mathscr{U})\to \Nrv(\mathscr{V})$. Thus, the nerve can be seen to be a functor $\Nrv\colon\Cov\to\SimpComplex$. By postcomposing this with the geometric realization functor $|\cdot|\colon\SimpComplex\to\Top$, we get the functor $|\Nrv| \colon \Cov \to \Top$ that takes a pair $(X, \mathscr{U})$
to the geometric realization~$|\Nrv(\mathscr{U})|$.
\end{defi}
\begin{rmk}
There is a variant of this definition, where the vertex set is the set of cover elements, in contrast to our definition, where it is the indexing set. While these definitions yield different simplicial complexes in general, as the same subset can appear multiple times in the indexed cover, they are always homotopy equivalent. More precisely, if $\mathcal{U}=(U_i)_{i\in I}$ is an indexed cover and $U_j\subseteq U_l$, with $j\neq l$, are cover elements, then the inclusion $\Nrv(\mathcal{V})\hookrightarrow \Nrv(\mathcal{U})$ is a homotopy equivalence, where $\mathcal{V}=(U_i)_{i\in I\setminus \{j\}}$: The link $\op{lk}(j)=\{\sigma\in \Nrv(\mathcal{U})\mid j\notin\sigma,\ \sigma\cup\{j\}\in \Nrv(\mathcal{U}) \}$ of the vertex $j$ in $\Nrv(\mathcal{U})$ is a cone with apex $l$, i.e., for all $\sigma\in \op{lk}(j)$ we have $\sigma\cup\{l\}\in \Nrv(\mathcal{U})$. Therefore, there exists a collapse~$\Nrv(\mathcal{U})\searrow \Nrv(\mathcal{V})$ (see \cref{discrete_morse} for a definition).
\end{rmk}

Now that we can understand the covered space and the nerve as functors,
we can consider natural transformations that relate them.
In general, if $F_1$ and $F_2$ are functors from some category $\mathscr{C}$ to $\Top$,
and $\sigma \colon F_1 \Rightarrow F_2$ is a natural transformation,
one says that $\sigma$ is a pointwise homotopy equivalence if the component
$\sigma_C \colon F_1(C) \to F_2(C)$ is a homotopy equivalence for all objects $C$ of $\mathscr{C}$.
Similarly one can consider pointwise weak homotopy equivalences,
pointwise homology isomorphisms, et cetera.
This paper is about nerve theorems that relate the covered space and the nerve through
pointwise~equivalences.

Most of these nerve theorems make use of a standard construction that is called the \emph{blowup complex} by Zomorodian--Carlsson \cite{blowup}, but goes back at least to Segal \cite{segal}. It is a functor $\SmallBarCon \colon \Cov \to \Top$, along with natural transformations
$\rho_S \colon \SmallBarCon \Rightarrow \Spc$ and
$\rho_N \colon \SmallBarCon \Rightarrow |\Nrv|$. In particular, for any morphism of covered spaces $(f,\varphi)\colon (X, \mathscr{U}) \to (Y, \mathscr{V})$ there exists a commuting diagram of the following form:\vspace{-2.5pt}
\begin{equation}
\begin{tikzcd}
	X \arrow[d,swap,"f"] & \SmallBarCon(\mathscr{U})
	\arrow[r,"\rho_N"] \arrow[d] \arrow[l,swap,"\rho_S"]
	& \vert \Nrv (\mathscr{U}) \vert \arrow[d, "|\varphi_*|"] \\
	Y & \SmallBarCon(\mathscr{V})
	\arrow[r,swap,"\rho_N"] \arrow[l,"\rho_S"]
	& \vert \Nrv (\mathscr{V}) \vert
\end{tikzcd}\vspace{-2.5pt}
\end{equation}
We write $[n]$ for the set $\{0,\dots,n\}$. If $\mathscr{U} = (U_i)_{i\in[n]}$ is a finite cover of a space $X$,
then the blowup complex is\vspace{-2.5pt}
\[
	\SmallBarCon(\mathscr{U}) = \bigcup_{J \in \Nrv(\mathscr{U})} U_J \times \Delta^{J}
	\subseteq X\times \Delta^{n}  \; ,\vspace{-2.5pt}
\]
where $\Delta^{n}$ is the standard topological $n$-simplex
and $\Delta^{J}$ is a face of  $\Delta^{n}$ determined by the inclusion $J \subseteq [n]$.
The idea is that each piece of $X$ expands according to the number of cover elements that contain it.
See \cref{section:func_nerve} for the definition for arbitrary covers.

To begin, we give three functorial nerve theorems (Theorems \ref{intro-thm-convex}, \ref{intro-thm-one-map} and \ref{intro-thm-simplicial-complex}) whose proofs will be particularly attractive to students and newcomers to applied topology. This is because the arguments are relatively elementary and use techniques that are interesting in their own right. In \cref{section:closed-convex},
we prove the following functorial nerve theorem:

\begin{thmx}[\cref{thm:comp_conv_zigzag}] \label{intro-thm-convex}
    If $X \subset \bbR^d$, and $\mathscr{B} = ( C_i )_{i\in[n]}$ is a cover by closed convex subsets,
then the natural maps $\rho_S \colon \SmallBarCon(\mathscr{B}) \to X$ and
$\rho_N \colon \SmallBarCon(\mathscr{B}) \to |\Nrv(\mathscr{B})|$ are homotopy equivalences.
\end{thmx}

The proof uses partitions of unity,
and is similar to the strategy for open covers in Hatcher's textbook.
In \cref{section:closed-convex}, we also prove a functorial nerve theorem for closed convex covers
that does not require any intermediate object,
subject to an additional assumption on the morphisms of covered spaces.
Before we state this theorem,
we elaborate shortly on why such a functorial nerve theorem cannot exist in general.
The reason is simple: there are no natural transformations between $\Spc$ and $|\Nrv|$ in either direction.
Consider the covered spaces~$(*,(*))$, where~$*$ is the one-point space, and $(Y,(Y))$, where $Y\neq *$ is any space. For any point~$p\in Y$ the inclusion $\iota_p\colon *\hookrightarrow Y$ gives rise to a morphism of covered space~$(\iota_p, *\mapsto Y)$. If there existed a natural transformation $|\Nrv|\Rightarrow \Spc$, then this would already fix a single inclusion $\iota_q\colon *\hookrightarrow Y$ as part of such a morphism of covered spaces, implying that~$Y=\{q\}$ is a single point, yielding a contradiction. Similarly, consider any covered space $(Z,(U,V))$ with $p\in U\cap V$ any point. Consider the two morphisms of covered spaces $(\iota_p,*\mapsto U),(\iota_p,*\mapsto V)\colon (*,(*))\to (Z,(U,V))$. Then, these maps induce different simplicial maps on the nerves, implying that there exists no natural transformation~$\Spc\Rightarrow|\Nrv|$.

Thus, in order to obtain a functorial nerve theorem that does not need an intermediate object, the map of indexed covers needs to have strong combinatorial control on the continuous map. To this end, we introduce the following notions.

\begin{defi}
\label{defi:pointed_covered}
 A \emph{pointed cover} $\mathscr{U}_*=(\mathscr{U}=(U_i)_{i\in I}, (u_\sigma)_{\sigma\in \Nrv (\mathscr{U})})$ of a topological space~$X$ is a cover $\mathscr{U}$ of $X$ together with a point $u_\sigma\in U_\sigma$ for every $\sigma \in \Nrv (\mathscr{U})$.

 The \emph{category of pointed covered spaces} $\Cov_*$ has objects tuples of the form $(X,\mathscr{A}_*)$,
where $X$ is a topological space and $\mathscr{A}_*=(\mathscr{A},(a_\sigma)_{\sigma\in\Nrv{\mathscr{A}}})$ is a pointed cover of $X$.
A morphism $(f,\varphi)\colon (X, \mathscr{A}_*) \to (Y, \mathscr{B}_*)$ of pointed covered spaces is a morphism of covered spaces $(f,\varphi)\colon (X, \mathscr{A}) \to (Y, \mathscr{B})$ that respects the basepoints, i.e., such that for any $\sigma\in \Nrv (\mathscr{A})$ we have~$f(a_\sigma)=b_{\varphi_*(\sigma)}$.
\end{defi}
There is an obvious functor $\Cov_*\to\Cov$ that forgets the pointing, and hence we get, as for the category of covered space, the functors $\Spc \colon \Cov_* \to \Top$ and $|\Sd\Nrv| \colon \Cov_* \to \Top$,
where $\Sd\Nrv$ is the subdivision of the nerve.

Now, we will describe a functorial nerve theorem that does not require an intermediate object. The subcategory $\ClConv_*$ of $\Cov_*$ consists of subsets of $\mathbb{R}^d$ that are covered by finitely many closed convex sets. Further, we restrict to morphisms of pointed covered spaces whose underlying continuous maps are affine linear on each cover element.
Many covers of interest in topological data analysis are pointed.

\begin{ex}
 Let $\{x_0,\dots,x_n\}\subseteq \mathbb{R}^d$ be a finite set of points. Then, we can point the cover $\mathscr{U}_r = (\CB{r}{x_i})_{i\in[n]}$ of the union of closed balls $O_r=\bigcup_{i=0}^n\CB{r}{x_i}$ in the following way: For each non-empty subset $\sigma\subseteq [n]$ there exists a smallest real number $r_\sigma$ such that the intersection $(\mathscr{U}_{r_\sigma})_\sigma$ is non-empty. We define the point $p_\sigma$ to be the unique point in this intersection. This gives the pointed cover $(\mathscr{U}_r,(p_\sigma)_{\sigma\in \Nrv(\mathscr{U}_r)})$ of $O_r$ for each $r\in\mathbb{R}_{\geq 0}$. With this at hand, we see that the offset filtration is a functor $\mathbb{R}_{\geq 0}\to \ClConv_*$.
\end{ex}

\begin{thmx}[\cref{one-arrow-functorial-nerve-thm}]  \label{intro-thm-one-map}
 For every pointed covered space $(X,\mathscr{A}_*)\in \ClConv_*$ there exists a homotopy equivalence \[\Gamma\colon |\Sd\Nrv(\mathscr{A})|\to X\]
 that is natural with respect to the morphisms in $\ClConv_*$.
\end{thmx}
\goodbreak

One says that a cover is \emph{good} if all non-empty finite intersections of cover elements are contractible. As we have already mentioned, nerve theorems usually assume that the covers involved are good. 
In \cref{section:triangulated-covers},
we again use the blowup complex to prove a functorial nerve theorem for simplicial complexes:

\begin{thmx}[\cref{functorial-simplicial-nerve}] \label{intro-thm-simplicial-complex}
Let $K$ be a simplicial complex and let $\mathscr{A} = (K_i \subseteq K)_{i \in I}$
be a good cover of $K$ by subcomplexes.
The natural maps $\rho_S \colon \SmallBarCon(|\mathscr{A}|) \to |K|$
and $\rho_N \colon \SmallBarCon(|\mathscr{A}|) \to |\Nrv(\mathscr{A})|$ are homotopy equivalences.
\end{thmx}

The proof is related to work of Bj{\"o}rner \cite{bjoerner,MR607041},
and uses elementary methods from combinatorial homotopy theory
for constructing homotopy equivalences between simplicial complexes,
together with discrete Morse theory.
Combining this result with a well-known
theorem on triangulations of semi-algebraic sets (\cref{triangulation}),
we obtain a nerve theorem for compact semi-algebraic sets
that are covered by finitely many closed, semi-algebraic subspaces.

Finally we use techniques from abstract homotopy theory to prove the following omnibus functorial nerve theorem.
In particular, this result implies
Theorems \ref{intro-thm-convex} and \ref{intro-thm-simplicial-complex}.
In parts 1(b) and 2(b) of the following \cref{intro-thm-unified},
we restrict attention to compactly-generated spaces.
This is a standard hypothesis in algebraic topology,
as these spaces form a ``convenient'' subcategory of topological spaces
that is suitable for developing the machinery of homotopy theory.
In part 1(b), the intersection ${A}_{T}$
and the \emph{latching space}
$L(T) = \bigcup_{T \subsetneq J} {A}_{J} \subseteq {A}_{T}$
are assumed to satisfy the homotopy extension property;
for example, CW-pairs satisfy the homotopy extension property (see \cref{rmk:cw_hep}).
These assumptions on the latching spaces together with the assumption that the cover is locally finite dimensional allow for inductive arguments
analogous to arguments that employ induction over the skeleton of a CW-complex.
At the beginning of \cref{section:unified-nerve-thm}
we introduce in detail all of the notions used in the statement of the following theorem.

\begin{thmx}[Unified Nerve \cref{abstract-nerve-thm}]
\label{intro-thm-unified}
Let $X$ be a topological space and let $\mathscr{A} = (A_i)_{i \in I}$ be a cover of $X$.
\begin{enumerate}
	\item Consider the natural map $\rho_{S} \colon \SmallBarCon(\mathscr{A}) \to X$.
	\begin{enumerate}[(a)]
		\item If $\mathscr{A}$ is an open cover,
		then $\rho_{S}$ is a weak homotopy equivalence.
		If furthermore $X$ is a paracompact Hausdorff space, or, more generally, if $\mathscr{A}$ is numerable, then $\rho_{S}$ is a homotopy equivalence.
		\item Assume that $X$ is compactly generated
		and that $\mathscr{A}$ is a closed cover
		that is locally finite and locally finite dimensional.
		If for any $T \in \Nrv(\mathscr{A})$ the latching space $L(T) \subseteq {A}_{T}$
		is a closed subset and the pair $({A}_{T},L(T))$ satisfies the homotopy extension property,
		then $\rho_{S}$ is a homotopy equivalence.
	\end{enumerate}
	\item Consider the natural map
	    $\rho_{N} \colon \SmallBarCon(\mathscr{A}) \to | \Nrv(\mathscr{A}) |$.
	\begin{enumerate}[(a)]
	    \item If $\mathscr{A}$ is (weakly) good, then $\rho_{N}$
		is a (weak) homotopy equivalence.
		\item If for all $J \in \Nrv({\mathscr{A}})$ the space ${A}_{J}$ is compactly generated
		and $\mathscr{A}$ is homologically good with respect to a coefficient ring $R$,
		then $\rho_{N}$ is an $R$-homology isomorphism.
	\end{enumerate}
\end{enumerate}
\end{thmx}

We now summarize the ingredients that go into the proof.
First of all, the blowup complex is closely related to another standard construction,
called the bar construction.
While the blowup complex has a natural map to the nerve of a cover,
the bar construction has instead a natural map to the subdivision of the nerve,
which is why we use the blowup complex in the statements of the theorems.
However, the bar construction is in some ways easier to work with;
in \cref{section:func_nerve} we define the bar construction and explain 
how we can work with either the blowup complex or the bar construction, 
whichever is more~convenient.

The statement in 1(a) about weak homotopy equivalences follows from work of Dugger--Isaksen
\cite{dugger-isaksen}; we give a short proof in \cref{appendix-open-covers}.
The statement in 1(a) about homotopy equivalences is proved in Hatcher's textbook
\cite[Proposition 4G.2]{hatcher}.
Statement 1(b) follows from a standard argument using Reedy model structures,
which is similar to the proof of \cite[Corollary 14.17]{dugger-hocolim}, for example.
Both parts of 2(a) follow from the fact that the bar construction is homotopical, 
both for homotopy equivalences and weak homotopy equivalences.
In the case of weak homotopy equivalences, this again uses Dugger--Isaksen
\cite{dugger-isaksen}; see \cref{prop:weak_homotopy_lemma_top}.
We also give a short proof of this in \cref{appendix-open-covers}.
Finally, 2(b) is proved using the bar construction in the setting of simplicial $R$-modules.

In summary, the proof of \cref{intro-thm-unified} is straightforward,
given some powerful tools for studying diagrams of spaces from abstract homotopy theory.
In \cref{section:unified-nerve-thm} we provide an introduction to these tools.
We hope this will be useful to interested members of the applied topology community.

\paragraph{The Literature on Nerve Theorems}
We now summarize the literature on the nerve theorem, with a particular focus on results that address functoriality.

The original work of Alexandroff \cite{MR1512423} on nerves,
as well as the early nerve theorems of Leray \cite[Théorème~12]{MR15786} \cite{MR157376} and Borsuk \cite[Corollary 3]{Borsuk1948},
considered closed covers, motivated in part by covers of polytopes by simplices.
Open covers were considered by Weil \cite[Section 5]{weil},
McCord \cite{mccord}, and Segal \cite{segal}.
There is renewed interest in the case of closed covers in applied topology,
motivated by geometric constructions such as alpha shapes
\cite{edelsbrunner-kirkpatrick-seidel}.

A common way to relate the nerve of an open cover $\mathscr{U}$ with the covered space $X$
is by a partition of unity subordinate to the cover.
Such a partition of unity defines a map from $X$ to the nerve of $\mathscr{U}$
in a straightforward way, which is a homotopy equivalence if $\mathscr{U}$ is good.
This idea appears in the work of Weil and Segal,
and the textbook proofs of the nerve theorem by Hatcher \cite[Corollary 4G.3]{hatcher} and Kozlov \cite[Theorem 15.21]{kozlov}.
Moreover, up to homotopy this map is independent of the choice of partition of unity (see \cite[Ch. X.11]{eil_steen} for a discussion in a slightly different setting),
and this map commutes up to homotopy with the maps on spaces and nerves
induced by a morphism of covered spaces.
This was observed by Chazal--Oudot \cite{chazal-oudot}
(for certain inclusions of covered spaces)
and by Bauer--Edelsbrunner--Jab{\l}o{\'n}ski--Mrozek \cite{bauer-edelsbrunner-jablonski-mrozek}
(for general morphisms);
for similar results, see 
Lim--M{\'e}moli--Okutan \cite[Theorem 6]{lim-memoli-okutan} %
and Virk \cite[Lemma 5.1]{virk21}.%

In the case of open covers of paracompact Hausdorff spaces,
Botnan--Spreemann \cite{bs-approximating}, following the approach that goes back at least to Segal,
observe that the blowup complex provides a zigzag of natural transformations relating
covered spaces and nerves.

Ferry--Mischaikow--Nanda \cite{ferry-mischaikow-nanda}
consider covers by open and closed balls in Euclidean space;
they also use the blowup complex as an intermediate object relating spaces and nerves,
but use the Vietoris--Smale theorem on proper maps with contractible fibers
to obtain a homotopy equivalence from nerve to space with control over the image of each simplex.

Bendich--Cohen-Steiner--Edelsbrunner--Harer--Mo\-ro\-zov \cite{bcehm-inferring}
give a nerve theorem for certain closed convex covers in Euclidean space;
they define a map from the subdivision of the nerve of $\mathscr{U}$ to the space by
choosing a point in each non-empty intersection~$\mathscr{U}_J$,
mapping the vertex $J$ to this point, and extending by piecewise linear interpolation.
They show this map commutes up to homotopy with maps induced by certain morphisms of covered spaces.

The references in the previous four paragraphs also give examples of applications in which functoriality of the nerve theorem is important. For more, see, e.g., work on approximate nerve theorems \cite{appr_nerve,gen_pers_nerve} and a comparison of persistent singular and \v{C}ech homology \cite{comparison-singular-cech}.

Borel--Serre \cite[Theorem 8.2.1]{MR387495} prove a nerve theorem for locally finite and closed covers whose nerve is finite dimensional and such that all finite non-empty intersections of cover elements are absolute retracts for metric spaces. Using similar techniques, Nag\'{o}rko \cite{nagorko} proves a nerve theorem for locally finite, locally finite dimensional,
star-countable closed covers of normal spaces
such that all non-empty intersections of cover elements are absolute extensors for metric spaces.

Bj\"{o}rner \cite{bjoerner} gives a proof of an $n$-connectivity version of the nerve theorem,
which we discuss in \cref{bjoerners-nerve-theorem}. 
Given a good cover of a finite simplicial complex by subcomplexes,
Barmak \cite{barmak} proves a related result, 
showing that the simplicial complex and the nerve
have the same simple homotopy type.

\section{Functorial Nerve Theorems via Homotopy Colimits}
\label{section:func_nerve}

Theorems \ref{intro-thm-convex}, \ref{intro-thm-simplicial-complex}, and \ref{intro-thm-unified}
use the \emph{blowup complex} $\SmallBarCon(\mathscr{U})$ of a cover $\mathscr{U}$
as an intermediate object to relate the nerve of $\mathscr{U}$ with the covered space.
In this section, we define the blowup complex and its natural maps
to the covered space and the nerve.
The construction is not difficult, but there is an important point here:
the blowup complex is closely related to the \emph{bar construction},
and because of this, properties of the bar construction are used in many proofs of the nerve theorem,
including Theorems \ref{intro-thm-convex} and \ref{intro-thm-unified}.

The bar construction is a standard model for the homotopy colimit:
like the colimit, the homotopy colimit can be defined via a universal property,
but since this universal property is phrased in terms of derived categories,
it takes some work to define it precisely.
A full discussion of the homotopy colimit is beyond the scope of this paper
(see \cite{dugger-hocolim} for a nice introduction to the topic,
or \cite[Part I]{cat_hom} for a more abstract approach).
However, in order to explain the properties of the bar construction that we will use,
we will at least describe the problem that the homotopy colimit addresses.
So, in this section we will introduce a basic problem with the colimit of a diagram of topological spaces,
give an idea of how the homotopy colimit addresses this problem,
define the bar construction and the blowup complex,
and explain how the properties of the bar construction can be used to prove functorial nerve theorems
using the blowup complex.

First we establish some notation and recall some facts related to the barycentric subdivision of a simplicial complex.
Write $\Po$ for the category of posets, and $\SimpComplex$ for the category of simplicial complexes.
Let $\Pos \colon \SimpComplex \to \Po$ be the functor
that takes a simplicial complex to its poset of simplices (ordered by inclusion),
and let $\Flag \colon \Po \to \SimpComplex$ be the functor
that takes a poset $P$ to the simplicial complex
whose vertices are the elements of~$P$ and
whose $n$-simplices are the chains $x_0 < \dots < x_n$ of elements of~$P$.
Then the barycentric subdivision of a simplicial complex~$K$ is
$\Sd(K) = \Flag(\Pos(K))$.
There is an affine linear homeomorphism $\alpha_K \colon |\Sd(K)| \to |K|$ defined by the vertex map that sends a vertex $\sigma$ of $\Sd(K)$ to the barycenter of $|\sigma|$ in $|K|$.
Note that, while the homeomorphism $\alpha_K$ is natural with respect to inclusions of simplicial complexes,
it is not natural with respect to general simplicial maps.

\subsection{Homotopy Colimits and the Bar Construction}
\label{section-hocolim-bar}

While colimits are used everywhere in topology to construct new spaces,
the colimit operation fails to respect homotopy equivalences, in the following sense.
Take $\mathscr{A}$ to be the category that looks like this:
\[
	\bullet \leftarrow \bullet \rightarrow \bullet
\]
and consider the commutative diagram:
\[
\begin{tikzcd}
	D^n \arrow[d] \arrow[r, hookleftarrow]
	& S^{n-1} \arrow[r, hookrightarrow] \arrow[d, "\mathrm{id}"]
	& D^n \arrow[d] \\
	* & S^{n-1} \arrow[l] \arrow[r] &*
\end{tikzcd}
\]
Here, the top maps are the boundary inclusions.
We think of the rows as $\mathscr{A}$-shaped diagrams,
and the vertical maps define a natural transformation between these two $\mathscr{A}$-shaped diagrams.
Every component of this natural transformation is a homotopy equivalence,
but the colimit of the top row is the sphere $S^n$,
while the colimit of the bottom row is a one-point space $*$,
so the induced map between the colimits cannot be a homotopy equivalence.

More generally, let $\mathscr{C}$ be a small category,
and write $\Top^{\mathscr{C}}$ for the category of functors $\mathscr{C} \to \Top$.
One says that a functor $\Omega \colon \Top^{\mathscr{C}} \to \Top$
is \emph{homotopical} if, given a natural transformation
$\lambda \colon F \Rightarrow F'$ between $\mathscr{C}$-shaped diagrams $F$ and $F'$
that is a pointwise homotopy equivalence,
the induced map $\Omega(F) \to \Omega(F')$ is also a homotopy equivalence.
For any small category $\mathscr{C}$,
the colimit defines a functor $\op{colim} \colon \Top^{\mathscr{C}} \to \Top$,
and the previous example shows that this functor is not homotopical in general.

A homotopy colimit is a homotopical functor $\op{hocolim} \colon \Top^{\mathscr{C}} \to \Top$,
together with a natural transformation $\op{hocolim} \Rightarrow \op{colim}$
that makes $\op{hocolim}$, in some sense,
the best possible homotopical approximation of the colimit functor.
We now show how to construct a particular model for the homotopy colimit,
called the bar construction,
and we will see that it can be thought of as a ``thickened'' version of the colimit; see \cref{ex:circle_arcs} for an illustration.

We write $\Delta^{n}$ for the standard topological $n$-simplex,
and for $0 \leq i \leq n$,
we write $d^i \colon \Delta^{n-1} \hookrightarrow \Delta^{n}$
for the inclusion of the face opposite the $i^{th}$ vertex.

\begin{defi} \label{defi:hocolim_spaces}
Let $P$ be a poset, and let $F \colon P \to \Top$ be a diagram of topological spaces.
The \emph{bar construction} $\BarCon(F)$ of $F$ is the quotient space
\[
	\BarCon(F) = \left( \bigsqcup_{\sigma = (v_0 < \cdots < v_n)}
	 F(v_0)\times \Delta^n  \right) \; / \; \sim
\]
where the disjoint union is taken over all chains in $P$, and
the equivalence relation $\sim$ is generated as follows.
For a chain $\sigma = (v_0 < \cdots < v_n)$ and $0 \leq i \leq n$,
we write
\[\tau_i = (v_0 < \cdots < \hat{v_i} < \cdots < v_n) = (w_0 < \cdots < w_{n-1})\]
for the subchain with $v_i$ left out, noting that if $i>0$, then $w_0=v_0$, and if $i=0$, then $w_0=v_1$.
Now for any $x \in F(v_0)$ and $\alpha \in \Delta^{n-1}$,
	we identify $(x,d^i(\alpha))$ in the copy of $F(v_0)\times \Delta^n $
	indexed by $\sigma$ with $(F(v_0 \leq w_0)(x),\alpha)$
	in the copy of $F(w_0)\times \Delta^{n-1} $
	indexed by $\tau_i$.
\end{defi}

\begin{ex}
Let $P = \{ 0 < 1 \}$. Then a diagram $F \colon P \to \Top$ is just a map
$F(0) \to F(1)$, and the bar construction $\BarCon(F)$
is the mapping cylinder of this map.
\end{ex}

\begin{defi} \label{defi:nerve_diagram}
Let $X$ be a topological space and $\mathscr{U} = (U_i)_{i \in I}$ a cover of $X$.
Writing $P_{\mathscr{U}} = \Pos(\Nrv(\mathscr{U}))^{\opp}$, the \emph{nerve diagram} of the cover $\mathscr{U}$ is the functor $\mathscr{D}_{\mathscr{U}} \colon P_{\mathscr{U}} \to \Top$ with
$\mathscr{D}_{\mathscr{U}}(J) = U_J$.
\end{defi}

\begin{rmk}
\label{space-is-colimit}
In many cases, the colimit of the diagram $\mathscr{D}_{\mathscr{U}}$ simply gives us back $X$:
the inclusions $U_J \subseteq X$ induce a continuous map $\op{colim} \mathscr{D}_{\mathscr{U}} \to X$,
which is in fact a bijection.
If $\mathscr{U}$ is an open cover, or if it is a closed cover that is \emph{locally finite}
(i.e., every point of $X$ has an open neighborhood that intersects only finitely many cover elements),
then this bijection is a homeomorphism.
\end{rmk}

\begin{ex}
\label{ex:circle_arcs}
In this paper, we will mainly consider bar constructions of diagrams associated to a cover.
For example, consider the following cover $\mathscr{U}$ of the circle $S^1$:

\begin{center}
\begin{minipage}[c]{0.4\textwidth}
\centering
  \includegraphics[width=0.6\textwidth]{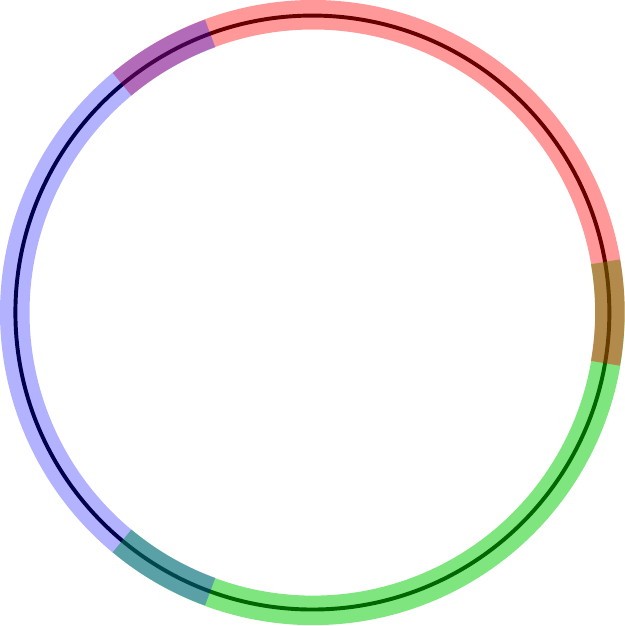}
\end{minipage}
\begin{minipage}[c]{0.4\textwidth}
\centering
  \includegraphics[width=0.6\textwidth]{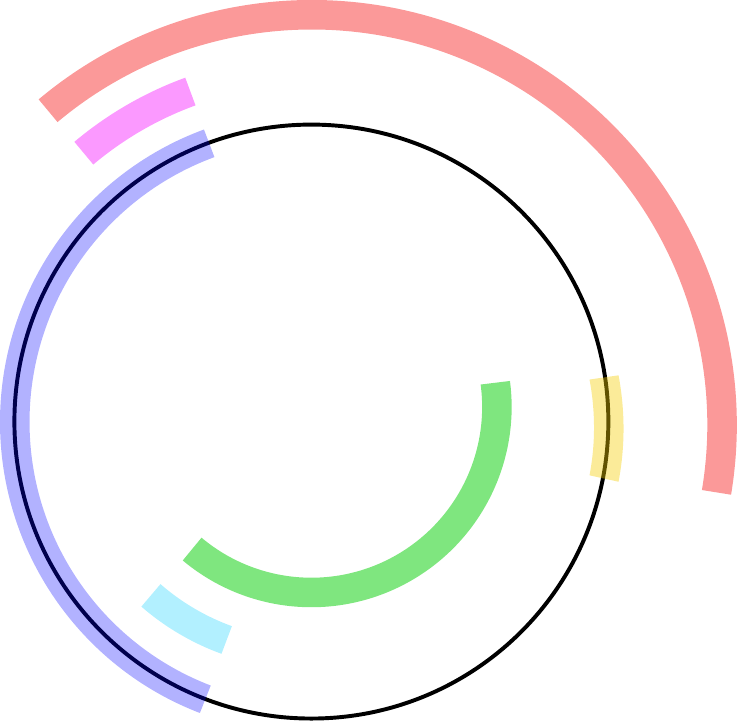}
\end{minipage}
\begingroup
\captionof{figure}{A cover by three arcs (left) and the intersections of those (right).}
\endgroup
\end{center}

If we label the three arcs $\{b,r,g\}$, the poset $P_{\mathscr{U}}$ associated to this cover
has the following~form:
\[
\begin{tikzcd}[column sep={5ex,between origins}]
	& & & \{r\} & &  \\
	& & \{b,r\} \arrow[ur] \arrow[dl] & & \{r,g\} \arrow[ul] \arrow[dr] & \\
	& \{b\} & & \{b,g\} \arrow[ll] \arrow[rr] & & \{g\}
\end{tikzcd}
\]

By definition, the bar construction $\BarCon(\mathscr{D}_{\mathscr{U}})$ of the nerve diagram $\mathscr{D}_{\mathscr{U}}\colon P_{\mathscr{U}} \to \Top$ associated to the cover $\mathscr{U}$ is built from pieces indexed by chains $v_0 < \cdots < v_n$ in $P_{\mathscr{U}}$ and are of the form $\mathscr{D}_{\mathscr{U}}(v_0)\times \Delta^n$. More concretely, the bar construction in our example is built from the following pieces:

\begin{center}
\begin{minipage}[c]{0.4\textwidth}
\centering
  \includegraphics[width=0.6\textwidth]{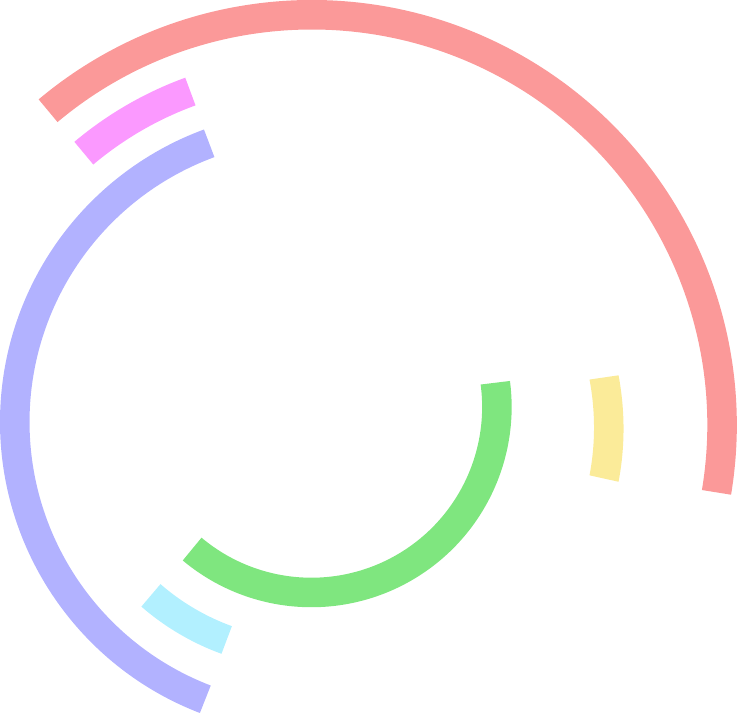}
\label{fig:zero_simp}
\end{minipage}
\begin{minipage}[c]{0.4\textwidth}
\centering
  \includegraphics[width=0.6\textwidth]{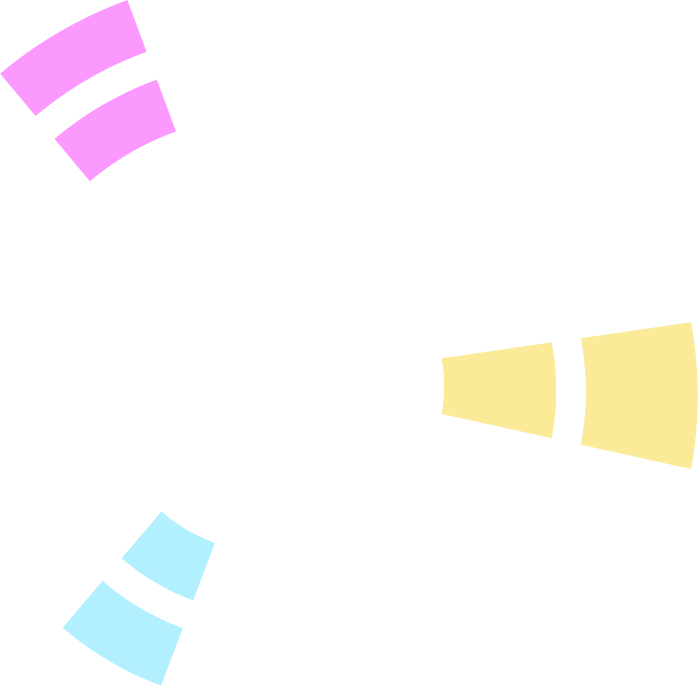}
\label{fig:one_simp}
\end{minipage}
\begingroup
\captionof{figure}{Pieces indexed by chains in $P_{\mathscr{U}}$ of length zero (left) and of length one (right).}
\endgroup
\end{center}

After making all identifications,
the bar construction $\BarCon(\mathscr{D}_{\mathscr{U}})$
is the following ``thickened'' version
of $\op{colim} \mathscr{D}_{\mathscr{U}} \iso S^1$:

\begin{center}
\begin{minipage}[c]{0.4\textwidth}
\centering
  \includegraphics[width=0.6\textwidth]{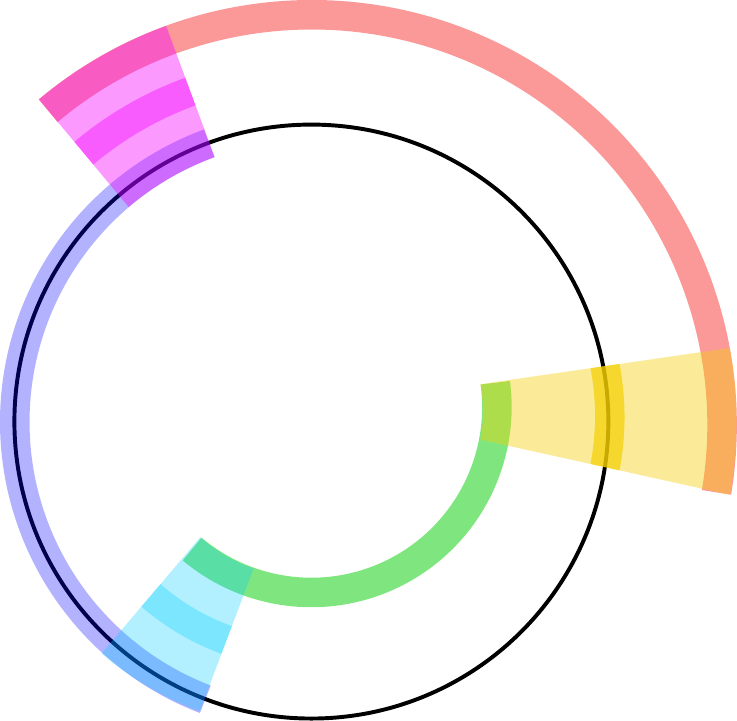}
\end{minipage}
\begingroup
\captionof{figure}{The bar construction of the nerve diagram.}
\endgroup
\end{center}
\end{ex}

By $\op{Diag}_{\Po}(\Top)$ we denote the category of diagrams over a poset: the objects are tuples $(P,F)$, where $P$ is a poset and $F\colon P\to \Top$ is a functor. A morphism $(g,\lambda)\colon (P,F)\to (R,G)$ consists of a poset map $g\colon P\to R$ and a natural transformation $\lambda\colon F\Rightarrow G\circ g$. Then the bar construction defines a functor
$\BarCon \colon \op{Diag}_{\Po}(\Top) \to \Top$:
a morphism $(g,\lambda)$ induces a continuous map $\BarCon(F) \to \BarCon(G)$
defined by the maps
\[
	\lambda(v_0)\times |\Flag(g)|  \colon F(v_0)\times \Delta^n  \to G(g(v_0))\times \Delta^m  \; ,
\]
where $|\Flag(g)|\colon \Delta^n\to \Delta^m$ is the affine map that sends the vertex $v_i$ to $g(v_i)$.
Moreover, the projection maps $F(v_0)\times \Delta^n  \to F(v_0)$ define a natural map
$\BarCon(F) \to \op{colim} F$.

There are analogues of this quotient space construction in other settings, which are also called bar constructions. We will encounter some of these in \cref{section:unified-nerve-thm}.
For a very general discussion of the bar construction,
including a proof that it is a model of the homotopy colimit, see
\cite[Chapters 4--5]{cat_hom}.
The bar construction for topological spaces is homotopical
(see \cite[Theorem~15.12]{kozlov} or \cite[Proposition~4G.1]{hatcher}):

\begin{prop} \label{prop:homotopy_lemma_top}
Let $P$ be a poset, $F, G \colon P \to \Top$ diagrams of topological spaces,
and let $\lambda \colon F \Rightarrow G$ be a natural transformation.
If the component $\lambda(v)$ is a homotopy equivalence for all $v \in P$,
then so is the induced map
$\BarCon(F) \to \BarCon(G)$.
\end{prop}

\subsection{Functorial Nerve Theorems via the Bar Construction}
\label{functorial_nerve_via_bar}

We can now explain how the bar construction can be used to prove functorial nerve~theorems.
For any poset $P$, if we write $*^P$ for the diagram $P \to \Top$ with constant value the one-point space $*$,
there is a canonical identification $\BarCon(*^P) \cong |\Flag(P)|$.
In particular, if~$P_{\mathscr{U}}$ is the poset associated to a cover $\mathscr{U}$, we have $\BarCon(*^{P_{\mathscr{U}}}) \cong |\Flag(P_{\mathscr{U}})| = |\Sd \Nrv(\mathscr{U})|$.

A morphism of covered spaces $(f, \varphi) \colon (X, \mathscr{U}) \to (Y, \mathscr{V})$ induces a poset map $g\colon {P_\mathscr{U}\to P_\mathscr{V}}$ and a natural transformation $\lambda \colon\mathscr{D}_\mathscr{U}\Rightarrow \mathscr{D}_\mathscr{V} \circ g $. Thus, by what we have seen before, the operation $(X, \mathscr{U}) \mapsto \BarCon(\mathscr{D}_{\mathscr{U}})$
defines a functor $\Cov \to \op{Diag}_{\Po}(\Top)\to  \Top$. Moreover, the unique natural transformation $\mathscr{D}_{\mathscr{U}} \Rightarrow *^{P_{\mathscr{U}}}$
induces a natural map
\[
	\pi_{\Sd N} \colon \BarCon(\mathscr{D}_{\mathscr{U}}) \to \BarCon(*^{P_{\mathscr{U}}})
	\cong |\Sd \Nrv(\mathscr{U})| \; .
\]
If every non-empty finite intersection of cover elements happens to be contractible,
then this map is a homotopy equivalence by \cref{prop:homotopy_lemma_top}.
Using also the natural map from the bar construction to the colimit as mentioned at the end of \cref{section-hocolim-bar},
we get a diagram that is natural in morphisms of covered spaces:
\[
	X \leftarrow \op{colim} \mathscr{D}_{\mathscr{U}}
	\leftarrow \BarCon(\mathscr{D}_{\mathscr{U}}) \to |\Sd \Nrv(\mathscr{U})| \; .
\]

In \cref{section:closed-convex} and \cref{section:unified-nerve-thm}
we will use this diagram to prove functorial nerve theorems
by finding various sets of assumptions that make these natural maps equivalences of various kinds.
This strategy -- also employed in the well-known proof of the nerve theorem for open covers
in Hatcher's textbook \cite[Section 4.G]{hatcher} -- relies on the well-known good properties of the bar construction.
We exploit this established theory repeatedly in \cref{section:unified-nerve-thm},
where we use the fact that the bar construction is homotopical in several contexts,
including homological algebra.
In \cref{section:triangulated-covers},
we will prove a functorial nerve theorem using a bar construction constructed in the category of posets,
rather than the topological construction.

However, for purposes of computational topology,
we want a nerve theorem to relate the space $X$ directly with the nerve of $\mathscr{U}$,
not the much larger subdivision of the nerve.
In order to obtain a functorial nerve theorem that works for morphisms of covered spaces as we have defined them,
in which the map of indexed covers need not be an inclusion,
we cannot simply apply the usual homeomorphism $\alpha_K \colon |\Sd(K)| \to |K|$
defined for any simplicial complex $K$,
as this map is natural only in inclusions of simplicial complexes.
In the case of diagrams $\mathscr{D}_{\mathscr{U}}$ associated to a cover,
the \emph{blowup complex} is an efficient way to build a space homeomorphic to the bar construction,
which comes with a natural map to $|\Nrv(\mathscr{U})|$ rather than $|\Sd \Nrv(\mathscr{U})|$.
In the definition, for a non-empty finite set $J$,
we write~$|J|$ for the geometric realization of the full simplicial complex generated by~$J$,
which is homeomorphic to the standard topological simplex $\Delta^{|J|-1}$.

\begin{defi}
\label{defi:blowup}
Let $\mathscr{U} = ( U_i )_{i \in I}$ be a cover of a topological space $X$.
The \emph{blowup complex} $\SmallBarCon(\mathscr{U})$ is the quotient space
\[
	\SmallBarCon(\mathscr{U}) = \left( \bigsqcup_{J \in \Nrv(\mathscr{U})}
	 U_J \times |J|  \right) \; / \; \sim ,
\]
where the disjoint union is taken over all simplices $J \in \Nrv(\mathscr{U})$,
and the equivalence relation~$\sim$ identifies, for all $J \subseteq J'$,
the spaces
$U_J \times |J|$
and
$U_{J'} \times |J'|$
along their common subspace
$U_{J'} \times |J|$.

\end{defi}

\begin{rmk}
For a finite cover $\mathscr{U} = (U_i)_{i\in[n]}$ the blowup complex can be defined as a subspace of the product $X\times \Delta^{n}$,
as mentioned in the introduction.
This is the approach used by Zomorodian and Carlsson \cite[Definition 3]{blowup}.
\end{rmk}

As before, the operation
$(X, \mathscr{U}) \mapsto \SmallBarCon(\mathscr{U})$
defines a functor $\Cov \to \Top$,
and the projection maps $U_J\times |J| \to U_J$ define a natural map
$\SmallBarCon(\mathscr{U}) \to \op{colim} \mathscr{D}_{\mathscr{U}}$,
which gives us a natural map $\rho_S \colon \SmallBarCon(\mathscr{U}) \to X$.
But now the projection maps $U_J\times |J| \to |J|$ assemble to define a natural map
$\rho_N \colon \SmallBarCon(\mathscr{U}) \to |\Nrv(\mathscr{U})|$.

\begin{ex}
As in \cref{ex:circle_arcs}, we consider the cover of the circle by three arcs.

\begin{center}
 \begin{minipage}[c]{0.4\textwidth}
\centering
 \includegraphics[width=0.6\textwidth]{circle_cover_decomposed_glued.pdf}
\end{minipage}
\begin{minipage}[c]{0.4\textwidth}
\centering
  \includegraphics[width=0.6\textwidth]{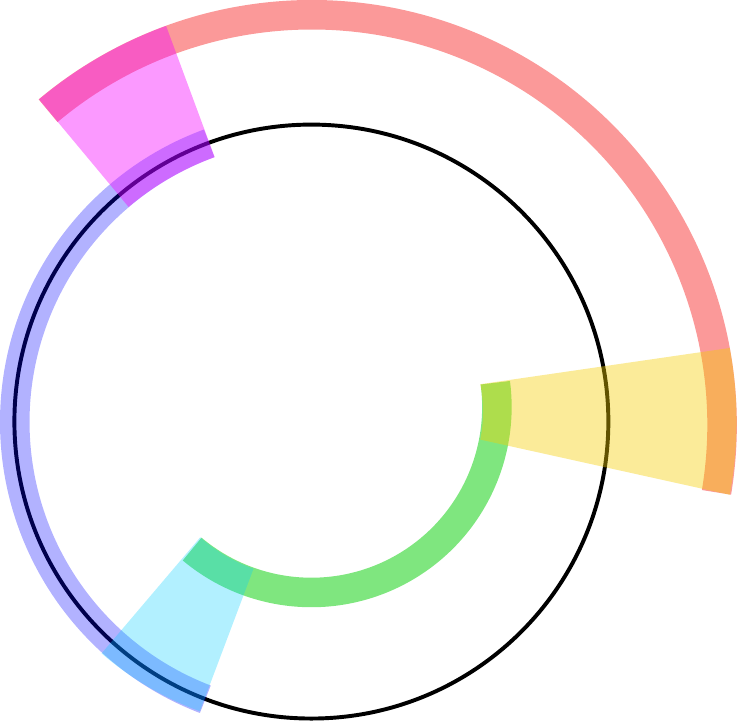}
\end{minipage}
\begingroup
\captionof{figure}{The bar construction (left) and the blowup complex (right).}
\endgroup
\end{center}
Note that the blowup complex is combinatorially simpler.
\end{ex}

We can use the homeomorphism $\alpha_K \colon |\Sd(K)| \to |K|$ defined for any simplicial complex $K$
to construct a homeomorphism
$\BarCon(\mathscr{D}_{\mathscr{U}}) \to \SmallBarCon(\mathscr{U})$.
For any simplex $J \in \Nrv(\mathscr{U})$, any flag $J \supset J_1 \supset \cdots \supset J_n$
in $P_{\mathscr{U}}$ indexes a piece $U_J\times \Delta^n $ in
$\BarCon(\mathscr{D}_{\mathscr{U}})$.
The flags of this form glue together to give a copy of
$ U_J\times | \Sd J | $ inside $\BarCon(\mathscr{D}_{\mathscr{U}})$,
where $\Sd J$ is the subdivision of the simplicial complex generated by $J$.
Now for all $J \in \Nrv(\mathscr{U})$ we have a map
\[
	U_J\times | \Sd J |  \xrightarrow{\alpha_J \times \mathrm{id}}
	U_J \times |J|  \subset \SmallBarCon(\mathscr{U}) \; ,
\]
and assembling these maps gives the homeomorphism
$\BarCon(\mathscr{D}_{\mathscr{U}}) \to \SmallBarCon(\mathscr{U})$.

This homeomorphism is not natural in arbitrary morphisms of covered spaces,
but it does fit into the following commutative diagram,
where the solid arrows are natural in morphisms of covered spaces:
\begin{equation} \label{diagram:two-roofs}
\begin{tikzcd}
	& X \arrow[d, equal] & \BarCon(\mathscr{D}_{\mathscr{U}})
	\arrow[l, "\pi_{S}"'] \arrow[r, "\pi_{\Sd N}"] \arrow[d, dashed, "\cong"]
	& \vert \Sd \Nrv(\mathscr{U}) \vert \arrow[d, dashed, "\cong"] \\
	& X & \SmallBarCon(\mathscr{U}) \arrow[l, "\rho_{S}"] \arrow[r, "\rho_{N}"']
	& \vert \Nrv(\mathscr{U}) \vert
\end{tikzcd}
\end{equation}
The somewhat subtle point here is that, even though the homeomorphism
$\BarCon(\mathscr{D}_{\mathscr{U}}) \to \SmallBarCon(\mathscr{U})$
is not natural in arbitrary morphisms of covered spaces,
we can use this homeomorphism and the good properties of the bar construction
to prove functorial nerve theorems for the blowup complex:
if some set of assumptions on the covered space~$(X, \mathscr{U})$
imply that the top maps from $\BarCon(\mathscr{D}_{\mathscr{U}})$
are equivalences of some kind,
then the commutativity of the diagram \ref{diagram:two-roofs}
implies that the bottom maps from $\SmallBarCon(\mathscr{U})$ are equivalences of the same kind.

\section{Nerve Theorems for Closed Convex Covers}
\label{section:closed-convex}

In this section, we consider nerves of finite closed and convex covers of subsets of $\mathbb{R}^d$. The motivating instances are alpha complexes and \v{C}ech complexes of finite point sets. Our approach to prove a nerve theorem in this context is elementary and does not make explicit use of abstract homotopy-theoretic machinery. Nevertheless, it foreshadows the concepts that will take on a central role in~\cref{section:unified-nerve-thm}.

We write $[n]$ for the set $\{0,\dots,n\}$. Let $\mathscr{C}=(C_i)_{i\in [n]}$ be a collection of closed convex subsets of $\bbR^d$, and let $X$ be their union. We construct a continuous map
$\Gamma\colon |\Sd \Nrv(\mathscr{C})| \to X$
and show that this map is a homotopy equivalence, establishing a nerve theorem for this setting. We then extend this result to prove the two functorial versions discussed in the introduction, \cref{intro-thm-convex} and \cref{intro-thm-one-map}.

Each vertex $J\in \Sd  \Nrv(\mathscr{C})$ represents a simplex in the nerve $\Nrv(\mathscr{C})$, and hence we can choose a point $p_J$ from the non-empty intersection $C_J =\bigcap_{j\in J} C_j$.
	By convexity of the cover elements in $\mathscr{C}$, this choice extends uniquely to a map $\Gamma\colon |\Sd \Nrv(\mathscr{C})| \to X$
	that is affine linear on each simplex of the barycentric subdivision; see \cref{fig:gamma} for an illustration. Similar constructions can be found in the literature \cite[Theorem 13.4]{bott_tu} \cite[p. 179]{MR1368659} \cite[p. 544]{bcehm-inferring}.
\begin{figure}[h!]
    \centering
    \includegraphics[width=0.6\textwidth]{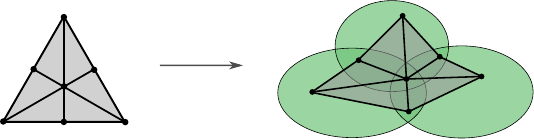}
    \caption{Illustration of the map $\Gamma$.}
    \label{fig:gamma}
\end{figure}
\begin{thm}
	\label{thm:nerve_comp_conv}
	The map $\Gamma$ is a homotopy equivalence. In particular, $|\Nrv(\mathscr{C})|$ is homotopy equivalent to $X$.
	\end{thm}

We prove the theorem by constructing a homotopy inverse $\Psi$ to $\Gamma$.
For this construction, we work with an open cover and a subordinate partition of unity, as in the familiar proof of the nerve theorem for open covers \cite[Proposition 4G.2]{hatcher}. To this end, we thicken the subsets $C_i$ slightly so that the nerve remains unchanged.
If the $C_i$ are compact, it is possible to choose an $\varepsilon$ such that the open $\varepsilon$-neighborhoods of the $C_i$ have this desired property.
More generally, we can choose such neighborhoods according to the following lemma.

\begin{lemma}
\label{lemma:thickening_choice}
Let $\mathscr{C}=(C_i)_{i\in [n]}$ be a finite collection of closed subsets of a metric space $Y$. Then there exists a collection of open subsets $\mathcal{G}=(U_i \supseteq C_i)_{i\in[n]}$ with $\Nrv(\mathscr{C}) = \Nrv(\mathcal{G})$.
\end{lemma}

\begin{proof}%
Let $X\subseteq Y$ be the union of the finitely many cover elements $C_i$. Let $p\in X$ be any point.
For every simplex $J\in \Nrv(\mathscr{C})$ such that $p$ is not contained in $C_J$ consider the positive real number 
\begin{equation}
	\label{open_thickening_radius}
	r_{p,J}=\frac{d(p,C_J)}{2}>0,
\end{equation}
where $d(p,C_J)=\inf_{q\in C_J}d(p,q)$ is the distance of $p$ to $C_J$.
If no such $J$ exists, define $r_{p}=1$.
Otherwise, define
$r_p=\min_{p\notin C_J} r_{p,J}>0$, where the minimum exists and is positive, because the nerve $\Nrv(\mathscr{C})$ is finite.
We show that the open cover elements
\begin{equation}
	\label{open_thickening_closed_cover_element}
	U_i=\bigcup\limits_{p\in C_i} \B{r_p}{p},
\end{equation}
where $\B{r_p}{p}=\{q\in Y\ |\ d(p,q)< r_p\}$ is the open ball of radius $r_p$ centered at $p$,
have the desired property $\Nrv(\mathscr{C}) = \Nrv(\mathcal{G})$.
Note first that $\Nrv(\mathscr{C}) \subseteq \Nrv(\mathcal{G})$, since $C_i \subseteq U_i$ by construction.
To show the reverse inclusion $\Nrv(\mathscr{C}) \supseteq \Nrv(\mathcal{G})$, it suffices to argue that for every $k$-simplex $J\in \Nrv(\mathcal{G})$ we have 
\begin{equation}
	\label{open_thickenting_intersections}
	U_J=\bigcup\limits_{p\in C_J} \B{r_p}{p},
\end{equation}
meaning that the intersection of the open thickenings equals the thickening of the intersection of the corresponding original cover elements. 
For $k=0$, meaning that $J=\{i\}$ for some $i\in [n]$, this is simply \cref{open_thickening_closed_cover_element}. 
Let $k>0$ and assume that \cref{open_thickenting_intersections} is true for all simplices of dimension smaller than $k$. Let $J\in \Nrv(\mathcal{G})$ be any $k$-simplex and take any partition $J_1 \sqcup J_2=J$ so that we can write $C_J = C_{J_1}\cap C_{J_2}$ and $U_J = U_{J_1}\cap U_{J_2}$.
To show \cref{open_thickenting_intersections}, note
first that by the induction assumptions $\bigcup_{p\in C_{J_1}} \B{r_p}{p} = U_{J_1}$ and $\bigcup_{p\in C_{J_2}} \B{r_p}{p} = U_{J_2}$ we have 
\[\bigcup_{p\in C_J} \B{r_p}{p}= \bigcup_{p\in C_{J_1}\cap C_{J_2}} \B{r_p}{p}\subseteq U_{J_1}\cap U_{J_2}=U_J.\]
To show the reverse inclusion, let $x\in U_J = U_{J_1}\cap U_{J_2}$ be any point.
Again by the induction assumptions, there exist $p\in C_{J_1}$ and $q\in C_{J_2}$ such that $x\in \B{r_p}{p}\cap \B{r_q}{q}$. 
Assume that neither $p$ nor $q$ are contained in $C_J = C_{J_1}\cap C_{J_2}$. Together with \cref{open_thickening_radius} and the subsequent definitions, we then get the estimate
\[
	d(p,q)\leq d(p,x)+d(x,q)<r_p+r_q\leq \frac{d(p,C_{J_2})}{2}+\frac{d(q,C_{J_1})}{2}\leq d(p,q),
\]
a contradiction. Thus, either $p$ or $q$ must be contained in $C_J = C_{J_1}\cap C_{J_2}$, and therefore 
\[
x\in \B{r_p}{p}\cap \B{r_q}{q}\subseteq \bigcup_{p\in C_J} \B{r_p}{p},
\]
proving the reverse inclusion $U_J \subseteq \bigcup_{p\in C_J} \B{r_p}{p}$.
\end{proof}

We choose $\mathcal{G}$ according to \cref{lemma:thickening_choice}.
For each $i\in [n]$, there exists a Urysohn function $\varphi_i\colon \bbR^d\to [0,1]$ that takes on the value $0$ outside of $U_i$ and the value $1$ on $C_i$.
For example, we may take \[x\mapsto\frac{d(x,\bbR^d\setminus U_i)}{d(x,C_i)+d(x,\bbR^d\setminus U_i)}.\]
	Normalizing these functions $\varphi_i$ yields a partition of unity on $X$ subordinate to the cover $(U_i \cap X)_{i\in[n]}$ of $X$: $\psi_i={\varphi_i}/{\sum_{j=0}^n \varphi_j}\colon X\to [0,1].$
	We define the map $\Phi\colon X\to |\Nrv(\mathscr{C})|$ in barycentric coordinates for $|\Nrv(\mathscr{C})|$ as
	\begin{equation}
	\label{eq:closed_convex_phi}
	\Phi \colon x \mapsto \sum\limits_{i=0}^n \psi_i(x) \cdot |v_i|,
	\end{equation}
	where $v_i=\{i\}$ is the vertex in $\Nrv(\mathscr{C})$ corresponding to $i$ and $|v_i|$ is the corresponding point in the geometric realization. The map $\Psi\colon X\to |\Sd\Nrv(\mathscr{C})|$ is then given as the composite $\alpha_{\Nrv(\mathscr{C})}\circ\Phi$, where $\alpha_{\Nrv(\mathscr{C})} \colon |\Nrv(\mathscr{C})| \to |\Sd\Nrv(\mathscr{C})|$ is the standard homeomorphism from the nerve to its barycentric subdivision.

	In order to show that $\Psi$ is a homotopy inverse to $\Gamma$, we analyze more closely how these maps are related combinatorially.
	To this end, we use the following construction.
	\begin{defi}
	For every vertex $v$ of a simplicial complex $K$, define the \textit{closed barycentric star}
	as the subspace \[\operatorname{bst} v = |\operatorname{Cl}\operatorname{St}_{\Sd K}v|\subseteq |\Sd K|,\] where
	$\operatorname{Cl}\operatorname{St}_{\Sd K}v=\{\sigma\in \Sd K \mid \sigma\cup\{v\}\in \Sd K \}$
	is the closure of the star of $v$ in the barycentric subdivision of $K$.
\end{defi}

We now state two lemmas about the closed barycentric stars, deferring the proofs to \cref{appendix-auxiliary-lemmas}.

\begin{lemma}
	\label{lemma:int_contractible}
	Let $K$ be a simplicial complex and let $\sigma\in K$ be a simplex. Then the intersection $\bigcap_{v \in \sigma}\operatorname{bst}v$ is contractible. In particular, the collection of closed barycentric stars forms a good cover of $|K|$.
\end{lemma}
It is not hard to see that the nerve of this cover is isomorphic to $K$.
The following statement describes the closed barycentric stars in terms of barycentric coordinates.
\begin{lemma}
	\label{lemma:bst_maximal}
	Let $K$ be a simplicial complex and let $v$ be a vertex of $K$.
	The closed barycentric star $\operatorname{bst} v$ consists of all points $x \in |K|$ that satisfy
	\begin{equation}
	\label{equation:bst_maximal}
	b_v(x)\geq b_w(x) \quad \text{for all }w\in \vertx{K} ,
	\end{equation}
	where $b_v$ denotes the barycentric coordinate with respect to the vertex $v$.
\end{lemma}

\begin{prop}
\label{prop:psi_carried}
The pair of maps
$(\Psi, \operatorname{id}_{[n]})$
constitutes a morphism of covered spaces
\[ (X, \mathscr{C}=(C_i)_{i\in[n]})
\to
 (|\Sd \Nrv(\mathscr{C})|, (\operatorname{bst}v_i)_{i\in[n]} ) .\]
\end{prop}

\begin{proof}%
	Recall that $\Psi = \alpha_{\Nrv(\mathscr{C})}\circ\Phi$, where
	$\alpha_{\Nrv(\mathscr{C})}$ is the isomorphism $|\Nrv(\mathscr{C})|\cong |\Sd \Nrv(\mathscr{C})|$
	and
	$\Phi \colon X \to |\Nrv(\mathscr{C})|, \, x \mapsto \sum_{i=0}^n \psi_i(x) \cdot |v_i|$.
	Note that if $x\in C_i$, then $\phi_i(x) = 1$ and thus $\psi_i(x)$ is maximal
	among the $\psi_j(x)$. Hence, by \cref{lemma:bst_maximal}
	we know that $\Psi(x)\in \operatorname{bst}(v_i)$ and the claim follows.
\end{proof}

\begin{prop}
\label{prop:gamma_carried}
The pair of maps $(\Gamma, \operatorname{id}_{[n]})$ constitutes a morphism of covered spaces
\[
 (|\Sd \Nrv(\mathscr{C})|, (\operatorname{bst}v_i)_{i\in[n]})\to (X, \mathscr{C}=(C_i)_{i\in[n]}) .
\]
\end{prop}
\begin{proof}
	By definition, the map $\Gamma$ sends the vertices of a geometric simplex $\sigma$ in $\operatorname{bst}v_i$ to~$C_i$.
	As the cover element $C_i$ is convex and $\Gamma$ is affine linear on $\sigma$, it follows that $\Gamma(\sigma)$ is also contained in $C_i$. This shows $\Gamma(\operatorname{bst}v_i)\subseteq C_i$, proving the claim.
\end{proof}

We will now show that $\Psi$ is a homotopy inverse to $\Gamma$, which implies $|\Nrv(\mathscr{C})|\cong|\Sd\Nrv(\mathscr{C})|\simeq X$.

\begin{proof}[Proof of \cref{thm:nerve_comp_conv}]
 It follows from \cref{prop:gamma_carried} and \cref{prop:psi_carried} that the pair of maps $(\Gamma\circ\Psi,\operatorname{id}_{[n]})$ constitutes a morphism of covered spaces. Hence, $\Gamma\circ\Psi$ is carried by the identity on $\mathscr{C}$ and thus it is homotopic to the identity $\operatorname{id}_X$ by a straight line homotopy:
for every $x \in C_i$, we have $\Gamma\circ\Psi(x) \in C_i$, and since the $C_i$ are convex, the line segment joining $x$ and $\Gamma\circ\Psi(x)$ lies in $C_i$.
Similarly, the pair of maps $(\Psi\circ\Gamma, \operatorname{id}_{[n]})$ constitutes a morphism of covered spaces. That the composition $\Psi\circ\Gamma$ is homotopic to $\operatorname{id}_{|\Sd\Nrv(\mathscr{C})|}$ now follows from \cref{lemma:int_contractible} and the following \cref{prop:hequ_bstar}.
\end{proof}

Recall that any two maps into a contractible space are homotopic (to a constant map). The following statement generalizes this fact to good covers, where contractibility is only guaranteed locally.
\begin{prop}
	\label{prop:hequ_bstar}
	Let $K$ be a finite simplicial complex and let $Y$ be a topological space. Assume we have two morphisms of covered spaces
	\[(f,\varphi),(g,\varphi)\colon (|K|,(|L_i|)_{i\in[n]}) \to (Y,(V_j)_{j\in J}),\]
	with the same map of index sets $\varphi \colon [n] \to J$, where $(|L_i|)_{i\in[n]}$ is a cover by subcomplexes and $(V_j)_{j\in J}$ is a good cover. Then $f$ is homotopic to $g$.
\end{prop}
\begin{proof}
	Let $I=[0,1]$ denote the unit interval. We inductively construct homotopies $H^m\colon |\sk{m}{K}|\times I \to Y$
between $f |_{|\sk{m}{K}|}$ and $g |_{|\sk{m}{K}|}$
such that $H^m$ is carried by the map of indexed covers
$\varphi \colon (|\sk{m}{L_i}| \times I)_{i\in[n]}  \to (V_{j})_{j\in J}$ induced by the given map of index sets $\varphi \colon [n] \to J$. If $m=\dim K$ is the dimension of the simplicial complex, the map $H=H^m$ is the desired homotopy between~$f$ and $g$.

	To establish the base case $m=0$, let $p$ be any vertex of $K$ and let $i_0,\dots,i_k\in [n]$ be those indices $i$ with $|p|\in |L_i|$.
	By the assumption that $f$ and $g$ are carried by $\varphi$, we know that both $f(|p|)$ and $g(|p|)$ are contained in
	$S:=\bigcap_{l=0}^k V_{\varphi(i_l)}$, which is contractible by assumption, and thus we can choose a path in $S$ that connects these two points. This defines the desired homotopy $H^0$.
	To see that the map $H^0$ is carried by $P^0$, let $(|p|,t)\in |p|\times I$ be a point. If $(|p|,t)\in |L_i|\times I$, then $i=i_l$ is one of the indices above. Thus, by construction, $H^0((|p|,t))\in S\subseteq V_{\varphi(i_l)}=V_{\varphi(i)}$, and the claim is proven.

	For the induction step from $(m-1)$ to $m$, let $H^{m-1}$ satisfy the induction hypothesis. Let $\sigma$ be an $m$-simplex in $\sk{m}{K}$. Furthermore, let $i_0,\dots,i_k\in [n]$ be those indices $i$ with $\sigma
	\in L_i$. By the induction hypothesis, we have
	\[H^{m-1}(|\partial\sigma|\times I)\subseteq W:= \bigcap_{l=0}^k V_{\varphi(i_l)}.\]
	By the assumption that $(V_j)$ is good, the space $W$ is contractible,
	and so we can extend the homotopy
	$H^{m-1}|_{|\partial\sigma|\times I}$
	to a homotopy $H^{m}|_{|\sigma|\times I}$
	from $f |_{|\sigma|}$ to $g |_{|\sigma|}$:

	\[
	\begin{tikzcd}[column sep=12ex,/tikz/column 1/.append style={anchor=base east}]
	(|\partial \sigma| \times I) \cup (|\sigma|\times\{0,1\})\cong
	\makebox[0pt][l]{$\subnode{Sp}{S^{m}}$}\phantom{{B}^{m+1}}
	& \subnode{W}{W}\subseteq Y\\
	\ |\sigma|\times I\cong \subnode{B}{B^{m+1}}\arrow[swap,dotted]{ur}{H^{m}|_{|\sigma|\times I}}&
	\end{tikzcd}
\begin{tikzpicture}[overlay, remember picture]
   \draw [{Hooks[width=+0pt 10.8,length=+0pt 3.6,harpoon,line cap=round,right]}->] (Sp) edge (B.north -| Sp.south);
   \draw [->] (Sp)
   edge node[auto, font=\everymath\expandafter{\the\everymath\scriptstyle}, inner sep=+0.5ex]
   {$(H^{m-1},(f,g))$}
   (W);
\end{tikzpicture}
	\]
	Because the $m$-simplex $\sigma$ was arbitrary, we can extend $H^{m-1}\colon |\sk{m-1}{K}|\times I \to Y$ to $H^{m}\colon |\sk{m}{K}|\times I \to Y$.

	By construction, this map is carried by $\varphi \colon (|\sk{m}{L_i}| \times I)_{i\in[n]}  \to (V_{j})_{j\in J}$.
	To see this, we verify that for any $i$, every point
	$(x,t)\in |\sk{m}{L_i}|\times I$
	is mapped to $H^m(x,t)\in V_{\varphi(i)}$. 
	By induction, this is true whenever $x \in |\sk{m-1}{L_i}|$, so it remains to show the claim for~$x$ in the interior of some $m$-simplex $\sigma \in L_i$. 
	Now $i=i_l$ is one of the indices above, and by construction of $H^{m}$, we have
	\[H^m(x,t)\in H^{m}(|\sigma|\times I)\subseteq W \subseteq V_{\varphi(i_l)}=V_{\varphi(i)},\] proving the claim.
\end{proof}
\subsection{A Functorial Nerve Theorem for Closed Convex Covers}
Now, we will discuss two ways of turning the result in \cref{thm:nerve_comp_conv} into a functorial nerve theorem. We start by giving a version that follows the strategy explained in \cref{section:func_nerve}. After this, we will give a version that is functorial on the nose but needs to use the concept of pointed covers.
\begin{thm}
    \label{thm:comp_conv_zigzag}
    If $X \subset \bbR^d$, and $\mathscr{C} = ( C_i )_{i\in[n]}$ is a cover of $X$ by closed convex subsets,
then the natural maps $\rho_S \colon \SmallBarCon(\mathscr{C}) \to X$ and
$\rho_N \colon \SmallBarCon(\mathscr{C}) \to |\Nrv(\mathscr{C})|$, introduced in \cref{section:func_nerve}, are homotopy equivalences.
\end{thm}
\begin{proof}
    As explained at the end of \cref{functorial_nerve_via_bar}, it suffices to consider the (not necessarily commutative) diagram
\[
\begin{tikzcd}%
    &\subnode{u}{\BarCon(\mathscr{D}_\mathscr{C})}%
    &\\
    \subnode{d1}{X}&& \subnode{d2}{|\Sd}\Nrv(\mathscr{C})\arrow{ll}{\Gamma}|,
\end{tikzcd}
\begin{tikzpicture}[overlay, remember picture]
  \draw [->] (u)
  edge node[auto, font=\everymath\expandafter{\the\everymath\scriptstyle}, inner sep=+0.5ex]
  {$\pi_{\Sd N}$}
  (d2);
  \draw [->] (u)
  edge node[auto, font=\everymath\expandafter{\the\everymath\scriptstyle}, inner sep=+0.5ex, swap]
  {$\pi_S$}
  (d1);
\end{tikzpicture}
\]
where $\Gamma$ is as in \cref{thm:nerve_comp_conv}, and show that $\pi_S$ and $\pi_{\Sd N}$ are homotopy equivalences.

    By \cref{prop:homotopy_lemma_top} and the fact that convex sets are contractible, we know that $\pi_{\Sd N}$ is a homotopy equivalence.
    Every point $p\in \BarCon(\mathscr{D}_\mathscr{C})$ can be described as a pair $p=(x,\alpha)$, where $\alpha$ is a point in $|\sigma|$, for some $\sigma=(J_n\subset \cdots\subset J_0)\in \Sd\Nrv(\mathscr{C})$, and $x\in C_{J_0}$. By construction, we have $\Gamma(\pi_{\Sd N}(p))=\Gamma(\alpha)\in C_{J_n}$ and $\pi_{S}(p)= x\in C_{J_0}\subseteq C_{J_n}.$ Therefore, a straight line homotopy shows that the maps $\Gamma\circ \pi_{\Sd N}\simeq \pi_{S}$ are homotopic. As $\Gamma$ and $\pi_{\Sd N}$ are homotopy equivalences the same is true for $\pi_{S}$.
\end{proof}

\subsection{A One-Arrow Functorial Nerve Theorem for Pointed Covers}

We will now describe the second way of obtaining a functorial nerve theorem. Recall from \cref{defi:pointed_covered} and the paragraphs afterwards the definition of $\Cov_*$ and its subcategory $\ClConv_*$. Before stating the functorial nerve theorem let us consider one more important example of a pointed covered space.
\begin{ex}
\label{ex:pointed_covered_simp_comp}
 Let $K$ be a simplicial complex. The cover $(\operatorname{bst}(v))_{v\in \vertx{K}}$ of $|K|$ by the closed barycentric stars is pointed by the vertices $(|w|)_{w\in \vertx{\Sd K}}$.
\end{ex}

Let $(X,\mathscr{A}_*)\in \ClConv_*$ be a pointed covered space. Recall that the construction of $\Gamma$ at the beginning of \cref{section:closed-convex} requires many choices. Those choices can be made such that $\Gamma$ is a morphism of pointed covered spaces, where the nerve is a pointed covered space as described in \cref{ex:pointed_covered_simp_comp}, and such that it is affine linear on each simplex of the barycentric subdivision of the nerve.

\begin{thm} \label{one-arrow-functorial-nerve-thm}
 The homotopy equivalence $\Gamma\colon|\Sd\Nrv(\mathscr{A})|\to X$
 is natural with respect to the morphisms in $\ClConv_*$.
\end{thm}
\begin{proof}
To show naturality, let $(f,\varphi)\colon (X,\mathscr{A}_*)\to (Y,\mathscr{C}_*)$ be a morphism in $\ClConv_*$. Then we need to prove that the diagram \[\begin{tikzcd}
                                                                                                                                                                                                                                                                                                                                                                                               X\arrow{r}{f} & Y\\
\ |\Sd \Nrv (\mathscr{A})|\arrow{u}\arrow{r}{|\Sd \varphi_*|} & \arrow{u} |\Sd\Nrv (\mathscr{C})|                                                                                                                                                                                                                                                                                                                                                                                                                                                                                                                                                                \end{tikzcd}
 \]
 commutes. Both compositions are maps $|\Sd\Nrv(\mathscr{A})|\to Y$ that are affine linear on each simplex of the barycentric subdivision. Hence, they are completely determined by their values on the vertices, where both compositions coincide by construction.
\end{proof}

\section{Nerve Theorems for Simplicial and Semi-Algebraic Covers}
\label{section:triangulated-covers}

One can prove a nerve theorem for simplicial complexes
as a corollary of Quillen's Theorem A for posets.
In this section, we use combinatorial arguments to prove a functorial version of this result.
Using a well-known triangulation theorem for semi-algebraic sets,
this functorial nerve theorem for simplicial complexes implies such a theorem
for finite, closed, semi-algebraic covers of compact semi-algebraic sets.
Finally, we use the same combinatorial methods to prove a functorial version of a nerve theorem of Bj\"orner.

\paragraph{Posets and Homotopy Theory}
We begin by reviewing some basic facts about posets,
following Quillen \cite{quillen-posets}.

Recall that the flag complex $\Flag(P)$ of a poset $P$ is the simplicial complex
whose vertices are the elements of $P$ and
whose $n$-simplices are the chains $x_0 < \dots < x_n$ of elements of $P$.
We will sometimes say that a poset has a certain topological property if its flag complex has that property.
For example, we say that a poset $P$ is contractible if $|\Flag(P)|$ is contractible,
and we say a map $f \colon P \to Q$ of posets is a homotopy equivalence if
the induced map $|\Flag(P)| \to |\Flag(Q)|$ is a homotopy equivalence.
If $P$ and $Q$ are posets, then there is a canonical homeomorphism
\begin{equation} \label{flag-and-products}
	|\Flag(P \times Q)| \xrightarrow{\iso} |\Flag(P)| \times |\Flag(Q)|
\end{equation}
induced by the projection maps.
As explained in \cite{quillen-posets},
the product must be taken in the category of compactly generated spaces,
\cref{compactly-generated}.
However, if one of $P$ or $Q$ is finite, then this agrees with the usual product.
It follows that if $f,g \colon P \to Q$ are maps of posets such that $f(x) \leq g(x)$ for all $x \in P$,
then $|\Flag(f)|, |\Flag(g)| \colon |\Flag(P)| \to |\Flag(Q)|$ are homotopic.
To see this, observe that the relation $f \leq g$ determines a map of posets $H: P \times \{0 < 1\} \to Q$,
and $|\Flag(\{0 < 1\})|$ is an interval.

The main result about posets that we need is Quillen's Theorem A \cite{quillen-K-theory}.
Given a map $f \colon P \to Q$ of posets and $y \in Q$, define the subposet of $P$:
$f/y = \{ x \in P \mid f(x) \leq y \}$.

\begin{thm}[Quillen's Theorem A]
If $f \colon P \to Q$ is a map of posets, and $f/y$ is contractible for all $y \in Q$,
then $f$ is a homotopy equivalence.
\end{thm}

It should be said that Quillen's theorem is more general than this result,
but this is what we will use. For a nice proof at this level of generality, see \cite{walker} or \cite{barmak}, where it is shown that for finite posets the map $f$ is even a simple homotopy equivalence.
We now use Quillen's Theorem A to give a simple proof of the nerve theorem
for covers of a simplicial complex by subcomplexes; 
this is similar to \cite[Lemma 1.1]{MR607041}, 
\cite[Theorem 6]{bjoerner}, 
and \cite[Theorem 4.3]{barmak}.

\begin{prop} \label{nerve-tri-cov}
Let $K$ be a simplicial complex and let $\mathscr{A} = (K_i \subseteq K)_{i \in I}$
be a locally finite good cover of $K$ by subcomplexes.
Then $K$ is homotopy equivalent to $\Nrv(\mathscr{A})$.
\end{prop}

\begin{proof}
Define a map of posets $f \colon \Pos(K) \to \Pos(\Nrv(\mathscr{A}))^{\opp}$ by
\[
	f(\sigma) = \{ i \in I  \mid  \sigma \in K_i \} \; .
\]
As $\mathscr{A}$ is locally finite, $f(\sigma)$ is finite for all $\sigma \in K$. 
We will show that $f$ is a homotopy equivalence.
As usual, for $J \subseteq I$, we write $K_J = \cap_{i \in J} K_i$.
By Quillen's Theorem A, it suffices to show that,
for all elements $J$ of $\Pos(\Nrv(\mathscr{A}))^{\opp}$,
the poset $f/J$ is contractible.
Unwinding the definition,
$f/J$ is the subposet of $\Pos(K)$ with elements $\sigma \in \Pos(K)$ such that $J \subseteq f(\sigma)$.
By definition, $J \subseteq f(\sigma)$ if and only if $\sigma \in K_J$.
So, $f/J = \Pos(K_J)$.
As the intersection $K_J$ is nonempty, it is contractible by assumption,
and so $|\Flag(\Pos(K_J))| \iso |K_J|$ is contractible.
Thus, $f$ is a homotopy equivalence.
The homotopy equivalence of the proposition is the composition:
\[
	|K| \iso |\Sd(K)| \xrightarrow{f_*}
	|\Sd(\Nrv(\mathscr{A}))| \iso |\Nrv(\mathscr{A})| \; . \qedhere
\]
\end{proof}

\subsection{A Functorial Nerve Theorem for Simplicial Covers}

In order to prove a functorial version of \cref{nerve-tri-cov},
we now introduce a poset $\PoBar$ that is intermediate between
the posets $\Pos(K)$ and $\Pos(\Nrv(\mathscr{A}))^{\opp}$
that appeared in the proof.
We use the notation $\PoBar$ because this construction can be seen as a bar construction
taken in the category of posets, as we explain in \cref{pobar-is-a-bar}. 
An additional benefit of using this intermediate object is that 
it allows one to remove the assumption that the cover is locally finite. 
This is similar to the strategy of Bj\"{o}rner \cite[Lemma 1.1]{MR607041}, 
which he attributes to Quillen.

\begin{defi}
If $K$ is a simplicial complex and $\mathscr{A} = (K_i \subseteq K)_{i \in I}$
is a cover of $K$ by subcomplexes,
let $\PoBar(\mathscr{A})$ be the poset with the underlying set
\[
	\PoBar(\mathscr{A}) =
	\left\lbrace (\sigma, J)  \mid  J \subseteq I \text{ finite, } \sigma \in K_J\right\rbrace
\]
where
$(\sigma, J) \leq (\sigma', J')$ if and only if $\sigma \subseteq \sigma'$ and $J \supseteq J'$.
\end{defi}

Since $\PoBar(\mathscr{A})$ is a subposet of the product
$\Pos(K)\times \Pos(\Nrv(\mathscr{A}))^{\opp} $, it comes with projection maps
$\lambda_S \colon \PoBar(\mathscr{A}) \to \Pos(K)$ and
$\lambda_N \colon \PoBar(\mathscr{A}) \to \Pos(\Nrv(\mathscr{A}))^{\opp}$.
In the next lemma, $f$ denotes the poset map defined in the proof of \cref{nerve-tri-cov}.

\begin{prop} \label{proj-are-equiv}
Let $K$ be a simplicial complex and let $\mathscr{A} = (K_i \subseteq K)_{i \in I}$
be a cover of $K$ by subcomplexes.
Then, the map $\lambda_S$ is a homotopy equivalence and if $\mathscr{A}$ is good, then $\lambda_N$ is a homotopy equivalence, as well.
Moreover, if $\mathscr{A}$ is locally finite,
the map $f$ is defined and the following diagram of posets 
commutes up to homotopy after taking flag complexes:
\[
\begin{tikzcd}
& & \PoBar(\mathscr{A}) \arrow[dl, "\lambda_{S}"'] \arrow[dr, "\lambda_{N}"] & \\
& \Pos(K) \arrow[rr, "f"'] & & \Pos(\Nrv(\mathscr{A}))^{\opp}
\end{tikzcd}
\]
\end{prop}

\begin{proof}
We begin by showing that $\lambda_{S}$ is a homotopy equivalence.
Since a map of posets $P\to Q$ is a homotopy equivalence if and only if the induced map on opposite posets $P^{\opp}\to Q^{\opp}$ is a homotopy equivalence, it suffices, by Quillen's Theorem A, to show that for any $\sigma\in \Pos(K)$ the subposet $\sigma\textbackslash \lambda_{S}=\{(\tau,J) \mid \sigma \subseteq \lambda_{S}(\tau,J)= \tau\}\subseteq \PoBar(\mathscr{A})$ is contractible:
Consider the fiber $\lambda_{S}^{-1}(\sigma) = \{ (\sigma, J) \mid \sigma \in K_J\}\subseteq \PoBar(\mathscr{A})$,
and define the poset map $\mu \colon \sigma\textbackslash \lambda_{S} \to \lambda_{S}^{-1}(\sigma)$
by $\mu(\tau, J) = (\sigma, J)$.
The map $\mu$ is a homotopy inverse to the inclusion of
$\lambda_{S}^{-1}(\sigma)$ into $\sigma\textbackslash \lambda_{S}$,
as, for any $(\tau, J) \in \sigma\textbackslash \lambda_{S}$ we have the relation
$\mu(\tau, J) \leq (\tau, J)$ in $\sigma\textbackslash\lambda_{S}$.
The fiber $\lambda_{S}^{-1}(\sigma)$ is contractible,
as it is isomorphic to the opposite face poset of a full simplicial complex.
We conclude that $\sigma \textbackslash \lambda_{S}$ is contractible and thus $\lambda_{S}$ is a homotopy equivalence.

Now, assume that $\mathscr{A}$ is good. We show that $\lambda_{N}$ is a homotopy equivalence,
using Quillen's Theorem A. So, we take $J \in \Pos(\Nrv(\mathscr{A}))^{\opp}$, 
and we must check that $\lambda_{N} / J$ is contractible. 
Consider the fiber $\lambda_{N}^{-1}(J) = \{ (\sigma, J) \mid \sigma \in K_J\}\subseteq \PoBar(\mathscr{A})$,
and define the poset map $\nu \colon \lambda_{N} / J \to \lambda_{N}^{-1}(J)$ 
by $\nu(\sigma, \tilde{J}) = (\sigma, J)$. 
The map $\nu$ is a homotopy inverse to the inclusion of 
$\lambda_{N}^{-1}(J)$ into $\lambda_{N} / J$, 
as, for any $(\sigma, \tilde{J}) \in \lambda_{N} / J$ we have the relation 
$\nu(\sigma, \tilde{J}) \geq (\sigma, \tilde{J})$ in $\lambda_{N} / J$. 
The fiber $\lambda_{N}^{-1}(J)$ is contractible, 
as it is isomorphic to $\Pos(K_J)$, 
which is contractible as $\mathscr{A}$ is good. 
We conclude that $\lambda_{N} / J$ is contractible and thus Quillen's Theorem A implies that
$\lambda_{N}$ is a homotopy equivalence.

Finally, assume $\mathscr{A}$ is locally finite, so that $f$ is defined. 
If $(\sigma,J)$ is in $\PoBar(\mathscr{A})$,
then $\lambda_N(\sigma,J) = J \subseteq f(\sigma) = (f \circ \lambda_{S})(\sigma,J)$.
So, we have $(f \circ \lambda_{S})(\sigma,J) \leq \lambda_N(\sigma,J)$,
which implies that
$|\Flag(f)| \circ |\Flag(\lambda_{S})|$ and $|\Flag(\lambda_{N})|$ are homotopic.
\end{proof}

The strategy now is to use what we have proved about the $\PoBar$ construction
to show that the natural maps from the blowup complex to $|K|$ and $|\Nrv(\mathscr{A})|$
(defined in \cref{section:func_nerve}) are homotopy equivalences.

To do this, we will identify a subcomplex
$\Tri(\mathscr{A}) \subseteq \Flag(\PoBar(\mathscr{A}))$
that is homeomorphic to the blowup complex after realization;
we then show, using discrete Morse theory,
that the inclusion $\vert\Tri(\mathscr{A})\vert\hookrightarrow\vert\Flag(\PoBar(\mathscr{A}))\vert$ is a homotopy equivalence.
In particular, it follows that the blowup complex is homotopy equivalent to
$\vert\Flag(\PoBar(\mathscr{A}))\vert$.

\begin{defi}
Let $K$ be a simplicial complex,
and let $\mathscr{A} = (K_i \subseteq K)_{i \in I}$ be a cover of $K$ by subcomplexes.
Let $\Tri(\mathscr{A})$ be the subcomplex of $\Flag(\PoBar(\mathscr{A}))$
consisting of the simplices
$(\sigma_0,J_0) < \dots < (\sigma_m,J_m)$ such that $\sigma_m \in K_{J_0}$.
\end{defi}

The letter $\Tri$ stands for ``triangulation'',
since $\Tri(\mathscr{A})$ turns out to be a triangulation of the blowup complex.
\goodbreak

\begin{lemma} \label{triangulation-of-blowup}
Let $K$ be a simplicial complex,
and let $\mathscr{A} = (K_i \subseteq K)_{i \in I}$ be a cover of $K$ by subcomplexes;
write $|\mathscr{A}|=(|K_i| \subseteq |K|)_{i \in I}$.
There is a homeomorphism $\varphi \colon \SmallBarCon(|\mathscr{A}|) \to |\Tri(\mathscr{A})|$
such that the following diagram commutes:
\begin{equation} \label{pobar-and-blowup}
\begin{tikzcd}
	& \vert \Flag(\PoBar(\mathscr{A})) \vert
	\arrow[rd, "\lambda_{N}"] \arrow[ld, "\lambda_{S}"']\\
	\vert \Sd(K) \vert & \vert \Tri(\mathscr{A}) \vert
	\arrow[r] \arrow[l] \arrow[u, hook]
	& \vert \Sd(\Nrv (\mathscr{A})) \vert \\
	\vert K \vert \arrow[u, "\iso"] & \SmallBarCon(|\mathscr{A}|)
	\arrow[r, "\rho_{N}"'] \arrow[u, "\varphi"] \arrow[l, "\rho_{S}"]
	& \vert \Nrv (\mathscr{A}) \vert \arrow[u, "\iso"'] 
\end{tikzcd}
\end{equation}
Here, the vertical maps on the left and right are the standard homeomorphisms.
\end{lemma}

\begin{proof}
The blowup complex $\SmallBarCon(|\mathscr{A}|)$ is defined by glueing together pieces
of the form $|K_J|\times |J| $ for $J \in \Nrv(\mathscr{A})$.
We abuse notation and write $J$ also for the full simplicial complex on $J$.
For any such piece, define $\varphi$ by the composition
\begin{align*}
	 |K_J| \times |J|  &\iso  |\Sd K_J|\times |\Sd J|\\
		&=  |\Flag(\Pos(K_J))|\times |\Flag(\Pos(J))|  \\
		&\iso |\Flag(\Pos(K_J))|\times |\Flag(\Pos(J)^{\opp})|  \\
		&\iso |\Flag(\Pos(K_J)\times \Pos(J)^{\opp} )| \subseteq |\Tri(\mathscr{A})|
\end{align*}
where the last homeomorphism is an instance of \ref{flag-and-products}.
As these maps respect the equivalence relation from the definition of the blowup complex,
together they define a continuous map
$\varphi \colon \SmallBarCon(|\mathscr{A}|) \to |\Tri(\mathscr{A})|$.
By construction, the diagram \ref{pobar-and-blowup} commutes.

To see that $\varphi$ is a homeomorphism, we can construct its inverse.
As $J$ varies, the subcomplexes $\Flag(\Pos(K_J)\times \Pos(J)^{\opp} )$
cover $\Tri(\mathscr{A})$.
For each $J$, we can reverse the homeomorphisms in the definition of $\varphi$
to define $\varphi^{-1}$ on $|\Flag(\Pos(K_J)\times \Pos(J)^{\opp} )|$.
Since these maps agree on intersections, they glue together to define the inverse
$\varphi^{-1} \colon |\Tri(\mathscr{A})| \to \SmallBarCon(|\mathscr{A}|)$.
\end{proof}

\begin{lemma} \label{collapse-to-T}
Let $K$ be a simplicial complex,
and let $\mathscr{A} = (K_i \subseteq K)_{i \in I}$ be a cover of~$K$ by subcomplexes.
Then the inclusion $\vert \Tri(\mathscr{A})\vert\hookrightarrow\vert\Flag(\PoBar(\mathscr{A}))\vert$ is a homotopy equivalence.
\end{lemma}

\begin{proof}

We construct a discrete gradient vector field $V$ on $\Flag(\PoBar(\mathscr{A}))$ such that the set of critical simplices is $\Tri(\mathscr{A})$. Then it follows from \cref{infinite_discrete_morse} that the inclusion $|\Tri(\mathscr{A})| \hookrightarrow |\Flag(\PoBar(\mathscr{A}))|$ is a homotopy equivalence.

To this end, let $L=\Flag(\PoBar(\mathscr{A}))\setminus \Tri(\mathscr{A})$ and consider the function $f\colon L\to\mathbb{N}\cup\{\infty\}$
that assigns to a simplex $\tau=((\sigma_0,J_0)< \cdots < (\sigma_m,J_m))$ the value
\[f(\tau)=\min\{i\in\{0,\dots,m-1\}\mid \sigma_i<\sigma_{i+1},\ J_i<J_{i+1} \}\]
with $f(\tau)=\infty$ if no such $i$ exists. Moreover, we consider the function $g\colon L\to\mathbb{N}\cup\{\infty\}$
that assigns to a simplex $\tau=((\sigma_0,J_0)< \cdots < (\sigma_m,J_m))$ the value
\[g(\tau)=\min\{i\in\{0,\dots,m-2\}\mid \sigma_i=\sigma_{i+1}<\sigma_{i+2},\ J_{i}<J_{i+1}=J_{i+2} \}\]
with $g(\tau)=\infty$ if no such $i$ exists. As $\tau\in L$ we have $\sigma_m\notin K_{J_0}$, by definition of $\Tri(\mathscr{A})$, and hence we get that if $f(\tau)=\infty$ then $g(\tau)<\infty$, as otherwise $\sigma_m\in K_{J_0}$. This and the definitions of the function $f$ and $g$ imply that either $f(\tau)<g(\tau)$ or $f(\tau)>g(\tau)+1$.

We now define the discrete vector field $V$ on $\Flag(\PoBar(\mathscr{A}))$ that partitions $L$ into pairs of simplices and where we let every simplex in $\Tri(\mathscr{A})$ be critical. Take any simplex~$\tau$ in $L=\Flag(\PoBar(\mathscr{A}))\setminus \Tri(\mathscr{A})$, i.e., a chain $(\sigma_0,J_0)< \cdots < (\sigma_m,J_m)$ in $\PoBar(\mathscr{A})$ such that $\sigma_m\notin K_{J_0}$. If $i=f(\tau)<g(\tau)$, consider the chain
\[
(\sigma_0,J_0)< \cdots< (\sigma_{i},J_{i})< (\sigma_{i},J_{i+1})< (\sigma_{i+1},J_{i+1})< \cdots < (\sigma_m,J_m)
\]
 and pair the corresponding simplex $\mu$ in $L$ with $\tau$; note that $f(\mu)>f(\tau)+1=g(\mu)+1$.
 We verify that $V$ is a discrete vector field that partitions $L$: For any simplex $\mu=((\tilde{\sigma}_0,\tilde{J}_0)< \cdots < (\tilde{\sigma}_m,\tilde{J}_m))\in L$ with $f(\mu)>g(\mu)+1=j$ consider the facet $\tau$ of $\mu$ that skips the element $(\tilde{\sigma}_{j},\tilde{J}_{j})$; note that $f(\tau)=g(\mu)<g(\tau)$. It is straightforward to see that the sets $\{\tau\in L\mid f(\tau)<g(\tau)\}$ and $\{\tau\in L\mid f(\tau)>g(\tau)+1\}$ partition~$L$ and that the above constructions yield mutually inverse bijections between those sets, implying that $V$ is a discrete vector field that partitions~$L$ into pairs of simplices.

We prove that $V$ is a discrete gradient vector field by showing that there are no non-trivial closed $V$-paths: Consider any $V$-path $\tau_0\to\mu_0\leftarrow\cdots\to\mu_r\leftarrow\tau_{r+1}$                                                                                                                                                                                                                                                                                with $\{\tau_i,\mu_i\}\in V$ and $\tau_i\neq \tau_{i+1}$ for all $i$. To show that it is not closed, i.e., $\tau_{r+1}\neq\tau_0$, consider first any chain $((\sigma,J)<(\tilde{\sigma},\tilde{J}))\in L$ of length $2$ and the set
\[S=\{\tau\in L\mid \min \tau = (\sigma,J),\ \max\tau= (\tilde{\sigma},\tilde{J})\}.\]
Note that $V$ restricts to a partition of $S$. If $R \neq S$ is another such set of chains with $\tau_0\in S$ and $\tau_i\in R$ for some $i$, then $\tau_j\notin S$ for all $j\geq i$ as at least one of the inequalities $\min\tau_0\leq \min\tau_j$ and $\max\tau_j\leq\max\tau_0$ is strict.
Therefore the $V$-path cannot be closed. Moreover, as $S$ is finite, there are only finitely many possible other such~$R$. Therefore, it is enough to show that for any such $S$ there is no non-trivial closed $V$-path with $\tau_i\in S$ for all $i$.

To this end, we construct a lexicographic partial order on the set of chains $S$. First, we consider the product inclusion order on pairs of simplices of $K$ and subsets of $I$,
given by $(\sigma,J)\subseteq(\tilde{\sigma},\tilde{J})$ if and only if $\sigma \subseteq \tilde{\sigma}$ and $J \subseteq \tilde{J}$. Now, we extend this partial order to a partial order on $S$: For any two chains $\tau=((\sigma_0,J_0)< \cdots < (\sigma_m,J_m))$ and $\mu=((\tilde{\sigma}_0,\tilde{J}_0)< \cdots < (\tilde{\sigma}_m,\tilde{J}_m))$ in $S$ of equal length, we let $\tau\leq_{\op{lex}} \mu$ if $\tau=\mu$ or if for the smallest index $j$ with $(\sigma_j,J_j)\neq (\tilde{\sigma}_j,\tilde{J}_j)$ we have~$(\sigma_j,J_j)\subseteq (\tilde{\sigma}_j,\tilde{J}_j)$.
We show that for any two gradient pairs $(\tau,\mu),(\tilde{\tau},\tilde{\mu})\in V$ with $\tilde{\tau}$ a facet of $\mu$, we have $\mu>_{\op{lex}}\tilde{\mu}$, proving that the $V$-path above cannot be closed. We have
\[\mu = (\sigma_0,J_0)< \cdots< (\sigma_{i},J_{i})< (\sigma_{i},J_{i+1})< (\sigma_{i+1},J_{i+1})< \cdots < (\sigma_m,J_m)\]
with $i=f(\tau)$, and $\tilde{\tau}$ is a facet of $\mu$ that skips some element $(\sigma_j,J_j)$ with~$0<j<m$. Note that $j$ cannot be greater than $i+1$, as otherwise $g(\tilde{\tau})=i<f(\tilde{\tau})$, contradicting the assumption that $\tilde{\tau}$ is a gradient facet of $\tilde{\mu}$ and therefore $f(\tilde{\tau})<g(\tilde{\tau})$.
Moreover, if $j\leq i$ then $f(\tilde{\tau})=j-1$, and if $j=i+1$ then $f(\tilde{\tau})=j$. 
In any case, $\tilde{\mu}$~is obtained by adding the element $(\sigma_{j-1},J_{j+1})$ to~$\tilde{\tau}$.
Now $\sigma_{j-1} \subseteq \sigma_j$ and $J_j \supseteq J_{j+1}$, and thus
$(\sigma_{j-1},J_{j+1})\subseteq (\sigma_j,J_j)$ and $\tilde{\mu}\leq_{\op{lex}}\mu$. As $\tilde{\tau}\neq\tau$, we must have $\tilde{\mu}\neq \mu$ and thus $\tilde{\mu}<_{\op{lex}}\mu$.

The reasoning above also implies that for every simplex in $\Flag(\PoBar(\mathscr{A}))$ its $V$-path height is finite and hence it follows from \cref{lemma:equivalent_height_def} and \cref{infinite_discrete_morse} that the inclusion $|\Tri(\mathscr{A})| \hookrightarrow |\Flag(\PoBar(\mathscr{A}))|$ is a homotopy equivalence.
\end{proof}

\begin{thm} \label{functorial-simplicial-nerve}
Let $K$ be a simplicial complex and let $\mathscr{A} = (K_i \subseteq K)_{i \in I}$
be a good cover of $K$ by subcomplexes.
Then, the natural maps $\rho_S \colon \SmallBarCon(|\mathscr{A}|) \to |K|$
and $\rho_N \colon \SmallBarCon(|\mathscr{A}|) \to |\Nrv(\mathscr{A})|$ are homotopy equivalences.
\end{thm}

\begin{proof}
Consider Diagram \ref{pobar-and-blowup}.
By \cref{proj-are-equiv} the maps $\lambda_S$ and $\lambda_N$ are homotopy equivalences.
By \cref{collapse-to-T}, the inclusion
$|\Tri(\mathscr{A})| \hookrightarrow |\Flag(\PoBar(\mathscr{A}))|$
is a homotopy equivalence, and by \cref{triangulation-of-blowup},
$\varphi$ is a homeomorphism.
It follows that $\rho_S$ and $\rho_N$ are homotopy equivalences.
\end{proof}

\subsection{A Functorial Nerve Theorem for Semi-Algebraic Covers}

As a corollary of \cref{functorial-simplicial-nerve},
we get a functorial nerve theorem for finite, closed, semi-algebraic covers of compact semi-algebraic sets.
For this, we need a well known theorem on the existence of triangulations of semi-algebraic sets
\cite[Theorem 9.2.1]{BCR}, which we now state.
For $K$ a simplicial complex and $\sigma$ a simplex of $K$,
we write $\op{int}| \sigma | = |\sigma| \setminus |\partial \sigma| \subset |K|$ for the open simplex.

\begin{lemma} \label{triangulation}
Let $S \subset \bbR^n$ be a compact semi-algebraic set,
and let $(S_i)_{i=0}^q$ be a finite family of semi-algebraic subsets of $S$.
There is a finite simplicial complex $K = \{\sigma_j\}_{j=0}^p$
and a homeomorphism $h \colon |K| \to S$,
such that every $S_i$ is the union of some images of simplices $h(\op{int} | \sigma_j |)$.
\end{lemma}

\begin{thm}
\label{compact-semialgebraic}
Let $S \subset \bbR^n$ be a compact semi-algebraic set,
and let $\mathscr{A} = (S_i)_{i=0}^{q}$ be a finite good cover of $S$
such that each $S_i$ is semi-algebraic and closed in $S$.
Then, the natural maps $\rho_S \colon \SmallBarCon(\mathscr{A}) \to S$
and $\rho_N \colon \SmallBarCon(\mathscr{A}) \to |\Nrv(\mathscr{A})|$ are homotopy equivalences.
\end{thm}

\begin{proof}
By \cref{triangulation}, there is a simplicial complex $K$,
a homeomorphism $h \colon |K| \to S$, and a cover of $K$ by subcomplexes
$\mathscr{B} = (K_i \subseteq K)_{i=0}^{q}$ such that
$h|_{K_i}$ is a homeomorphism between $K_i$ and $S_i$.
Then, $h$ induces a homeomorphism $\SmallBarCon(|\mathscr{B}|) \to \SmallBarCon(\mathscr{A})$
such that the following diagram commutes:
\[
\begin{tikzcd}
	\vert K \vert \arrow[d, "h"'] & \SmallBarCon(|\mathscr{B}|)
	\arrow[r, "\rho_{N}"] \arrow[d] \arrow[l, "\rho_{S}"']
	& \vert \Nrv (\mathscr{B}) \vert \arrow[d, equal] \\
	S & \SmallBarCon(\mathscr{A})
	\arrow[r, "\rho_{N}"'] \arrow[l, "\rho_{S}"]
	& \vert \Nrv (\mathscr{A}) \vert
\end{tikzcd}
\]
By \cref{functorial-simplicial-nerve} the top horizontal maps are homotopy equivalences,
and the corollary follows.
\end{proof}

\subsection{A Functorial Version of Bj\"{o}rner's Nerve Theorem}
\label{bjoerners-nerve-theorem}

If $K$ is a simplicial complex, and $\mathscr{A}$ is a locally finite cover of $K$ by subcomplexes,
then we have a comparison map $|K| \to |\Nrv(\mathscr{A})|$ induced by the
map of posets $f \colon \Pos(K) \to \Pos(\Nrv(\mathscr{A}))^{\opp}$ defined in the proof of \cref{nerve-tri-cov}.
In \cite{bjoerner}, Bj\"orner gives a detailed analysis of how the connectivity of this map
is affected by the connectivity of the finite intersections of cover elements.
For the final result of this section, we will use the $\PoBar$ construction and the blowup complex
to prove a functorial version of Bj\"orner's theorem.

\begin{defi}
Let $k \geq 0$. A topological space $X$ is $k$-\emph{connected}
if, for every $0 \leq r \leq k$, every map of the $r$-sphere into $X$ is homotopic to a constant map.
\end{defi}

\begin{prop}[Bj\"orner {\cite[Theorem 6]{bjoerner}}] \label{connectivity-nerve}
Let $K$ be a simplicial complex,
let $\mathscr{A}$ be a locally finite cover of $K$ by subcomplexes,
and let $k \geq 0$.
Assume that every non-empty intersection $K_{i_1} \cap \dots \cap K_{i_t}$
is $(k-t+1)$-connected, for all $t \geq 1$.
Then $f$ induces a bijection
\[\pi_0(K) \iso \pi_0(\Nrv(\mathscr{A})),\]
and for all $1 \leq j \leq k$, and for all $x \in |K|$,
$f$ induces an isomorphism
\[\pi_j (|K|, x) \iso \pi_j(|\Nrv(\mathscr{A})|, f_* (x)).\]
\end{prop}

In fact, the theorem of Bj\"orner deals with regular CW complexes, rather than simplicial complexes.
The assumption that $\mathscr{A}$ is locally finite is omitted from the original statement,
but it is used in the proof.
For convenience, Bj\"orner assumes that $K$ is connected: \cref{connectivity-nerve}
follows from Bj\"orner's theorem and the following lemma, which is easily proved.

\begin{lemma}
Let $K$ be a simplicial complex,
and let $\mathscr{A} = (K_i)_{i \in I}$ be a locally finite cover of $K$ by subcomplexes
such that $K_i$ is non-empty and connected for all $i \in I$.
Then~$f$ induces a bijection $\pi_0(K) \iso \pi_0(\Nrv(\mathscr{A}))$.\qed
\end{lemma}

Using the $\PoBar$ construction and the blowup complex as before,
we obtain the following functorial version of Bj\"orner's theorem.

\begin{thm}
\label{functorial-bjorner-nconnectivity}
Let $k \geq 0$, let $K$ be a simplicial complex, and let $\mathscr{A} = ( K_i )_{i \in I}$
be a locally finite cover of $K$ by subcomplexes.
Assume that every non-empty intersection $K_{i_1} \cap \dots \cap K_{i_t}$
is $(k-t+1)$-connected, for all $t \geq 1$.
The natural map $\rho_S \colon \SmallBarCon(|\mathscr{A}|) \to |K|$ is a homotopy equivalence,
and the natural map $\rho_N \colon \SmallBarCon(|\mathscr{A}|) \to |\Nrv(\mathscr{A})|$
induces a bijection in path components,
and for all $1 \leq j \leq k$, and for all $x \in \SmallBarCon(|\mathscr{A}|)$,
$\rho_N$ induces an isomorphism
$\pi_j (\SmallBarCon(|\mathscr{A}|), x) \iso \pi_j(|\Nrv(\mathscr{A})|, \rho_N (x))$.
\end{thm}

\begin{proof}
Note that the proof of \cref{proj-are-equiv}
shows that the poset map $\lambda_S$ is a homotopy equivalence,
and that the triangle commutes up to homotopy, without the assumption that $\mathscr{A}$ is good.
So, by Bj\"orner's \cref{connectivity-nerve},
$\lambda_N \colon |\Flag(\PoBar(\mathscr{A}))| \to | \Sd(\Nrv(\mathscr{A})) |$
induces a bijection in path components,
and for all $1 \leq j \leq k$, and for all $x \in |\Flag(\PoBar(\mathscr{A}))|$,
$\lambda_N$ induces an isomorphism
$\pi_j (|\Flag(\PoBar(\mathscr{A}))|, x) \iso \pi_j(|\Nrv(\mathscr{A})|, \lambda_N (x))$.
The result follows from commutativity of Diagram \ref{pobar-and-blowup}.
\end{proof}

\section{A Unified Nerve Theorem} \label{section:unified-nerve-thm}

We now prove the unified nerve theorem
(\cref{intro-thm-unified} in the introduction),
which subsumes \cref{thm:comp_conv_zigzag,compact-semialgebraic} as special cases, and which implies \cref{functorial-simplicial-nerve} with the additional assumption that the cover by subcomplexes is locally finite and locally finite dimensional.
Like in \cref{thm:comp_conv_zigzag},
we use the connection between the blowup complex and the bar construction
(explained in \cref{section:func_nerve})
to deduce statements about the blowup complex from the corresponding statements
about the bar construction.
Since the bar construction is a standard tool in homotopy theory,
we can use well-known results to prove the requisite properties.
We begin by introducing the various notions from topology we need
to state the unified nerve theorem.

In order to avoid pathological behavior in the category $\Top$ of all topological spaces,
results in algebraic topology are often restricted to certain full subcategories
that include all the spaces of primary interest (such as CW complexes)
and that have better categorical properties.
For example, it is often convenient to work in a category of topological spaces
that is \emph{cartesian closed}:
roughly speaking, this means that for any spaces $X$ and $Y$ in the category,
we have a ``mapping space'' $Y^X$ in the category such that for a fixed space $Z$, the set of maps
$Z \to Y^X$ is in bijection with the set of maps $X \times Z \to Y$,
and this bijection is natural in~$Y$ and $Z$.
Letting $Z = *$, we see that the points of $Y^X$ are in bijection with
continuous maps $X \to Y$.
Such mapping spaces play an important role in algebraic topology,
because they encode homotopy-theoretic information.
For example, a path $\gamma \colon [0,1] \to Y^X$ in the mapping space
corresponds to a homotopy $H \colon X \times [0,1] \to Y$.
There is more than one standard choice for a cartesian closed subcategory.
We will consider the following one.

\begin{defi} \label{compactly-generated}
A topological space $X$ is \emph{weak Hausdorff} if $g(K)$ is closed in $X$
for every continuous map $g \colon K \to X$ with $K$ compact Hausdorff.
A subspace $A$ of $X$ is \emph{compactly closed} if $g^{-1}(A)$ is closed in $K$
for every continuous map $g \colon K \to X$ with $K$ compact Hausdorff.
A space $X$ is a $k$-\emph{space} if every compactly closed subspace of $X$ is closed.
A space $X$ is \emph{compactly generated} if it is a weak Hausdorff $k$-space.
The full subcategory of $\Top$ of compactly generated spaces is denoted by $\CGSpc$.
\end{defi}

A note of warning: there is conflicting terminology in the literature surrounding compactly generated spaces.
See \cite[Chapter 5]{may-concise} or \cite{strickland} for basic facts about these spaces. For example, there exist inclusions and adjoint functors
\[\begin{tikzcd}
   \CGSpc\arrow[hookrightarrow,bend right,swap]{r}\arrow[phantom, "\dashv" rotate=-90]{r} & \op{k-spaces} \arrow[bend right]{l}  \arrow[phantom, near start, "\dashv" rotate=90]{r} \arrow[hookrightarrow, bend right, swap]{r} & \Top \arrow[bend right]{l},
  \end{tikzcd} \]
where $\op{k-spaces}$ is the full subcategory of $\Top$ consisting of $k$-spaces.

\begin{ex}
\label{ex:compactly_generated}
Many spaces are compactly generated:
 \begin{itemize}
   \item Every closed subspace of a compactly generated space is compactly generated.
   \item Every CW-complex is compactly generated.
  \item Every locally compact Hausdorff space is compactly generated \cite[Proposition 1.7]{strickland}. In particular, $\mathbb{R}^d$ is compactly generated.
 \end{itemize}
\end{ex}

A simplicial complex $K$ is sometimes said to be \emph{locally finite dimensional}
if every vertex $v$ of $K$ has a finite dimensional star, i.e.,
$\sup \{ \mathrm{dim} \, \sigma \mid v \in \sigma \} < \infty$.
Following this usage, we say that a cover of a topological space is locally finite dimensional
if the nerve of the cover is so. More explicitly, we have the following:

\begin{defi}
If $X$ is a topological space, and $\mathscr{A} = (A_i)_{i\in I}$ is a cover,
then $\mathscr{A}$ is \emph{locally finite dimensional} if for each cover element $A_i$
there exists $k_i \in \mathbb{N}$ such that for any
$J \subseteq I$ with ${A}_J \neq \emptyset$ and $i \in J$,
we have $|J| \leq k_i$.
\end{defi}

\begin{defi}
Let $R$ be a commutative ring.
We say that a continuous map $f$ between topological spaces is an \emph{$R$-homology isomorphism} if $H_n(f, R)$ is an isomorphisms for all $n \geq 0$.
We say that a cover $\mathscr{A} = (A_i)_{i \in I}$ is \textit{homologically good}
with respect to $R$ if, for all non-empty $J \subseteq I$ such that ${A}_{J} \neq \emptyset$,
the map to the one point space ${A}_{J} \to *$ is an $R$-homology isomorphism.
\end{defi}

\begin{defi}
We say that a cover $\mathscr{A} = (A_i)_{i \in I}$ is \textit{weakly good} if, for all non-empty $J \subseteq I$ such that ${A}_{J} \neq \emptyset$,
the map ${A}_{J} \to *$ is a weak homotopy equivalence, where $*$ is the one point space.
\end{defi}

\begin{defi}
Let $X$ be a topological space, and let $\mathscr{A} = (A_i)_{i \in I}$ be a cover.
For $T \in \Nrv(\mathscr{A})$, the \textit{latching space} is the subset
\[
	L(T) := \bigcup_{T \subsetneq J\subseteq I} {A}_{J} \subseteq {A}_{T} \; .
\]
\end{defi}

Finally, let us recall the {homotopy extension property}.
\begin{defi} \label{homotopy-extension-property}
 Let $X$ be topological spaces and let $A$ be a subset. We say that the pair $(X,A)$ satisfies the \emph{homotopy extension property} if for every commutative diagram of the following shape the dotted arrow exists
 \[
 \begin{tikzcd}
   A \arrow[hookrightarrow]{r}\arrow{d}[swap]{\op{id}_A\times \{0\}} & X \arrow{d}[swap]{\op{id}_X\times\{0\}}\arrow[bend left]{ddr}{f} &\\
   A\times [0,1] \arrow[hookrightarrow]{r}\arrow[bend right]{drr}{H} & X\times [0,1]\arrow[dotted]{dr}{\tilde{H}}&\\
   && Y.
 \end{tikzcd}
 \]
 In words, the pair $(X,A)$ has the homotopy extension property if for any map $f$, every homotopy $H$ of $f$ on $A$ can be extended to a homotopy $\tilde{H}$ of $f$ defined on all of $X$.
\end{defi}
\begin{rmk}
\label{rmk:cw_hep}
A large class of pairs has the homotopy extension property. For example, if $X$ is a CW-complex and $A$ a subcomplex, then $(X,A)$ satisfies the homotopy extension property (\cite[Proposition 0.16]{hatcher} or \cite[Proposition 7.10]{kozlov}).
We say more about the homotopy extension property in \cref{subsection:applications-UNT}.
\end{rmk}

We can now state the unified version of the Nerve Theorem.

\begin{thm}[Unified Nerve Theorem]
\label{abstract-nerve-thm}
Let $X$ be a topological space and let $\mathscr{A} = (A_i)_{i \in I}$ be a cover of $X$.
\begin{enumerate}
	\item Consider the natural map $\rho_{S} \colon \SmallBarCon(\mathscr{A}) \to X$.
	\begin{enumerate}
		\item If $\mathscr{A}$ is an open cover,
		then $\rho_{S}$ is a weak homotopy equivalence.
		If furthermore $X$ is a paracompact Hausdorff space, or, more generally, if $\mathscr{A}$ is numerable, then $\rho_{S}$ is a homotopy equivalence.
		\item Assume that $X$ is compactly generated
		and that $\mathscr{A}$ is a closed cover
		that is locally finite and locally finite dimensional.
		If for any $T \in \Nrv(\mathscr{A})$ the latching space $L(T) \subseteq {A}_{T}$
		is a closed subset and the pair $({A}_{T},L(T))$ satisfies the homotopy extension property,
		then $\rho_{S}$ is a homotopy equivalence.
	\end{enumerate}
	\item Consider the natural map
	    $\rho_{N} \colon \SmallBarCon(\mathscr{A}) \to | \Nrv(\mathscr{A}) |$.
	\begin{enumerate}
	    \item If $\mathscr{A}$ is (weakly) good, then $\rho_{N}$
		is a (weak) homotopy equivalence.
		\item If for all $J \in \Nrv({\mathscr{A}})$ the space ${A}_{J}$ is compactly generated
		and $\mathscr{A}$ is homologically good with respect to a coefficient ring $R$,
		then $\rho_{N}$ is an $R$-homology isomorphism.
	\end{enumerate}
\end{enumerate}
\end{thm}

We prove \cref{abstract-nerve-thm} in \cref{section:proof}.

\begin{rmk}
The compactly generated assumption in 2(b) is satisfied for example if $X$ is a locally compact Hausdorff space and $\mathscr{A}$ is an open cover. The assumption also holds if $X$ is compactly generated and $\mathscr{A}$ is a closed cover; this also includes the case of a cover of a CW-complex by subcomplexes. See \cref{ex:compactly_generated}.
\end{rmk}

\begin{rmk}
 If $X$ is a regular CW-complex and $\mathscr{A}$ is a cover of subcomplexes, then 2(b) can also be proven using spectral sequence techniques \cite[Chapter VII, Section~4]{brown}. Note that in this reference, the total complex of the double complex associated to the cover is isomorphic to the cellular chain complex of $\SmallBarCon(\mathscr{A})$. Moreover, these techniques can also be used to prove an analogous statement to \cref{connectivity-nerve} for homology groups \cite[Theorem 2.1]{meshulam}.
\end{rmk}

\begin{rmk}
\label{counterexamples-nervetheorem}
To illustrate the role of the technical assumptions in the unified nerve theorem, we now discuss some counterexamples when these assumptions are violated.
\begin{itemize}
\item The classical nerve theorem for a good open cover of a paracompact Hausdorff space is proven using 1(a) and 2(a). We will now give an example that shows that this paracompactness Hausdorff assumption, which ensures that the open cover is numerable, cannot be omitted in order to establish a homotopy equivalence between space and nerve. Consider the \emph{long ray} ${L}$, which is constructed as follows: Take the first uncountable ordinal $\omega_1$, which is a well-ordered set and its elements are all countable ordinals, and insert a unit interval $(0,1)$ between each countable ordinal $\alpha$ and its successor $\alpha+1$, yielding a totally ordered set. The topology on $L$ is given by the order topology. The long ray is a standard example for a non-paracompact Hausdorff space that is also not contractible \cite{counterexamples-topology,MR709260}. However, $L$ is weakly contractible and for any point $p\in L$ the open set $L_{<p}=\{t\in L\mid t<p\}\subset L$ is homeomorphic to the interval $[0,1)$. Thus, the open cover~$\mathscr{A}=(L_{<p})_{p\in \omega_1}$ is a good cover of $L$ and it follows from 1(a) and 2(a) that the nerve $|\Nrv{\mathscr{A}}|$ is weakly contractible and hence contractible by Whitehead's theorem \cite[Theorem 4.5]{hatcher}. This implies that the space $L$ and the nerve $|\Nrv{\mathscr{A}}|$ are not homotopy equivalent.

Note that the Hausdorff assumption cannot be dropped either; there exist non-Hausdorff paracompact spaces with good open covers that are not homotopy equivalent to the nerve of the cover. Specifically, any finite simplicial complex $K$ is weakly homotopy equivalent to a finite topological space $X$ whose points correspond to the simplices and whose open sets are upsets in the face poset \cite{MR196744}. The open sets corresponding to vertex stars form a good open cover of $X$ whose nerve is isomorphic to~$K$. However, it is straightforward to verify that every map $X \to |K|$ is locally constant, and therefore $|K|$ and $X$ are not homotopy equivalent in general.

\item The finiteness conditions in 1(b) control the size of the cover. If $\mathscr{A}$ is the cover of the circle $S^1$ by its points, then all conditions in 1(b) and 2(a) are satisfied except the locally finiteness assumption. As the nerve $|\Nrv{\mathscr{A}}|$ is a disjoint union of points, it is not homotopy equivalent to $S^1$.

\item Even if we are only interested in finite good and closed covers, the covered space does not need to be homotopy equivalent to the nerve of the cover. Consider the \emph{double comb space} $C$ and denote the two combs by $A_1$ and $A_2$ (see \cref{double_comb_warsaw}).
 Then, the nerve $|\Nrv{\mathscr{A}}|$ of the finite good and closed cover $\mathscr{A}=\{A_1,A_2\}$ of $C$ is contractible. Hence, it can not be homotopy equivalent to $C$, because the latter is not contractible. In this example, the pairs $(A_1,A_1\cap A_2)$ and $(A_2,A_1\cap A_2)$ do not satisfy the homotopy extension property. This shows that the conditions on the latching spaces are crucial, as all other assumptions in 1(b) and 2(a) are satisfied.

\item If $\mathscr{A}$ is any homologically good open cover of a locally compact Hausdorff space~$X$, then it follows from 1(a) and 2(b) that the space $X$ and the nerve $\Nrv{\mathscr{A}}$ have isomorphic homology groups. This conclusion does not hold if one replaces the open cover by a closed cover. Consider the \emph{Warsaw circle} $W\subseteq S^2$ that separates the sphere into two connected components $U_1$ and~$U_2$ (see \cref{double_comb_warsaw}).
The closed sets $A_1=U_1\cup W$ and $A_2=U_2\cup W$ cover the sphere and are contractible. Moreover, the intersection $A_1\cap A_2=W$ is acyclic and hence $\mathscr{A}=\{A_1,A_2\}$ is a homologically good closed cover of $S^2$. Nevertheless, the space $S^2$ and the nerve~$\Nrv{\mathscr{A}}$ do not have isomorphic homology groups, as $H_2(S^2)\cong\mathbb{Z}$ and $H_2(\Nrv{\mathscr{A}})\cong 0$. Hence, the conditions on the latching spaces are crucial, as all other assumptions in 1(b) and 2(b) are satisfied. This counterexample also shows that the nerve of a weakly good closed cover is not necessarily weakly equivalent to the space it covers.
\end{itemize}

\bigskip

\begin{center}
 \begin{minipage}[c]{.85\textwidth}
\centering
 \includegraphics[height=.3\textwidth]{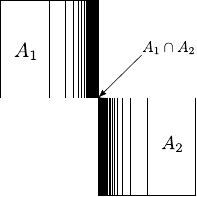}
 \hspace{.2\textwidth}
 \includegraphics[height=.3\textwidth]{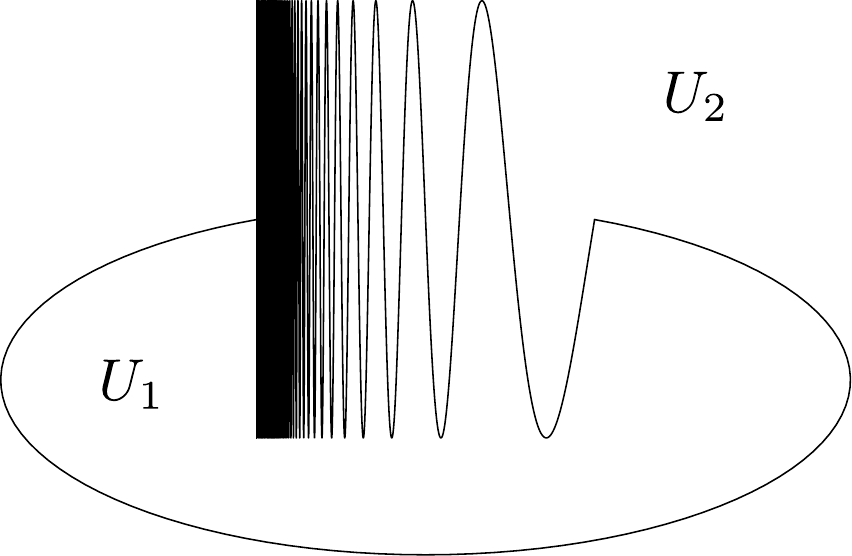}
\end{minipage}
\begingroup
\captionof{figure}{The double comb space $C$ (left) and the Warsaw circle $W$ (right).}
\label{double_comb_warsaw}
\endgroup
\end{center}
\end{rmk}

\subsection{Applications of the Unified Nerve Theorem} \label{subsection:applications-UNT}
The assumptions on the latching spaces in 1(b) of \cref{abstract-nerve-thm} might not be easy to check in all situations. We now give a reformulation, and a union theorem for pairs that satisfy the homotopy extension property, which help to verify these assumptions.
We also show in this subsection that \cref{abstract-nerve-thm}
implies the functorial nerve theorems
\cref{thm:comp_conv_zigzag} and \cref{functorial-simplicial-nerve}.

To study pairs that satisfy the homotopy extension property it suffices to consider neighborhood deformation retracts.
\begin{defi}
\label{defi:ndr_pair}
 A pair of topological spaces $(X,A)$ is called an \emph{NDR-pair} if there exist continuous maps $u\colon X\to [0,1]$ and $h\colon X\times[0,1]\to X$ such that
 \begin{multicols}{2}
    \begin{enumerate}[(i)]
    \item $A=u^{-1}(0)$
    \item $h(-,0)=\op{id}_X$
    \item $h(a,-)=a$ for all $a\in A$
    \item $h(x,1)\in A$ for all $x\in X$ with $u(x)<1$.
    \end{enumerate}
\end{multicols}
\end{defi}

 \begin{prop}[{\cite[Proposition 7.7]{kozlov}}]
 \label{prop:ndr_cofib}
  Let $A$ be a closed subspace of $X$. Then $(X,A)$ is an NDR-pair if and only if $(X,A)$ satisfies the homotopy extension property.
 \end{prop}
The following union theorem is due to Lillig \cite{lillig}.
\begin{prop}
         \label{prop:lilli_union}
          Let $A_0,\dots, A_n\subseteq X$ be closed subsets and assume that for all~$J\subseteq [n]$ the pair $(X, A_J)$ satisfies the homotopy extension property. Then the pair $(X,\bigcup_{i=0}^nA_i)$ also satisfies the homotopy extension property.
\end{prop}
This proposition, together with 1(b) in \cref{abstract-nerve-thm}, implies the following corollary, which does not involve the latching spaces.
\begin{cor}
\label{rho-s-no-latching-space}
 Let $X$ be a compactly generated topological space and $\mathscr{A}=(A_i)_{i\in[n]}$ a finite closed cover. Assume that for all $ I\subseteq J\subseteq [n]$ the pair $(A_I,A_J)$ satisfies the homotopy extension property. Then $\rho_{S} \colon \SmallBarCon(\mathscr{A}) \to X$ is a homotopy equivalence.
\end{cor}

We will now illustrate how these statements can be used to deduce the functorial nerve theorem for closed convex sets in $\bbR^d$ (\cref{thm:comp_conv_zigzag}) from the unified nerve theorem (\cref{abstract-nerve-thm}).
The proof of the following lemma is elementary and left to the reader.
\begin{lemma}
\label{interior-convex-set-is-convex}
        Let $K\subseteq \bbR^d$ be compact and convex. Assume that $\op{aff}K=\bbR^d$, where $\op{aff}K$ is the affine hull of $K$. Then $\op{int}K$ is convex and $\overline{\op{int}K}=K$.
\end{lemma}
\begin{prop}
       \label{prop:cofib_comp_conv}
       Let $K$ and $K'$ be compact and convex sets in $\bbR^d$ with $K\subseteq K'$.
       Then the pair $(K',K)$ satisfies the homotopy extension property.
\end{prop}
\begin{proof}
        Without loss of generality, we can assume that $\op{aff}K'=\bbR^d$ holds.

        First of all, let us assume that $K$ is the intersection of $K'$ with an affine subspace.
        Now, choose a point $x$ in $K$. By \cref{interior-convex-set-is-convex} and the proof of \cite[Lemma 1.1]{munkres}, we see that there exists a homeomorphism $\varphi\colon\bbR^d\to\bbR^d$ such that $\varphi(K')=\overline{\mathbb{B}}_1(0)$, $\varphi(x)=0$ and $\varphi(K)=\overline{\mathbb{B}}_1(0)\cap \bbR^{l}\times\{0\}^{d-l}$, with $l=\op{dim}\op{aff}K$. The pair \[(\varphi(K'),\varphi(K))=(\overline{\mathbb{B}}_1(0),\overline{\mathbb{B}}_1(0)\cap\bbR^{l}\times\{0\}^{d-l}) \] is a CW-pair and hence satisfies the homotopy extension property (\cref{rmk:cw_hep}).

        Now let $K\subseteq K'$ be such that $\op{aff}K=\op{aff}K'=\bbR^d$. As before, choose a point $x\in K$ and let $\varphi\colon\bbR^d\to\bbR^d$ be a homeomorphism with $\varphi(K)=\overline{\mathbb{B}}_1(0)$, $\varphi(x)=0$ and $\varphi(K')$ star-shaped with respect to $0$.
It is easy to see that $(\varphi(K'),\varphi(K))$ is an NDR-pair.
Hence, it follows from \cref{prop:ndr_cofib} that the pair $(K',K)$ satisfies the homotopy extension~property.

Finally, for arbitrary compact convex sets $K\subseteq K'$ we factor the inclusion as $K\hookrightarrow \op{aff}K\cap K'\hookrightarrow K'$. The claim now follows from the previous two cases together with transitivity of the homotopy extension property: if two pairs $(X,Y)$ and $(Y,Z)$ satisfy the homotopy extension property, then so does the pair $(X,Z)$.
\end{proof}

Using a truncation argument, we extend this result to any pair of closed convex~sets.
\begin{prop}
      Let $K$ and $K'$ be closed and convex sets in $\bbR^d$ with $K\subseteq K'$.
       Then the pair $(K',K)$ satisfies the homotopy extension property.
\end{prop}
\begin{proof}
We verify \cref{homotopy-extension-property}: Let $f\colon K'\to Y$ be any continuous map and $H\colon K\times [0,1]\to Y$ be any homotopy with $H(-,0)=f|_{K}$. We inductively construct an extension $\tilde{H}\colon K'\times [0,1]\to Y$ of $H$ with~$\tilde{H}(-,0)=f$. To this end, consider for every $n\in \mathbb{N}$ the compact convex sets $K_n=\CB{n}{0}\cap K\subseteq K$ and $K'_n=\CB{n}{0}\cap K'\subseteq K'$, where $\CB{n}{0}$ is the closed ball of radius $n$ centered at the origin. Denote by $f_n\colon K'_n\to Y$ and $H_n\colon K_n\times [0,1]\to Y$ the restrictions of $f$ and $H$, respectively.

By \cref{prop:cofib_comp_conv}, the pair $(K'_1,K_1)$ satisfies the homotopy extension property.
Hence, there exists an extension $\tilde{H}_1\colon K'_1\times [0,1]\to Y$ of $H_1$ that satisfies $\tilde{H}_1(-,0)=f_1$.

Let $n\in \mathbb{N}$ be arbitrary and consider the following diagram of inclusions:
\[\begin{tikzcd}[column sep={6em,between origins},row sep={2em,between origins}]
   K'_n\arrow{rr}\arrow{dr} & & K'_{n+1}\\
   & K'_n\cup K_{n+1}\arrow{ur}  &\\
   K_n\arrow{uu}\arrow{ur}\arrow{rr}  &&  K_{n+1}\arrow{uu}\arrow{ul}
  \end{tikzcd}
\]
By \cref{prop:lilli_union} and \cref{prop:cofib_comp_conv}, the pair $(K'_{n+1},K'_n\cup K_{n+1})$ satisfies the homotopy extension property.
Hence, we can extend the homotopy on $K'_n\cup K_{n+1}$ induced by $(\tilde{H}_n,H_{n+1})$ to a homotopy $\tilde{H}_{n+1}\colon K'_{n+1}\times [0,1]\to Y$ that satisfies $\tilde{H}_{n+1}(-,0)=f_{n+1}$. Taking to the colimit over all $n$ yields the desired extension $\tilde{H}$.
\end{proof}

Let $X \subset \bbR^d$ be a subset and let $\mathscr{B} = ( C_i )_{i \in[n]}$ be a finite cover of $X$ by closed convex sets. The previous corollary shows, together with the fact that any closed subset of $\bbR^d$ is compactly generated (\cref{ex:compactly_generated}), that the assumptions in \cref{rho-s-no-latching-space} are satisfied and hence, the map $\rho_S \colon \SmallBarCon(\mathscr{B}) \to X$ is a homotopy equivalence. As any cover by convex sets is good, it follows from 2(a) in \cref{abstract-nerve-thm} that the map $\rho_N \colon \SmallBarCon(\mathscr{B}) \to |\Nrv(\mathscr{B})|$ is a homotopy equivalence as well. This proves \cref{thm:comp_conv_zigzag}.

If in the functorial nerve theorem for covers by subcomplexes (\cref{functorial-simplicial-nerve}) we additionally assume that the cover is locally finite dimensional, then this theorem also follows readily from the unified nerve theorem (\cref{abstract-nerve-thm}): the realization of a simplicial complex is compactly generated (\cref{ex:compactly_generated}); moreover, the latching spaces are subcomplexes and thus satisfy the homotopy extension properties (\cref{rmk:cw_hep}).

\subsection{Simplicial Model Categories}
\label{subsection:bar_simp_model}

In order to prove \cref{abstract-nerve-thm},
we will need a generalization of the bar construction to other settings than
the category of topological spaces.
To make sense of the homotopy invariance property in other settings,
we will need a general framework for studying analogues of homotopy equivalences in other contexts.
There are many choices for such frameworks:
we will work with \emph{model categories}, which have been a standard tool of abstract homotopy theory
since they were introduced by Quillen in the 1960s.
A thorough introduction to model categories is beyond the scope of this paper
(see, e.g., \cite{dwyer-spalinski} for a friendly introduction),
but we will briefly introduce the aspects of model categories that are most relevant to this paper.

\paragraph{Model categories}
A model category is a category together with three distinguished classes of morphisms,
the \emph{weak equivalences}, \emph{fibrations}, and \emph{cofibrations},
which are required to satisfy certain axioms.
An admissible choice of these classes is called a \emph{model structure} on the underlying category.
The distinguished classes of morphisms also determine two distinguished classes of objects:
an object $X$ is \emph{fibrant} if the unique map from $X$ to the terminal object is a fibration,
and it is \emph{cofibrant} if the unique map from the initial object is a cofibration.
Before we give the axioms in \cref{defi:model}, it is useful to have in mind some basic examples.
\begin{ex} \label{model-structures-on-Top}
There are several important model categories whose objects are topological spaces.
As discussed earlier in this section, in order to avoid pathological behavior, one often considers some subcategory of $\Top$;
we choose the subcategory of compactly generated spaces.
There is a model structure on the category of compactly generated spaces
for which the weak equivalences are the homotopy equivalences
and the cofibrations are the \emph{Hurewicz cofibrations},
which are the maps $i \colon A \to X$ that satisfy the homotopy extension property
(see \cref{homotopy-extension-property},
and replace the inclusion $A \subset X$ with~$i$).
This is called the \emph{Hurewicz model structure}.
It was originally shown to be a model structure (on the category of all topological spaces)
by Str{\o}m \cite{strom}; see \cite[Theorem 17.1.1]{may-ponto}
for an account in the setting of compactly generated spaces.
Every space is both fibrant and cofibrant in the Hurewicz model structure,
which is quite rare.

There is another model structure on the category of compactly generated spaces,
called the \emph{Quillen model structure},
for which the weak equivalences are the weak homotopy equivalences,
i.e., the maps that induce a bijection on path components
and an isomorphism on homotopy groups for all choices of base point.
This was first studied by Quillen in his original work on model categories
\cite{quillen-homotopical-algebra}; see \cite[Theorem 17.2.2]{may-ponto}
for an account in our setting.
Every space is fibrant in the Quillen model structure,
and every CW complex is cofibrant.
\end{ex}

\begin{ex} \label{chain-complexes}
Model categories can be used to study homological algebra.
Let $R$ be a commutative ring. There is a model structure on the category of non-negatively graded
chain complexes of $R$-modules, for which the weak equivalences are the quasi-isomorphisms,
and the cofibrations are those monomorphisms that have a degreewise-projective cokernel.
In particular, the cofibrant objects are the degreewise-projective chain complexes.
This is another of the original examples from \cite{quillen-homotopical-algebra}.
\end{ex}

\begin{defi}
\label{defi:model}
A \emph{model category} $\mathscr{M}$ is a category which is equipped with three subcategories of morphisms
called \emph{weak equivalences}, \emph{fibrations} and \emph{cofibrations} such that the following axioms hold:
\begin{enumerate}
  	\item The category $\mathscr{M}$ has all small limits and colimits.
  	\item (2-of-3) If $f$ and $g$ are maps of $\mathscr{M}$ such that $g \circ f$ is defined
  	and two of the maps $f, \, g, \, g \circ f$ are weak equivalences, then so is the third.
    \item If $f$ is a retract of $g$ and $g$ is a weak equivalence, fibration, or cofibration,
  	then so is $f$.
    \item Given a commutative square
    \[
    \begin{tikzcd}
    A \arrow[d, "i"'] \arrow[r, "f"] & X \arrow[d, "p"] \\
    B\arrow[r, "g"'] & Y,
    \end{tikzcd}
    \]
    where $i$ is a cofibration and $p$ is a fibration,
    then there is a map $h \colon B \to X$ such that $f = h \circ i$ and $g = p \circ h$
    if either of $i$ or $p$ is a weak equivalence.
    \item Any map $f$ can be factored as (i) $f = p \circ i$,
    where $i$ is a cofibration and $p$ is a fibration and a weak equivalence,
    and (ii) $f = p' \circ i'$
    where $i'$ is a cofibration and a weak equivalence and $p'$ is a fibration.
\end{enumerate}
\end{defi}

\begin{rmk}
The definition of model category has evolved since it was first introduced.
For example, we require all small limits and colimits,
while Quillen originally required only all finite limits and colimits.
For a discussion, see \cite[Chapter 1]{hovey}
\end{rmk}

\begin{rmk}
In a model category the weak equivalences together with the fibrations or the cofibrations determine the third subcategory; see, e.g., \cite[Lemma 1.1.10]{hovey}.
\end{rmk}

Many algebraic topologists prefer to work with certain kinds of combinatorial models of spaces,
rather than with topological spaces themselves.
These combinatorial models are called \emph{simplicial sets},
and they are somewhat similar to simplicial complexes.
While they may appear more complicated than simplicial complexes
-- for example, every simplicial set has simplices in every dimension,
even the simplicial set that models the one-point space --
they have better categorical properties. For example, there is a geometric realization functor $|-|$
from the category of simplicial sets
to the category of compactly generated topological spaces,
and this functor preserves all small colimits
and all finite limits (we define this construction in \cref{section:proof}).
So, one can take limits and colimits in the category of simplicial sets,
and these will model the corresponding limits and colimits of topological spaces.
See \cite{sset-intro} for a friendly introduction to this topic.

\begin{defi}
The \emph{simplex category}, denoted by $\Delta$, has as objects the finite ordinals
$\{[n]=\{0,1,\dots,n\} \mid n \geq 0 \}$, with the morphisms being the order preserving~maps.
\end{defi}

 \begin{defi}
A \emph{simplicial set} is a functor $X \colon \Delta^{\opp} \to \Set$,
and a \emph{morphism of simplicial sets} is a natural transformation.
The set $X_n = X([n])$ is the set of $n$-\emph{simplices} of $X$.
The category of simplicial sets is denoted by $\sSet$.
More generally, if $\mathscr{C}$ is any category, a \emph{simplicial object in} $\mathscr{C}$
is a functor $\Delta^{\opp} \to \mathscr{C}$, and the category of simplicial objects in $\mathscr{C}$ is denoted by $\sC$.
\end{defi}

\begin{ex}
\label{ex:n_simp}
The \emph{Yoneda embedding} $Y\colon \Delta\hookrightarrow \op{sSet},\ [n]\mapsto \op{Hom}_\Delta(-,[n])$ gives rise to a simplicial set for each $n \in \mathbb N$. We denote $Y([n])$ by $\Delta^n$ and call it the \textit{standard $n$-simplex}.
\end{ex}

\begin{ex}
Let $X$ be a topological space. The \textit{singular simplicial set of $X$} is the simplicial set $\op{Sing}(X)$ with \[\op{Sing}(X)([n])=\Hom(|\Delta^n|,X), \] where $|\Delta^n|$ is the standard topological $n$-simplex,
and $\Hom(-,-)$ denotes the set of continuous maps.
\end{ex}

A fundamental fact about the relationship between simplicial sets and topological spaces
is that the functor $\op{Sing}\colon \CGSpc \to \sSet$ is right adjoint to the geometric
realization $| - |\colon \sSet \to \CGSpc$ mentioned above.
This adjunction is what allows us to use simplicial sets as a model for spaces;
we say more about this below.

\begin{ex}
Let $\mathcal{C}$ be a category. The \emph{(categorical) nerve of $\mathcal{C}$} is the simplicial set $N(\mathcal{C})$ such that \[N(\mathcal{C})([n])=\{v_0\to v_1\to\cdots\to v_n\mid \text{string of composable morphisms in $\mathcal{C}$} \}. \]

If $\mathscr{A}=(A_i)_{i\in I}$ is a cover of a topological space, then the finite non-empty intersections of cover elements form a category with morphisms given by inclusion $A_J\hookrightarrow A_{J'}$ if $J'\subseteq J$. 
The nerve of this category and the nerve of the cover have homeomorphic geometric realizations, explaining the common name for the two constructions.
\end{ex}

We can now introduce two more examples of model categories,
both of which play a role in the proof of \cref{abstract-nerve-thm}.

\begin{ex}
The category of simplicial sets can be equipped with a model structure, also called the \emph{Quillen model structure}, where the the cofibrations are the monomorphisms and the weak equivalences are those maps that are mapped to weak homotopy equivalences when applying the geometric realization functor; see, for example, \cite[Chapter I]{goerss-jardine}.
From the definition of cofibration, it follows that every simplicial set is cofibrant.
\end{ex}

\begin{ex} \label{sRMod}
An alternative to the chain complexes of \cref{chain-complexes} is the category of simplicial $R$-modules.
Here, $R$ is a commutative ring as before,
and a simplicial $R$-module is a simplicial object in the category $\RMod$ of $R$-modules, i.e., a functor $\Delta^{\opp} \to \RMod$.
Denoting the category of simplicial $R$-modules by $\sRMod$,
we let $\freeR\colon \sSet \to \sRMod$ denote the functor that is induced by the free $R$-module functor $\freeR\colon \mathsf{Set}\to \RMod$;
the forgetful functor $U\colon \sRMod \to \sSet$ is right adjoint to $\freeR$.
Then there is a model structure on $\sRMod$ such that the weak equivalences and fibrations are exactly those morphisms whose underlying map of simplicial sets is a weak equivalence and fibration, respectively \cite[Proposition 4.2]{goerss-schemmerhorn}.
Moreover, a continuous map $X\to Y$ is an $R$-homology isomorphism if and only if the induced map $\freeR(\Sing (X))\to \freeR(\Sing (Y))$ is a weak equivalence of simplicial $R$-modules \cite[Dold-Kan Theorem 8.4.1]{weibel}.
\end{ex}

We have now encountered two important adjunctions connecting model categories:
the adjunction $(| - |, \op{Sing})$ relating spaces with simplicial sets, and
$(\freeR, U)$ relating simplicial sets with simplicial $R$-modules.
In general, a \emph{Quillen adjunction} between model categories is an adjunction
such that the left adjoint preserves cofibrations and trivial cofibrations,
or equivalently,
the right adjoint preserves fibrations and trivial fibrations.
These adjunctions are the main way to relate model categories;
both of the adjunctions just mentioned are Quillen adjunctions.

A Quillen adjunction $(F, G)$,
where the left adjoint $F$ is a functor $\mathscr{M} \to \mathscr{N}$,
is a \emph{Quillen equivalence} if, for all cofibrant $X$ in $\mathscr{M}$
and all fibrant $Y$ in $\mathscr{D}$, a map $FX \to Y$ is a weak equivalence in $\mathscr{N}$
if and only if the corresponding map $X \to GY$ is a weak equivalence in $\mathscr{M}$.
The adjunction $(| - |, \op{Sing})$ is a Quillen equivalence
when $\CGSpc$ is given the Quillen model structure;
see for example \cite[Theorem 2.4.25,Theorem 3.6.7]{hovey}.
This simple definition has powerful consequences, and
we now describe one that plays a role in the proof of \cref{abstract-nerve-thm}.
Since $(F, G)$ is an adjunction, there is a natural map
$\eta\colon X \to GFX$ called the \emph{unit},
and another $\varepsilon\colon FGY \to Y$ called the \emph{counit}.
If $(F, G)$ is a Quillen equivalence,
then these maps are weak equivalences, subject to additional fibrancy and cofibrancy assumptions.
See \cite[Proposition 1.3.13]{hovey} for details.
Once we know that $(| - |, \op{Sing})$ is a Quillen equivalence,
then it follows immediately that the unit
$K \to \op{Sing}(|K|)$ is a weak equivalence for every simplicial set $K$,
and the counit $|\op{Sing}(Y)| \to Y$ is a weak equivalence for every compactly-generated space $Y$. The additional fibrancy and cofibrancy assumptions are vacuous in this case,
as every simplicial set is cofibrant and every space is fibrant.

\paragraph{Simplicial model categories}
At the beginning of this section,
we discussed the importance of mapping spaces.
A simplicial model category $\mathscr{M}$
is a model category equipped with additional structure
that generalizes this feature of algebraic topology;
see \cite[Definition 11.4.4]{cat_hom} for a precise definition.
For any two objects $X$ and $Y$ of a simplicial model category $\mathscr{M}$,
we have a simplicial set $\textbf{Hom}_{\mathscr{M}}(X,Y)$
that encodes homotopy-theoretic information about $X$ and $Y$.
Formally, one requires that the model category $\mathscr{M}$
be enriched in simplicial sets, and tensored and cotensored.
One then imposes an additional axiom that relates this structure to the model structure. 
We will omit the formal definitions, since we will not use most of the structure explicitly. 
Rather, for the proof of the unified nerve theorem, 
we will need a few facts about simplicial model categories, 
principally \cref{prop:homotopy_lemma}. 
However, we will use the \emph{tensoring} explicitly in order to define the bar construction in a simplicial model category, and so we introduce it now.
If $\mathscr{M}$ is a simplicial model category,
then for any object $X$ of $\mathscr{M}$ and any simplicial set $K$
there is an object $X\otimes K$ of $\mathscr{M}$,
and this construction gives a functor
$\mathscr{M}\times \sSet  \to \mathscr{M}$.
Furthermore, for any $X$,
the functor $\textbf{Hom}_{\mathscr{M}}(X,-)\colon \mathscr{M} \to \sSet$
has a left adjoint given by $X\otimes- \colon \sSet \to \mathscr{M}$. 
This adjunction in particular yields an isomorphism
\[
	\Hom_{\mathscr{M}}(X \otimes K, Y) \iso \Hom_{\sSet}\left(K, \textbf{Hom}_{\mathscr{M}}(X,Y)\right)
\]
for any $X,Y$ in $\mathscr{M}$ and any simplicial set $K$. 
The motivation for the terminology ``tensoring'', and the notation $\otimes$, 
comes from this adjunction, which is analogous to the tensor-hom adjunction from linear algebra.

\begin{ex}
The category $\CGSpc$ of compactly-generated spaces is enriched in simplicial sets, and tensored and cotensored.
This makes $\CGSpc$ a simplicial model category with either the Hurewicz or Quillen model structures.
If $X$ and $Y$ are compactly generated spaces, $\textbf{Hom}(X,Y)$ is the simplicial set with
\[
	\textbf{Hom}(X,Y)_n = \Hom(X \times |\Delta^n|, Y).
\]
So, the zero-simplices of $\textbf{Hom}(X,Y)$ are maps from $X$ to $Y$,
the one-simplices are homotopies, the two-simplices are ``homotopies between homotopies'',
and so on.
The operation $\otimes$ is characterized by $X\otimes \Delta^n= X\times |\Delta^n|$.
\end{ex}

\begin{ex}
The Quillen model structure on $\sSet$ gives a simplicial model category,
and the operation $\otimes$ is the cartesian product.
\end{ex}

\begin{ex}
The category $\sRMod$ is a simplicial model category,
with the model structure described in \cref{sRMod}.
If $K$ is a simplicial set and $M$ is a simplicial $R$-module,
then $M\otimes K$ is the simplicial $R$-module with
$(M\otimes K)_n =M_n \otimes_{R}\freeR(K_n) $.
\end{ex}

\subsection{Proof of the Unified Nerve Theorem} \label{section:proof}

We can now define the bar construction in the setting of a simplicial model category,
generalizing the construction of \cref{section:func_nerve}.

\begin{defi} \label{simplicial-bar-construction}
Let $P$ be a poset and let $\mathscr{M}$ be a simplicial model category.
The \textit{simplicial bar construction} of a functor $F \colon P \to \mathscr{M}$
is the simplicial object
\[
	\AbBarCon_{\bullet}(F)\colon \Delta^{\opp} \to \mathscr{M}
\]
whose $n$-simplices $\AbBarCon_n(F)$ are defined by the coproduct
\[
	\AbBarCon_{n}(F) = \coprod\limits_{v_0 \leq v_1 \leq \cdots \leq v_n} F(v_0).
\]
Equivalently, the coproduct is indexed by functors of the form $\gamma \colon [n] \to P$.
For any map $\theta \colon [m] \to [n]$ in $\Delta$,
$\theta^* \colon \AbBarCon_{n}(F) \to \AbBarCon_{m}(F)$ takes the summand indexed by $\gamma$
to the summand indexed by $\gamma \circ \theta$,
via the map $F(\gamma(0)) \to F(\gamma(\theta(0)))$.
\end{defi}

The identifications that were used to define the bar construction for topological spaces
are achieved in this setting by the categorical notion of a \emph{coend}:

\begin{defi}
Let $\mathscr{C}$ be a small category, $\mathscr{E}$ any category,
and $H \colon \mathscr{C}^{\opp} \times \mathscr{C} \to \mathscr{E}$ a functor.
The \emph{coend} $\int^{\mathscr{C}} H$,
sometimes denoted $\int^{c \in \mathscr{C}} H(c, c)$,
is an object of $\mathscr{E}$ equipped with arrows $H(c, c) \to \int^{\mathscr{C}} H$
for each $c \in \mathscr{C}$ that are collectively universal with the property that the diagram
\[
\begin{tikzcd}
	H(c', c) \arrow[r, "f_*"] \arrow[d, "f^*"'] & H(c', c') \arrow[d] \\
	H(c, c) \arrow[r] & \int^{\mathscr{C}} H
\end{tikzcd}
\]
commutes for each $f \colon c \to c'$ in $\mathscr{C}$.
\end{defi}

We can use this notion to define the geometric realization functor
$| - |\colon \sSet \to \CGSpc$ mentioned in \cref{subsection:bar_simp_model}.
Writing $\Delta$ for the simplex category as before, there is a functor
$\Delta \to \CGSpc$ that takes $[n]$ to the standard $n$-simplex.
If $X\colon \Delta^{\opp} \to \Set$ is a simplicial set,
$|X|$ is the coend of the functor $\Delta^{\opp} \times \Delta \to \CGSpc$
that takes $([n], [m])$ to $X_n\times |\Delta^m|$,
where $X_n$ has the discrete topology. This is written
\[
	|X| = \int^{[n] \in \Delta} X_n\times |\Delta^n|  \, .
\]
We now apply the same idea to the simplicial bar construction:

\begin{defi} \label{defi:bar_const_sm}
Let $P$ be a poset and let $\mathscr{M}$ be a simplicial model category.
The \emph{bar construction} of a functor $F \colon P \to \mathscr{M}$ is the coend
\[
	\AbBarCon(F) = \int^{[n] \in \Delta} \AbBarCon_n(F)\otimes \Delta^n .
\]
\end{defi}

Note that there are canonical maps
$\AbBarCon_n(F)\otimes \Delta^n  \to \AbBarCon_n(F) \to \op{colim} F$ for all $n \geq 0$,
which induce a map $\AbBarCon(F) \to \op{colim} F$
by the universal property of the coend.

\begin{rmk}
\label{rmk:barInCGSpc}
Let $F \colon P \to \CGSpc$ be a functor valued in 
the category $\CGSpc$ of compactly generated spaces. 
Since $\CGSpc$ is a simplicial model category 
(with either the Hurewicz or Quillen model structures), 
we can consider the bar construction  $\AbBarCon(F)$. 
However, if one thinks of $F$ as a functor valued in $\Top$, 
forgetting that $F(p)$ is compactly generated for each $p \in P$, 
then one could instead consider the bar construction 
of \cref{defi:hocolim_spaces}; 
one could also mimic the construction of \cref{defi:bar_const_sm}, 
but taking the various limits and colimits in $\Top$ 
rather than in $\CGSpc$. 
Fortunately, all of these constructions coincide, 
as we will now explain.

Let $F \colon P \to \Top$ be a functor valued in the category of all topological spaces. 
Just in this remark, 
let $\AbBarCon^*(F)$ be the coend
\[
	\AbBarCon^*(F) = \int^{[n] \in \Delta} \AbBarCon_n(F)\times \vert \Delta^n \vert
\]
in $\Top$, where $\AbBarCon_{\bullet}(F) \colon \Delta^{\opp} \to \Top$ 
is defined as in \cref{simplicial-bar-construction}.

Following Dugger--Isaksen \cite[Appendix A]{dugger-isaksen},
we can compute $\AbBarCon^*(F)$
as a sequential colimit of pushouts in $\Top$, as follows.
For $k=0$, we define $\AbBarCon(F)(0) = \coprod_{v\in P} F(p)$,
and for $k>0$ we inductively define $\AbBarCon(F)(k)$ as the pushout
\begin{equation} \label{bar-as-pushout}
\begin{tikzcd}
\coprod\limits_{v_0<\dots<v_{k}} F(v_0) \times \vert \partial \Delta^k \vert
\arrow[dr,"\ulcorner",phantom, near start] \arrow{r} \arrow[hookrightarrow]{d}
& \AbBarCon(F)(k-1) \arrow[dotted]{d}{f_{k-1}} \\
\coprod\limits_{v_0<\dots<v_{k}} F(v_0) \times \vert \Delta^k \vert \arrow[dotted]{r} & \AbBarCon(F)(k)                                                                                                                                                                                                                                                \end{tikzcd}
\end{equation}
where the top horizontal map is defined using the face maps
$\AbBarCon_{k}(F) \to \AbBarCon_{k-1}(F)$.
Then we have an isomorphism $\AbBarCon^*(P) \cong \op{colim}_{k}\AbBarCon(F)(k)$.

Now, say $F$ is valued in compactly generated spaces. 
We will use this characterization of $\AbBarCon^*(F)$ 
to show that it coincides with the bar construction computed in $\CGSpc$. 
The key fact we need is that, 
for any diagram in $\CGSpc$, 
if its (co)limit computed in $\Top$ happens to be compactly generated, 
then this is also the (co)limit in $\CGSpc$. 
This follows from the existence of the pair of adjunctions relating $\CGSpc$ and $\Top$, 
mentioned right below \cref{compactly-generated}: 
because of these adjunctions, $\CGSpc$ is a reflective subcategory of $\op{k-spaces}$, 
and $\op{k-spaces}$ is a coreflective subcategory of $\Top$ 
\cite[Definition 4.5.12]{riehl-context}. 
As the disjoint union of compactly generated spaces is compactly generated, 
and because $| \partial \Delta^n  |$ as well as $| \Delta^n  |$ are locally compact Hausdorff, 
it follows from \cite[Proposition 2.6. and Corollary 2.16]{strickland} 
that the spaces on the left hand side of diagram \ref{bar-as-pushout} and $\AbBarCon(F)(0)$ are compactly generated.
Moreover, because the pushout of compactly generated spaces along a closed inclusion 
is again compactly generated \cite[p.40]{may-concise},
it follows that the spaces $\AbBarCon(F)(k)$ are compactly generated.
By \cite[Proposition 2.35]{strickland}, the maps $f_k$ are closed inclusions.
Finally, it follows again from \cite[p.40]{may-concise} that 
$\AbBarCon^*(F) \iso \op{colim}_{k}\AbBarCon(F)(k)$ is compactly generated.
Thus, $\AbBarCon^*(F)$ agrees with the bar construction 
$\AbBarCon(F)$ computed in $\CGSpc$.

Furthermore, one can use this method for building the bar construction
as a sequential colimit of pushouts
to check that, given $F \colon P \to \Top$, the bar construction we just discussed 
is naturally homeomorphic to the bar construction of \cref{defi:hocolim_spaces},
which justifies using the same notation in both places.
\end{rmk}

\begin{rmk}
\label{geom-realization-simplicial-bar}
The coend construction used to define $\AbBarCon(F)$ in \cref{defi:bar_const_sm}
is an example of the \emph{geometric realization} of a simplicial space.
We note that the blowup complex of \cref{defi:blowup} can also be seen as the geometric realization of a simplicial space,
the \emph{ordered \v{C}ech complex} of \cite[Section 2.5]{dugger-isaksen}.

The map
$\pi_{\Sd N} \colon \AbBarCon(\mathscr{D}_{\mathscr{U}}) \to
|\Sd \Nrv(\mathscr{U})|$
is the geometric realization of a map of simplicial spaces
$\AbBarCon_{\bullet}(\mathscr{D}_{\mathscr{U}}) \to
\AbBarCon_{\bullet}(\mathop{*^{\np{\mathscr{U}}}})$, and the map
$\pi_{S} \colon \AbBarCon(\mathscr{D}_{\mathscr{U}}) \to X$
is also the geometric realization of a map of simplicial spaces
$\AbBarCon_{\bullet}(\mathscr{D}_{\mathscr{U}}) \to X_\bullet$,
where $X_n = X$ for all $n$.
Analogously, the maps
$\rho_N \colon \SmallBarCon(\mathscr{U}) \to |\Nrv(\mathscr{U})|$
and $\rho_S \colon \SmallBarCon(\mathscr{U}) \to X$
can be seen as geometric realizations of maps of simplicial spaces.
\end{rmk}

\begin{ex} \label{pobar-is-a-bar}
If we leave aside the requirement that we work in a simplicial model category,
the $\PoBar$ construction from \cref{section:triangulated-covers} is a bar construction.
In more detail, let $K$ be a simplicial complex, and let $\mathscr{A} = (K_i)_{i \in I}$
be a cover by subcomplexes.
Let $\mathscr{D}_{\mathscr{A}} \colon \np{\mathscr{A}} \to \Po$
be the functor with $\mathscr{D}_{\mathscr{A}}(J) = \Pos(\cap_{i \in J} K_i)$.
There is a simplicial object $\AbBarCon_{\bullet}(\mathscr{D}_{\mathscr{A}})$ in $\Po$
defined as in \cref{simplicial-bar-construction},
and the inclusion $\Delta \subset \Po$ defines a cosimplicial object,
i.e., a functor $\Delta \to \Po$. Then
\[
	\PoBar(\mathscr{A}) = \int^{[n]}\AbBarCon_n(\mathscr{D}_{\mathscr{A}})\times  [n] .
\]
\end{ex}

Given a sufficiently well-behaved diagram $F \colon P \to \mathscr{M}$
in a simplicial model category~$\mathscr{M}$,
the bar construction of $F$ computes the homotopy colimit of $F$.
This appears as \cite[Corollary 5.1.3]{cat_hom}.
In the proof of \cref{abstract-nerve-thm},
we will use two statements closely related to this result.
The first one says that the bar construction is homotopical for pointwise cofibrant diagrams (see \cite[Corollary 5.2.5]{cat_hom}):

\begin{prop}
\label{prop:homotopy_lemma}
Let $\mathscr{M}$ be a simplicial model category, let $P$ be a poset,
and let $F, G \colon P \to \mathscr{M}$ be pointwise cofibrant diagrams.
For a natural transformation $F \Rightarrow G$ 
that is a pointwise weak equivalence,
the induced map $\AbBarCon(F) \to \AbBarCon(G)$ is a weak equivalence.
\end{prop}

Recall that in \cref{prop:homotopy_lemma_top} we already saw a similar statement for topological spaces, saying that the bar construction
respects pointwise homotopy equivalences,
without any pointwise cofibrancy or compactly-generated assumptions.
There is an analogous result for weak homotopy equivalences,
which follows from work of Dugger--Isaksen \cite{dugger-isaksen}:

\begin{prop}
\label{prop:weak_homotopy_lemma_top}
Let $P$ be a poset, and let $F,G\colon P\to\Top$ be diagrams of topological spaces.
For a natural transformation $F \Rightarrow G$ 
that is a pointwise weak homotopy equivalence,
the induced map $\AbBarCon(F) \to \AbBarCon(G)$ is a weak homotopy equivalence.
\end{prop}

\begin{proof}
The natural transformation $F \Rightarrow G$ induces a map
$\AbBarCon_{\bullet}(F) \to \AbBarCon_{\bullet}(G)$
such that $\AbBarCon_{n}(F) \to \AbBarCon_{n}(G)$ is a weak homotopy equivalence for all $n \geq 0$.
It is straightforward to check that the simplicial bar construction has free degeneracies in the sense of \cite[Definition A.4]{dugger-isaksen},
and therefore the induced map $\AbBarCon(F) \to \AbBarCon(G)$ is a weak homotopy equivalence
by \cref{geom-realization-simplicial-bar} together with \cite[Corollary A.6]{dugger-isaksen}.
\end{proof}

We will soon use this last result to prove the ``weak'' version of
\cref{abstract-nerve-thm} 2(a).
If we did not know this result, and applied only \cref{prop:homotopy_lemma},
we would need the additional assumption that
the intersection $A_J$ is cofibrant in the Quillen model structure on $\CGSpc$
for all $J \in \Nrv(\mathscr{A})$.

For the proof of \cref{abstract-nerve-thm},
we will also need a second result related to the general fact that the bar construction
computes the homotopy colimit.
This is similar to \cite[Lemma 4.5]{wzz-hocolims}, for example. Recall that $\mathscr{D}_{\mathscr{A}} \colon \np{\mathscr{A}} \to \Top$ denotes the nerve diagram of the cover $\mathscr{A}$ of a topological space $X$.

\begin{prop} \label{hocolim-colim-comparison}
Let $X$ be a compactly generated space,
and let $\mathscr{A}$ be a closed cover that is locally finite and locally finite dimensional.
If for all $T \in \Nrv(\mathscr{A})$ the latching space $L(T) \subseteq {A}_T$
is closed and the pair $({A}_T, L(T))$ satisfies the homotopy extension property,
then the natural map
$\AbBarCon(\mathscr{D}_{\mathscr{A}}) \to
\op{colim} \mathscr{D}_{\mathscr{A}}\cong X$
is a homotopy equivalence.
\end{prop}

Only in this proof, we will make use of model structures on functor categories, in particular, the projective and Reedy model structures, which we do not introduce in this paper.
For the interested reader, we refer to \cite[Section~13]{dugger-hocolim} and \cite[Chapter~15]{hirschhorn}.

\begin{proof}
A closed subspace of a compactly generated space is also compactly generated,
so $\mathscr{D}_{\mathscr{A}} \colon \np{\mathscr{A}} \to \Top$
takes values in the subcategory $\CGSpc$ of compactly generated spaces. 
Hence, we can take the bar construction of $\mathscr{D}_{\mathscr{A}}$, 
as discussed in \cref{rmk:barInCGSpc}.

As $\mathscr{A}$ is locally finite dimensional,
the poset $\np{\mathscr{A}}$ is an upwards-directed Reedy category,
with $\op{deg}(J) = \sup \{(|J'| - |J|) \mid J' \in \Nrv(\mathscr{A}) \text{ with } J \subseteq J' \}$.
Working with the Hurewicz model structure on $\CGSpc$,
the Reedy model structure on the functor category
$\Fun(\np{\mathscr{A}}, \CGSpc)$
coincides with the projective model structure as
$\np{\mathscr{A}}$ is upwards-directed.
This is immediate from the definition of the Reedy model structure;
see \cite[Proposition 13.12]{dugger-hocolim} for a clear discussion of the relationship with the projective model structure.

The condition on the latching spaces implies that all inclusions
$L(T) \subseteq {A}_T$ are Hurewicz cofibrations,
so that $\mathscr{D}_{\mathscr{A}}$ is Reedy cofibrant and thus projective cofibrant.
As the bar construction $\AbBarCon(\mathscr{D}_{\mathscr{A}})$
computes the homotopy colimit of $\mathscr{D}_{\mathscr{A}}$
\cite[Corollary 5.1.3]{cat_hom}, the natural map
$\AbBarCon(\mathscr{D}_{\mathscr{A}}) \to
\op{colim} \mathscr{D}_{\mathscr{A}}$
is a homotopy equivalence. As the cover is locally finite and $X$ is compactly generated, it follows from \cite[Corollary 2.23]{strickland} that the colimit calculated in $\CGSpc$ coincides with the one in $\Top$, and by \cref{space-is-colimit} this is naturally homeomorphic to $X$.
\end{proof}

We are now ready to prove the unified nerve theorem.

\begin{proof}[Proof of \cref{abstract-nerve-thm}]
In \cref{functorial_nerve_via_bar}, we explained how one can
prove that the natural maps $\rho_S$ and $\rho_N$ from the blowup complex
are equivalences by proving that the natural maps
$\pi_S$ and $\pi_{\Sd N}$ from the bar construction are equivalences (see Diagram \ref{diagram:two-roofs}).
So, we work with the bar construction in this proof.

The first part of 1(a) follows from work of Dugger and Isaksen
\cite[Theorem 2.1 and Proposition 2.7]{dugger-isaksen};
in \cref{appendix-open-covers} we give a short proof using their ideas.
The second part of 1(a) is essentially the content of \cite[Proposition 4G.2]{hatcher}; note that the author uses the convention that paracompact spaces are assumed to be Hausdorff.
Statement 1(b)
is the content of \cref{hocolim-colim-comparison}.

We now prove 2(a).
By assumption, the unique natural transformation $\mathscr{D}_{\mathscr{A}} \Rightarrow *^{\np{\mathscr{A}}}$ from the nerve diagram of the cover $\mathscr{A}$ to the constant diagram on the one-point space is a pointwise (weak) homotopy equivalence.
The results now follow from \cref{prop:homotopy_lemma_top} and \cref{prop:weak_homotopy_lemma_top}, respectively.
Alternatively, see \cref{appendix-open-covers}
for a short proof of the fact that $\rho_{N}$ is a weak homotopy equivalence
whenever $\mathscr{A}$ is a weakly good cover.

We now prove 2(b). For every compactly generated space $Z$, there is a natural weak homotopy equivalence
$|\Sing(Z)| \to Z$,
given by the counit of the adjunction $(|-|, \Sing)$
as explained in \cref{subsection:bar_simp_model}.
So, there is a pointwise weak homotopy equivalence
\[|-| \circ \Sing \circ \mathop{\mathscr{D}_{\mathscr{A}}} \Rightarrow \mathscr{D}_{\mathscr{A}},\] that induces, by \cref{prop:weak_homotopy_lemma_top}, a weak homotopy equivalence
\begin{equation}
\label{eq:unif_nerve_eq}
 \AbBarCon(|-| \circ \Sing \circ \mathop{\mathscr{D}_{\mathscr{A}}})
\to \AbBarCon(\mathscr{D}_{\mathscr{A}}).
\end{equation}

We work with the model structure on the category $\sRMod$ of simplicial $R$-modules
described in \cref{sRMod};
recall that a map $X \to Y$ of compactly-generated spaces
is an $R$-homology isomorphism if and only if the induced map
$\freeR(\Sing (X))\to \freeR(\Sing (Y))$
is a weak equivalence of simplicial $R$-modules.
By our assumption that the natural transformation
$\mathscr{D}_{\mathscr{A}} \Rightarrow *^{\np{\mathscr{A}}}$ is a pointwise $R$-homology isomorphism,
the natural transformation
$\freeR \circ \Sing \circ \mathop{\mathscr{D}_{\mathscr{A}}} \Rightarrow \freeR \circ \Sing \circ \mathop{*^{\np{\mathscr{A}}}}$
is a pointwise weak equivalence of simplicial $R$-modules.
As $\freeR$ preserves cofibrant objects (since it is the left adjoint of a Quillen adjunction), and every simplicial set is cofibrant,
both diagrams are pointwise cofibrant.
So, by \cref{prop:homotopy_lemma}, the induced map
$\AbBarCon(\freeR \circ \Sing \circ \mathop{\mathscr{D}_{\mathscr{A}}})
\to \AbBarCon(\freeR \circ \Sing \circ \mathop{*^{\np{\mathscr{A}}}})$
is a weak equivalence.
Furthermore, for any poset $P$ and any functor $F\colon P \to \sSet$,
there is a natural isomorphism
$\AbBarCon(\freeR \circ F) \iso \freeR(\AbBarCon(F))$,
using the definition of the tensor structure on $\sRMod$
and the fact that $\freeR$ preserves colimits (as it is a left adjoint).
So we have a commutative diagram:
\[\begin{tikzcd}
   	\AbBarCon(\freeR \circ \Sing \circ \mathop{\mathscr{D}_{\mathscr{A}}})\arrow{r}{\iso}\arrow{d}[swap]{\simeq} &\freeR(\AbBarCon(\Sing \circ \mathop{\mathscr{D}_{\mathscr{A}}}))\arrow{d}\\
	\AbBarCon(\freeR \circ \Sing \circ \mathop{*^{\np{\mathscr{A}}}})\arrow{r}{\iso} &\freeR(\AbBarCon(\Sing\circ \mathop{*^{\np{\mathscr{A}}}} ))
  \end{tikzcd}
\]
and hence, the morphism on the right is a weak equivalence by 2-of-3.

For any simplicial set $K$, the unit map $K \to \Sing(|K|)$ is a natural weak equivalence,
as explained in \cref{subsection:bar_simp_model}.
The functor $\freeR$ preserves all weak equivalences,
as every simplicial set is cofibrant \cite[Lemma 1.1.12 (Ken Brown’s lemma)]{hovey}.
So we have a commutative square,
in which the indicated maps are weak equivalences:
\[
\begin{tikzcd}
& \freeR(\AbBarCon(\Sing \circ \mathop{\mathscr{D}_{\mathscr{A}}})) \arrow[d, "\simeq"'] \arrow[r, "\simeq"]
& \freeR(\Sing(\vert \AbBarCon(\Sing \circ \mathop{\mathscr{D}_{\mathscr{A}}})\vert)) \arrow[d] \\
& \freeR(\AbBarCon(\Sing\circ \mathop{*^{\np{\mathscr{A}}}})) \arrow[r, "\simeq"']
& \freeR(\Sing(\vert \AbBarCon(\Sing\circ \mathop{*^{\np{\mathscr{A}}}})\vert))
\end{tikzcd}
\]
It follows, by 2-of-3, that the fourth map in the square is a weak equivalence as well, and so
$\vert \AbBarCon(\Sing \circ \mathop{\mathscr{D}_{\mathscr{A}}}) \vert
\to \vert \AbBarCon(\Sing\circ \mathop{*^{\np{\mathscr{A}}}}) \vert$ is an $R$-homology isomorphism.

For any poset $P$ and any functor $F\colon P \to \sSet$,
there is a natural isomorphism
$\AbBarCon(\mbox{$|-|$} \circ F) \iso |\AbBarCon(F)|$,
again using the definitions and the fact that geometric realization preserves colimits,
being a left adjoint.
So we have the following commutative diagram:
\[
\begin{tikzcd}
 	\vert \AbBarCon(\Sing \circ \mathop{\mathscr{D}_{\mathscr{A}}}) \vert
	\arrow{d}{\cong}\arrow{r}{\sim_R}  &
	 \vert \AbBarCon(\Sing\circ \mathop{*^{\np{\mathscr{A}}}}) \vert\arrow{d}{\cong}\\
	 \AbBarCon(|-|\circ \Sing \circ \mathop{\mathscr{D}_{\mathscr{A}}})\arrow{r}\arrow{dr} & \AbBarCon(|-|\circ \Sing \circ \mathop{*^{\np{\mathscr{A}}}})\\
	 & \AbBarCon(\mathscr{D}_{\mathscr{A}})\arrow{u}
\end{tikzcd}
\]
Recall that the map $\AbBarCon(|-| \circ \Sing \circ \mathop{\mathscr{D}_{\mathscr{A}}})
\to \AbBarCon(\mathscr{D}_{\mathscr{A}})$ from line \ref{eq:unif_nerve_eq} is a weak equivalence.
Together with the fact that weak homotopy equivalences are also $R$-homology isomorphisms \cite[Proposition 4.21]{hatcher}, we get that the canonical map
\[
	\pi_{\Sd N} \colon \AbBarCon(\mathscr{D}_{\mathscr{A}})
	\to \vert  \Sd \Nrv(\mathscr{A}) \vert \iso  \AbBarCon(*^{\np{\mathscr{A}}})  \iso \AbBarCon(|-|\circ \Sing \circ \mathop{*^{\np{\mathscr{A}}}})
\]
is, once more by 2-of-3, an $R$-homology isomorphism as well.
\end{proof}

\printbibliography

\appendix
\section{Auxiliary Lemmas about Geometric Simplicial Complexes}
\label{appendix-auxiliary-lemmas}

\begin{lemma}
\label{lemma:int_bar_star}
Let $\sigma = \{ v_0, \dots, v_k\} \in K$ be a simplex and consider the subcomplex $L = \{ \tau_0 \subset \dots \subset \tau_m \mid  \sigma \subseteq \tau_0\}\subseteq \op{Sd}K.$
Then $\bigcap_{i=0}^k \operatorname{bst}v_i = |L|.$
\end{lemma}

\begin{proof}[Proof of \cref{lemma:int_bar_star}]
First, let $\phi =(\tau_0\subset  \cdots \subset \tau_m) \in L$ be a simplex. By definition, $\phi$ is contained in the simplex $\sigma\subseteq \tau_0\subset \cdots\subset \tau_m$ of $\Sd K$. Thus, the realization of $\phi$ is contained in $\op{bst} v_i$ for all $i$, and so we have $|L| \subseteq \bigcap_{i=0}^k \operatorname{bst}v_i$.

Now, let $|\phi =(\tau_0\subset  \cdots \subset \tau_m) | \subseteq \bigcap_{i=0}^k \op{bst} v_i$. Since for all $i$ we have $|\tau_0\subset  \cdots \subset \tau_m| \subseteq \operatorname{bst}v_i$,
we know that $v_i\in \tau_0$. Thus, the simplex $\sigma$ is also contained in $\tau_0$. Therefore, $\phi \in L$ and so we have $\bigcap_{i=0}^k \operatorname{bst}v_i \subseteq |L|$.
\end{proof}

\begin{proof}[Proof of \cref{lemma:int_contractible}]
	By \cref{lemma:int_bar_star}, every (geometric) simplex in $\bigcap_{v \in \sigma}\operatorname{bst}v\subseteq |\Sd K|$ has a coface in this intersection with $z(|\sigma|)$ as a vertex, where $z(|\sigma|)$ is the barycenter of $|\sigma|$.
	Thus, $\bigcap_{v \in \sigma}\operatorname{bst}v$ is star-shaped with respect to $z(|\sigma|)$ and hence contractible.
\end{proof}

The following two lemmas are straightforward calculations (compare \cite[p.62]{eil_steen}).
\begin{lemma}
\label{lemma:subdivision_simplex}
	Let $K$ be a simplicial complex and let $x\in |K|$. Write $x$ in barycentric coordinates of $K$ as \[x=\sum\limits_{j=0}^m\nu_j\cdot |w_j| \] with $w_i \in \vertx{K}$, $\nu_i> 0$ and $\sum\limits_{j=0}^m\nu_j=1$ as well as $\nu_0\geq \nu_1\geq \cdots\geq \nu_m$. Then, using the (geometric) simplices \begin{align}
	\label{equation:bary_simplices}
	|\tau_i|=\op{conv}\{ |w_{0}| ,\dots,|w_i| \} \quad \text{for all } i\in\{0,\dots,m\}
	\end{align}
	in the realization $|K|$ and by writing $z(|\tau_i|)$ for the barycenter of $|\tau_i|$, we have $x\in \op{conv}\{ z(|\tau_0|),\dots,z(|\tau_m|)\} .$
	Specifically, writing $x$ in barycentric coordinates of $\Sd K$ as
	$x=\sum_{j=0}^m\mu_jz(|\tau_j|),$ we have
	\begin{align*}
	\mu_i&=(i+1)\Big(\nu_i(x)-\nu_{i+1}(x)\Big) \text{ for }i=0,\dots,m-1\\
	\mu_m&=(m+1)\nu_m(x).
	\end{align*}
\end{lemma}

\begin{lemma}
	\label{lem:subdivision_simplex_converse}
Let $x \in |\Sd K|$, written in barycentric coordinates as $x=\sum_{j=0}^m\mu_{j}z(|\tau_j|)$
	for some
	flag of simplices $\tau_0\subset \cdots\subset \tau_m$ in $K$, where
	\[
	|\tau_i|=\op{conv}\{ |w_{0}| ,\dots,|w_{i}| \} \quad \text{for all } i\in\{1,\dots,m\}
	\]
	and $w_i \in \vertx{K}$.
	Then we have $x \in |\tau_m| = \op{conv}\{ |w_{0}| ,\dots,|w_{m}| \}.$
	Specifically, the barycentric coordinates $\nu_i$ of $x$ in $K$ with respect to $|w_0|,\dots,|w_m|$ take the form
	\begin{equation}
	\label{eq:bar_coord_sum}
		\nu_i = \sum\limits_{j=i}^{m}\frac{1}{j+1}\mu_{j}.
	\end{equation}
\end{lemma}

\begin{proof}[Proof of \cref{lemma:bst_maximal}]
	Let $x\in |K|$ be a point satisfying \cref{equation:bst_maximal}.
	It suffices to show that $x$ is contained in a simplex of $|\Sd K|$ having $|v|$ as a vertex.
	Let $v_0,\dots,v_m$ be the vertices in $K$ with $b_{v_i}(x) > 0$ in descending order of barycentric coordinates. By \cref{equation:bst_maximal} we may choose $v_0=v$.
	Now, by \cref{lemma:subdivision_simplex}, we know that
	the point $x$ is contained in $\op{conv}\{ |v|=z(|\tau_0|),\dots, z(|\tau_m|) \}$ for the simplices $|\tau_i| \subseteq |K|$ specified as in \cref{equation:bary_simplices}. Hence, by definition the point $x$ is contained in $\operatorname{bst} v$.

	Conversely, let $x \in \operatorname{bst} v$ for some vertex $v \in \vertx{K}$.
	Then there exists a simplex $\tau \in \Sd K$ with $v$ as a vertex such that $x\in |\tau|$ and that $\tau$ corresponds to a flag $v = \tau_0 \subset \dots \subset \tau_m$ of simplices in $K$.
	From \cref{lem:subdivision_simplex_converse}, or more specifically \cref{eq:bar_coord_sum}, we deduce that the barycentric coordinate $\nu_0 = b_v(x)$ of $x$ in $K$ with respect to $v$ is maximal.
\end{proof}
\section{Discrete Morse Theory for Infinite Complexes}
\label{discrete_morse}
Discrete Morse theory for finite cell complexes was introduced by Forman \cite{forman-discrete-morse,forman-user-guide} and since then found its way into algorithms, applications and underwent reformulations as well as generalizations. For an extension to infinite cell complexes that is similar to the results presented here, see \cite{batzies-thesis}.

Let $K$ be a simplicial complex and $\sigma\in K$ a simplex. We call a facet $\tau\subseteq \sigma$ a \emph{free facet} if~$\sigma$ is the only proper coface of $\tau$. An \emph{elementary collapse} $K\searrow K\setminus \{\tau,\sigma\}$ is the removal of a pair of simplices, where $\tau$ is a free facet of $\sigma$. A \emph{collapse} $K\searrow L$ onto a subcomplex~$L$ is a sequence of elementary collapses starting in $K$ and ending in $L$. Moreover, an elementary collapse can be realized continuously by a strong deformation retract and therefore collapses preserve the homotopy type.

By a \emph{discrete vector field} on $K$ we mean a partition $V$ of $K$ into singletons $\{\sigma\}$, $\sigma$ is then called a \emph{critical simplex}, and pairs $\{\sigma,\tau\}$ corresponding to arcs $(\sigma, \tau)$ in the Hasse diagram $\mathcal{H}(K)$ of the face poset, i.e., the directed graph whose nodes are the simplices and whose arcs are the pairs $(\sigma, \tau)$ in which $\sigma$ is a facet of $\tau$. We call the discrete vector field $V$ a \emph{discrete gradient vector field} if the graph $\mathcal{H}(K,V)$ that is obtained from $\mathcal{H}(K)$ by reversing all the arcs $(\sigma,\tau)$ for which $\{\sigma,\tau\}\in V$ is acyclic. Note that it suffices to check that there are no non-trivial closed $V$-paths, i.e., that there are no undirected paths in $\mathcal{H}(K)$ that are of the form $\tau_0\to\mu_0\leftarrow\cdots\to\mu_r\leftarrow\tau_{r+1}$ with $\{\tau_i,\mu_i \}\in V$, $\tau_i\neq\tau_{i+1}$ and $\tau_0=\tau_{r+1}$.

Given a discrete gradient vector field $V$ on a simplicial complex $K$, we can define a poset structure on $V$ as follows: For two elements $A,B\in V$ be define $A\leq_V B$ if and only if there exists a sequence $A=C_0,C_1,\dots,C_n=B$ in $V$ such that for every $i$ there exist elements $x_{i-1}\in C_{i-1},x_{i}\in C_{i}$ with $x_{i-1}$ a face of $x_i$.

Moreover, for any element $A\in V$ we define its \emph{height} to be
\[\op{ht}(A)= \sup\{n\in \mathbb{N}\mid \exists\ A=B_0> \cdots > B_n \text{ in }V   \}. \]

The following lemma is useful in practice.
\begin{lemma}
\label{lemma:equivalent_height_def}
The height $\op{ht}(A)$ is finite for every $A\in V$ if and only if for every simplex $\sigma\in K$ its $V$-path height
\[\op{ht}_V(\sigma)=\sup\{n\in\mathbb{N}\mid \exists \text{ V-path }\sigma=\tau_0\to\mu_0\leftarrow\cdots\to\mu_{n-1}\leftarrow\tau_{n}\}\]
is finite.
\end{lemma}
\begin{proof}
Every $V$-path $\tau_0\to\mu_0\leftarrow\cdots\to\mu_{n-1}\leftarrow\tau_{n}$ induces a descending chain $\{\tau_0,\mu_0\}>\cdots>\{\tau_{n-1},\mu_{n-1}\}$. Hence, if the height is finite for every $A\in V$ this implies that the $V$-path height is finite for every simplex in $K$.

For the converse, we employ induction over the dimension $\dim A=\dim \min A$ of an element $A\in V$. If $\dim A=0$, then $\op{ht}(A)=\op{ht}_V(\min A)$ and this is finite by assumption. For the induction step,
consider the set $F_{\min A}$ of all $V$-paths starting in $\min A$. For a gradient path $\gamma= \tau_0\to\mu_0\leftarrow\cdots\to\mu_{n-1}\leftarrow\tau_{n}$ write $\op{end}\gamma=\tau_n$ and $\op{length}\gamma=n$. Then we can bound $\op{ht}(A)$ from above as follows:
\[\op{ht}(A)\leq \max\{\op{length}\gamma +1+ \op{ht}(B) \mid \gamma\in F_{\min A},\ \sigma\subsetneq\op{end}\gamma,\ \sigma\in B\in V\}. \]
To complete the induction step, note that for every $B$ as above we have $\dim B<\dim A$ and hence $\op{ht}(B)$ is finite by the induction assumption. Thus, it suffices to show that the set $F_{\min A}$ is finite. This can be seen as follows: Given a gradient path $\gamma$ as above, then the path ends in $\tau_i$ or there are $\dim\mu_i=\dim\tau_0+1$ choices for $\tau_{i+1}$ once $\tau_i$ is fixed. Hence, the cardinality of $F_{\min A}$ is bounded from above by $(\dim(\sigma)+2)^{\op{ht}_V(\sigma)}$.
\end{proof}

The essential ideas for the proof of the following proposition can be found already in the proof of \cite[Proposition 1]{brown-morse}, which predates Forman's papers.
\begin{prop}
\label{infinite_discrete_morse}
 Let $L\subseteq K$ be a pair of simplicial complexes and let $V$ be a discrete gradient vector field on $K$ such that for every element $A\in V$ its height $\op{ht}(A)$ is finite. Moreover, assume that $K\setminus L$ is the union of pairs in $V$. Then the inclusion $|L|\hookrightarrow |K|$ is a homotopy equivalence.
\end{prop}

Before we prove \cref{infinite_discrete_morse} we need one small lemma.
\begin{lemma}
\label{infinite_inclusion_homotop_eq}
 Let $K$ be a simplicial complex and $K_{0}\subseteq K_1\subseteq \cdots\subseteq K$ a filtration of subcomplexes whose union is $K$ such that for all $i\in\mathbb{N}$ the inclusion $|K_i|\hookrightarrow |K_{i+1}|$ is a homotopy equivalence. Then the inclusion $|K_0|\hookrightarrow |K|$ is also a homotopy equivalence.
\end{lemma}
\begin{proof}
By Whitehead's theorem \cite[Theorem 4.5]{hatcher}, it suffices to show that for all $n\in\mathbb{N}$ the induced morphism on homotopy groups $g_n\colon \pi_n(|K_0|)\to\pi_n(|K|)$ is an isomorphism. For any map $f\colon S^n\to |K|$ its image is compact and hence contained in some $|K_i|$. As $|K_0|\to |K_i|$ is a homotopy equivalence, it follows that the homotopy class $[f]$ is in the image of the composite $\pi_n(|K_0|)\to \pi_n(|K_i|)\to \pi_n(|K|)$ and hence also in the image of $g_n$. This shows surjectivity. A similar argument applied to any homotopy $h\colon S^n\times[0,1]\to |K|$ shows that $g_n$ is injective. This proves the lemma.
\end{proof}

\begin{proof}[Proof of \cref{infinite_discrete_morse}.]
Without loss of generality, we can assume that $L$ is the union of critical simplices.

Consider the filtration $L=K_{0}\subseteq K_1\subseteq \cdots\subseteq K$ of $K$, where $K_i$ is the subcomplex \[K_i=L\cup \bigcup\limits_{A\in V,\ \op{ht}(A)\leq i}A.\]
We show that for every $i\in \mathbb{N}$ the inclusion $|K_i|\hookrightarrow |K_{i+1}|$ is a homotopy equivalence and the proposition then follows from \cref{infinite_inclusion_homotop_eq}: Let $\{\tau,\sigma\}=A\in V$ be any element, with $\tau$ a facet of $\sigma$, such that $\op{ht}(A)=i+1$. Then, $\tau$ is a free facet of $\sigma$ in $K_{i+1}$, as otherwise there would exist a pair $B\in V$ with $\op{ht}(B)\leq i+1$ and $B>A$. But this cannot be true, because then the last property implies that the height of $B$ satisfies $\op{ht}(B)\geq \op{ht}(A)+1=i+2$, contradicting the construction. A similar argument shows that $\sigma$ is not properly contained in any simplex of $K_{i+1}$.
Therefore, the complement $K_{i+1}\setminus K_i$ is partitioned by pairs in $V$ of height $i+1$ and the corresponding simplices to different pairs can only possibly intersect in the subcomplex $K_i$. Thus, executing the elementary collapses that are encoded by those pairs simultaneously induces a strong deformation retract \[|K_{i+1}|\to \left|K_{i+1}\setminus \bigcup\limits_{
A\in V,\ \op{ht}(A)=i+1}A\right|=|K_{i}| \]
and hence the inclusion $|K_i|\hookrightarrow |K_{i+1}|$ is a homotopy equivalence.
\end{proof}

\section{The Blowup Complex for Open Covers}
\label{appendix-open-covers}

The parts of 1(a) and 2(a) in \cref{abstract-nerve-thm} that establish weak homotopy equivalences follow from work of Dugger and Isaksen \cite{dugger-isaksen}. In this section, we adapt their proof strategy to give a more direct proof of the fact that the natural map $\rho_{S} \colon \SmallBarCon(\mathscr{A}) \to X$ is a weak homotopy equivalence whenever $\mathscr{A}$ is an open cover of~$X$.
Moreover, we use the same ideas to give a short proof of the fact that $\rho_{N} \colon \SmallBarCon(\mathscr{A}) \to | \Nrv(\mathscr{A}) |$ is a weak homotopy equivalence whenever $\mathscr{A}$ is a weakly good cover.

By the following lemma, a map is a weak homotopy equivalence if it is so locally.
\begin{lemma}[{\cite[Lemma 16.24]{MR0402714}; \cite[Theorem 6.7.11]{MR2456045}}]
\label{lemma:union-weak-equ}
	Let $f\colon Y\to X$ be a continuous map and let $\mathscr{A} = (A_i)_{i \in I}$ be an open cover of $X$. If for every $\sigma\in \Nrv(\mathscr{A})$ the restricted map $f^{-1}(A_\sigma)\to A_\sigma$ is a weak homotopy equivalence, then so is $f$.
\end{lemma}
In order to apply this lemma to $\rho_S$, we need to determine the preimages of the finite intersections of cover elements in $\mathscr{A} = (A_i)_{i \in I}$. Recall from \cref{defi:blowup} that the blowup complex is defined as
\[
	\SmallBarCon(\mathscr{A}) = \left( \bigsqcup_{J \in \Nrv(\mathscr{A})}
	A_J\times |J|  \right) \; / \; \sim
\]
and $\rho_S$ is induced by the projections of the products $A_J\times |J| $ onto the first coordinate. For any $\sigma\in \Nrv(\mathscr{A})$ the preimage is the subspace
\[
	\rho_S^{-1}(A_\sigma)=\left( \bigsqcup_{J \in \operatorname{St}_{\Nrv(\mathscr{A})}(\sigma)}
	 A_J\times |J|  \right) \; / \; \sim\ ,
\]
where $\operatorname{St}_{\Nrv(\mathscr{A})}(\sigma)=\{J\in \Nrv(\mathscr{A})\mid \sigma\subseteq J \}$ is the star of $\sigma$ in $\Nrv(\mathscr{A})$.
This can be seen as follows: Whenever $A_J \cap A_{\sigma} =A_{J\cup \sigma}$ is non-empty for some $J$, the union $J\cup \sigma$ is a simplex in~$\operatorname{St}_{\Nrv(\mathscr{A})}(\sigma)$ and so
\[(A_J \cap A_{\sigma})\times |J| \subseteq A_{J\cup \sigma}\times |J\cup \sigma|\]
is contained in the right hand side of the equality above.
Conversely, $A_J\subseteq A_\sigma$ for every~$J\in \operatorname{St}_{\Nrv(\mathscr{A})}(\sigma)$, and so the above equality holds.

\begin{prop}
\label{prop:global-weak-heq-blowup-space}
 Let $\mathscr{A}=(A_i)_{i\in I}$ be an open cover of the topological space $X$. Then the natural map $\rho_S\colon \SmallBarCon(\mathscr{A}) \to X$ is a weak homotopy equivalence.
\end{prop}
\begin{proof}
By \cref{lemma:union-weak-equ} it suffices to prove that for every $\sigma\in \Nrv(\mathscr{A})$ the restricted map $\rho_S^{-1}(A_\sigma)\to A_\sigma$ is a weak homotopy equivalence. We show that this map is in fact a homotopy equivalence.

Choose any point $z\in |\sigma|$ and consider the following subspace
 \[A_\sigma\times \{z\} \hookrightarrow  A_\sigma\times |\sigma| \hookrightarrow \rho_S^{-1}(A_\sigma).\]
Note that the space $|\operatorname{St}_{\Nrv(\mathscr{A})}(\sigma)|$ is star-shaped with respect to $z$ and hence it deformation retracts onto this point. As $A_\tau\subseteq A_\sigma$ for every $\tau\in \operatorname{St}_{\Nrv(\mathscr{A})}(\sigma)$, this lifts to a deformation retraction of $\rho_S^{-1}(A_\sigma)$ onto $ A_\sigma\times \{z\} $. Therefore, we get the following commutative diagram, where the horizontal maps are homotopy equivalences:
\[\begin{tikzcd}
   \arrow[hookrightarrow]{r}{\simeq} A_\sigma\times \{z\} \arrow{d}[swap]{\pi_1} & \rho_S^{-1}(A_\sigma) \arrow{d}{\rho_S} \\
   \arrow[hookrightarrow]{r}{=}A_\sigma & A_\sigma
  \end{tikzcd}
\]
Obviously, $\pi_1$ is a homotopy equivalence and hence so is the map on the right. This proves that $\rho_S\colon \SmallBarCon(\mathscr{A}) \to X$ is weak homotopy equivalence.
\end{proof}

Recall now that $\rho_N$ is induced by the projections of the products $A_J\times |J|$ onto the second coordinate.
To apply \cref{lemma:union-weak-equ} to $\rho_N$,
we cover~$| \Nrv(\mathscr{A}) |$ by the open simplex stars~$(S_\sigma)_{\sigma\in\Nrv(\mathscr{A})}$, where
\[
	S_\sigma=\bigcup\{\op{int}|J|\mid J\in \operatorname{St}_{\Nrv(\mathscr{A})}(\sigma) \}.
\]
Note that this cover is closed under finite intersections. Hence, it suffices to consider for any~$\sigma\in \Nrv(\mathscr{A})$ the preimage
\[
	\rho_N^{-1}(S_\sigma)=\left( \bigsqcup_{J \in \operatorname{St}_{\Nrv(\mathscr{A})}(\sigma)}
	 A_J\times \operatorname{int}|J|  \right) \; / \; \sim .
\]
\begin{prop}
\label{prop:global-weak-heq-blowup-nerve}
 Let $\mathscr{A}=(A_i)_{i\in I}$ be a weakly good cover of the topological space $X$. Then the natural map $\rho_{N} \colon \SmallBarCon(\mathscr{A}) \to | \Nrv(\mathscr{A}) |$ is a weak homotopy equivalence.
\end{prop}

\begin{proof}
By \cref{lemma:union-weak-equ} it suffices to prove that for every $\sigma\in \Nrv(\mathscr{A})$ the restricted map $\rho_N^{-1}(S_\sigma)\to S_\sigma$ is a weak homotopy equivalence.

Similarly to the the proof of \cref{prop:global-weak-heq-blowup-space},
we get the following commutative diagram, where $z\in \op{int}|\sigma|$ is any point and the horizontal maps are homotopy equivalences:
\[\begin{tikzcd}
   \arrow[hookrightarrow]{r}{\simeq} A_\sigma \times \{z\} \arrow{d}[swap]{\pi_2} & \rho_N^{-1}(S_\sigma) \arrow{d}{\rho_N} \\
   \arrow[hookrightarrow]{r}{\simeq}\{z\} & S_\sigma
  \end{tikzcd}
\]
As $\mathscr{A}$ is a weakly good cover, the map $\pi_2$ is a weak homotopy equivalence and hence so is the map on the right. This proves that $\rho_{N} \colon \SmallBarCon(\mathscr{A}) \to | \Nrv(\mathscr{A}) |$ is a weak homotopy equivalence.
\end{proof}

\paragraph{Acknowledgements}
This research has been supported by
the Deutsche Forschungsgemeinschaft (DFG -- German Research Foundation) -- Project-ID 195170736 -- TRR109,
and Austrian Science Fund (FWF) grants P 29984-N35 and P 33765-N.

\end{document}